%% file: Conjugacy_classes_big_mcg.tex
\documentclass[12pt]{amsart}

\usepackage[paperwidth=21cm,paperheight=29.7cm,left=2cm,top=2cm]{geometry}

\usepackage{mathscinet,amssymb,latexsym,stmaryrd,phonetic}
\usepackage{enumerate,mathrsfs,hyperref,cmll,color}
\usepackage{cite}
\usepackage{amsthm}
\usepackage{amsmath}

\usepackage{tikz}
\usetikzlibrary{matrix, arrows, patterns}

\newtheorem{theorem}{Theorem}[section]
\newtheorem{proposition}[theorem]{Proposition}
\newtheorem{lemma}[theorem]{Lemma}

\newtheorem{question}[theorem]{Question}
\newtheorem{corollary}[theorem]{Corollary}

\newtheorem{example}[theorem]{Example}
\newtheorem{definition}[theorem]{Definition}

\theoremstyle{definition}
\newtheorem{remark}[theorem]{Remark}

\makeatletter
\@namedef{subjclassname@2020}{\textup{2020} Mathematics Subject Classification}
\makeatother

\DeclareMathOperator{\Aut}{Aut}
\DeclareMathOperator{\MCG}{Map}
\DeclareMathOperator{\PMCG}{PMap}

\DeclareMathOperator{\Homeo}{Homeo}
\DeclareMathOperator{\supp}{supp}
\newcommand{\g}{\mathfrak{g}}

\newcommand{\R}{\mathbb{R}}

\newcommand{\Q}{\mathbb{Q}}
\newcommand{\Z}{\mathbb{Z}}
\newcommand{\N}{\mathbb{N}}

\newcommand{\bd}{\begin{displaymath}}
\newcommand{\ed}{\end{displaymath}}
\newcommand*{\pmap}[1]{\mathrm{PMap}(#1)}
\newcommand*{\pmapc}[1]{\mathrm{PMap}_{c}(#1)}

\usepackage[utf8]{inputenc}

\title{Conjugacy classes of big mapping class groups}

\author[Hernández Hernández]{Jesús Hernández Hernández}
\address{Jesús Hernández Hernández \newline
Centro de Ciencias Matemáticas, UNAM Campus Morelia, Antigua Carretera a Pátzcuaro \#8701, Col.~Ex Hacienda San José de la Huerta, C.P.~58089, Morelia, Michoacán, México}
\email{jhdez@matmor.unam.mx}
\author[Hru\v{s}\'ak]{Michael Hru\v{s}\'ak}
\address{Michael Hru\v{s}\'ak \newline
Centro de Ciencias Matemáticas, UNAM Campus Morelia, Antigua Carretera a Pátzcuaro \#8701, Col.~Ex Hacienda San José de la Huerta, C.P.~58089, Morelia, Michoacán, México}
\email{michael@matmor.unam.mx}
\author[Morales]{Israel Morales}
\address{Israel Morales \newline
Instituto de Matemáticas, UNAM Unidad Oaxaca, Antonio de León \#2, altos, Col.~Centro, Oaxaca de Juárez, C.P.~68000, México}
\email{fast.imj@gmail.com}
\author[Randecker]{Anja Randecker}
\address{Anja Randecker \newline
Mathematisches Institut, Universit\"at Heidelberg, Im Neuenheimer Feld 205, 69120 Heidelberg, Germany}
\email{arandecker@mathi.uni-heidelberg.de}
\author[Sedano]{Manuel Sedano}
\address{Manuel Sedano \newline
Centro de Ciencias Matemáticas, UNAM Campus Morelia, Antigua Carretera a Pátzcuaro \#8701, Col.~Ex Hacienda San José de la Huerta, C.P.~58089, Morelia, Michoacán, México}
\email{msedano@matmor.unam.mx}
\author[Valdez]{Ferrán Valdez}
\address{Ferrán Valdez \newline
Centro de Ciencias Matemáticas, UNAM Campus Morelia, Antigua Carretera a Pátzcuaro \#8701, Col.~Ex Hacienda San José de la Huerta, C.P.~58089, Morelia, Michoacán, México}
\email{ferran@matmor.unam.mx}

\subjclass[2020]{57K20; 03E15; 20E45}

\begin{document}

\begin{abstract}
    We describe the topological behavior of the conjugacy action of the mapping class group of an orientable infinite-type surface $\Sigma$ on itself by proving that:
    \\
\begin{enumerate}
\item All conjugacy classes of $\MCG(\Sigma)$ are meager for every $\Sigma$,\\
\item $\MCG(\Sigma)$ has a somewhere dense conjugacy class if and only if $\Sigma$ has at most two maximal ends and no non-displaceable finite-type subsurfaces,\\
\item $\MCG(\Sigma)$ has a dense conjugacy class if and only if $\Sigma$ has a unique maximal end and no non-displaceable finite-type subsurfaces.\\
\end{enumerate}
Our techniques are based on model-theoretic methods developed by Kechris, Rosendal and Truss in~\cite{kechris-rosendal,truss}.
\end{abstract}

\maketitle

\input{introduction.tex}
\input{background.tex}
\input{meager.tex}

\input{dense.tex}
\input{nowhere-dense.tex}

\input{questions.tex}

\bibliography{biblio}{}
\bibliographystyle{plain}

\end{document}

%% file: introduction.tex
Let $\Sigma$ be an infinite-type\,\footnote{A surface is of finite type if its fundamental group is finitely generated and of infinite type if not.} surface and $\MCG^*(\Sigma)$ the group of homeomorphisms of $\Sigma$ modulo isotopy. This group is called the extended (big\,\footnote{This nomenclature is inspired by the fact that these groups are uncountable, in contrast to the countable $\MCG(S)$ when $S$ is a finite-type surface.}) mapping class group of $\Sigma$. Each $\MCG^*(\Sigma)$ acts on the curve graph $C(\Sigma)$, which is the countable graph whose vertices are (isotopy classes of essential) simple closed curves in $\Sigma$  and adjacency is defined by disjointness. Moreover, $\MCG^*(\Sigma)$ is isomorphic to $\Aut(C(\Sigma))$~\cite{Bavard-Dowdall-Rafi,hernandez-etal}, and thus $\MCG^*(\Sigma)$ is a Polish group with respect to the permutation topology. Let $\MCG(\Sigma)$ denote the index two subgroup of $\MCG^*(\Sigma)$ formed by orientation-preserving mapping classes. This group is called the (big) mapping class group of $\Sigma$.

The main contribution of this article is the topological description of the conjugacy action of $\MCG(\Sigma)$. First, we show that conjugacy classes of $\MCG(\Sigma)$ are, from a topological point of view, negligible, i.e.~\emph{meager} which means that a subset of a space is the countable union of nowhere dense subsets.

\begin{theorem}\label{THM:All_CC_are_meager}
 Let $\Sigma$ be an infinite-type surface. Then all conjugacy classes of $\MCG(\Sigma)$ are meager.
\end{theorem}

We actually prove a more general statement (see Theorem~\ref{THM:All_CC_are_meager_full_statement}): The statement above remains valid if one exchanges $\MCG(\Sigma)$ by any closed subgroup $\Gamma$ of $\MCG(\Sigma)$ which contains, for every essential simple closed curve $\alpha\subset\Sigma$, a non-trivial power of a Dehn twist  w.r.t.~$\alpha$. For a more detailed discussion about Dehn twists, see Section~\ref{SSEC:Dehn-Twists}.

Since comeager implies dense but not the other way around, we investigate which mapping class groups have dense conjugacy classes. Our main result in this vein is a complete characterization of mapping class groups having dense conjugacy classes. Two topological properties of the surface $\Sigma$ are involved in this characterization. First, the Mann-Rafi (partial) order on the space of ends of the surface. Roughly speaking, an end $x$ is smaller than~$y$ in this order if every neighbourhood of $y$ contains a copy of a neighbourhood of $x$. The second property is the existence of a (finite-type) \emph{non-displaceable} subsurface, \emph{i.e.}\ $S\subset\Sigma$ such that for every $f\in\Homeo(\Sigma)$, we have that $S\cap f(S)\neq\emptyset$.

\begin{theorem}[Characterization of dense conjugacy classes]\label{THM:Charact-Dense-Conjugacy-Classes-simple-version}
Let $\Sigma$ be an infinite-type surface. Then $\MCG(\Sigma)$ has a dense conjugacy class if and only if $\Sigma$ has no non-displaceable finite-type subsurfaces and a unique maximal end.
\end{theorem}

As a matter of fact, the characterization given above is formulated in more technical terms. For a precise and complete statement, see Theorem~\ref{THM:Charact-Dense-Conjugacy-Classes}.

While attending the Big Surf(aces) Seminar, we learned that J.~Lanier and N.~Vlamis have also proved Theorems~\ref{THM:All_CC_are_meager} and~\ref{THM:Charact-Dense-Conjugacy-Classes-simple-version}, see~\cite{Lanier-Vlamis}.

\begin{remark}\label{Remark:countable-one-max-end}
When $\Sigma$ is a genus zero surface whose space of ends is countable, then $\Sigma$ has no (finite-type) non-displaceable subsurfaces and a unique maximal end if and only if its space of ends is homeomorphic to $\omega^\alpha+1$, where $\alpha$ is a countable ordinal. These are precisely the surfaces with countable space of ends for which $\MCG(\Sigma)$ is quasi-isometric (w.r.t.\ the word metric) to a point, compare with Theorem 1.5 in~\cite{mann-rafi}.
\end{remark}

Let $\PMCG(\Sigma)$denote the \emph{pure mapping class group}, that is, the subgroup formed by all orientation-preserving mapping classes which fix every end of the surfaces. Furthermore, let $\PMCG_c(\Sigma)$ denote the group generated by compactly supported mapping classes, respectively. The following result tells us that, except for the case when $\Sigma$ is an infinite-genus surface with one end (\emph{a.k.a.}\ the Loch Ness monster), closed subgroups $\Gamma$ of $\PMCG(\Sigma)$ and $\PMCG_c(\Sigma)$ with lots of Dehn twists do not have dense conjugacy classes, that is, for every $g\in \Gamma$, the conjugacy class $g^\Gamma:= \{hgh^{-1}:h\in\Gamma\}$ is not dense in $\Gamma$.

\begin{theorem}
	\label{THM:CC_for_PMap}
 Let $\Sigma$ be an infinite-type surface that is not homeomorphic to the Loch Ness Monster, and $\Gamma$ be a closed subgroup of $\PMCG(\Sigma)$ that contains, for every closed essential curve $\alpha\subset\Sigma$, a non-trivial power of a Dehn twist w.r.t.\ $\alpha$. Then $\Gamma$ does not contain dense conjugacy classes.
\end{theorem}

The following is a corollary to Theorem~\ref{THM:Charact-Dense-Conjugacy-Classes-simple-version} because, when $\Sigma$ is homeomorphic to the Loch Ness Monster, $\PMCG_c(\Sigma)$ is dense in $ \MCG(\Sigma)$ (see Theorem 4 in~\cite{Patel-Vlamis2018}).

\begin{corollary}
 Let $\Sigma$ be the Loch Ness Monster, and $\Gamma$ be any closed subgroup with\linebreak $\PMCG_c(\Sigma) \leq \Gamma \leq \MCG(\Sigma)$. Then $\Gamma$ has a dense conjugacy class.
\end{corollary}

Furthermore, we show that the presence of a non-displaceable finite-type subsurface rules out any possible density for conjugacy classes:

\begin{theorem}
    \label{THM:non-displaceable-the-nowhere-dense-versionintro}
 If $\Sigma$ has a non-displaceable finite-type subsurface $S$ and $W \subset \MCG(\Sigma)$ is a non-empty open set, then there exist $V_{1}, V_{2} \subset W$ non-empty and disjoint open sets satisfying that any conjugacy class with non-empty intersection with $V_{i}$, is disjoint from $V_{j}$, for $i \neq j$. In particular, all conjugacy classes of elements are nowhere dense in $\MCG(\Sigma)$.
\end{theorem}

As before, we actually prove a more general statement for certain closed subgroups of $\MCG(\Sigma)$, see Theorem~\ref{THM:non-displaceable-the-nowhere-dense}.

Finally, we characterize big mapping class groups having a conjugacy class that is somewhere dense.

\begin{theorem}
	\label{thm:Somewhere-dense}
	Let $\Sigma$ be an infinite-type surface. Then $\MCG(\Sigma)$ has a conjugacy class which is somewhere dense
if and only if $\Sigma$ has at most two maximal ends and no non-displaceable subsurfaces of finite type.
\end{theorem}

\begin{remark}
Infinite-type surfaces with finite-type non-displaceable subsurfaces whose space of ends have exactly two maximal ends actually exist, see Figure~\ref{Fig:two-ends-nondisplaceable}. When $\Sigma$ is a genus zero surface with countable space of ends, then $\Sigma$ has two maximal ends if and only if its space of ends is homeomorphic to $\omega^\alpha\cdot 2+1$, $\alpha$ a countable ordinal. When $\alpha$ is a successor ordinal (\emph{e.g.}\ a natural number) then $\MCG(\Sigma)$ is not quasi-isometric to a point (w.r.t.\ the word metric). On the other hand, if $\alpha$ is a successor ordinal, it is not clear how to define the quasi-isometry type of $\MCG(\Sigma)$ as this group is not generated by a CB set. See Theorems~1.2 and~1.5 in \cite{mann-rafi} for details.
\end{remark}

\textbf{Technology}. The methods that we use to achieve the aforementioned results lie at the border between geometry and model theory. More precisely, our proofs rely on the work of A.~Kechris and C.~Rosendal~\cite{kechris-rosendal} on automorphism groups of countable structures. These authors provide a characterization of comeager and dense conjugacy classes using the formalism of Fra\"{i}ss\'e theory. In order to apply this formalism we need to do two things. The first is to find, for any closed subgroup $\Gamma$ of a big mapping class group, an ultrahomogeneous countable structure whose automorphism group is isomorphic to $\Gamma$. This is achieved thanks to a standard technique that we call Fra\"{i}ss\'efication. Roughly speaking, since $\Gamma$ acts on~$C(\Sigma)$, one can enrich the structure of this graph by considering only finite isomorphisms induced by the restriction of elements of $\Gamma$. Second, we need to translate the properties that characterize comeager and dense conjugacy classes from the language of Fra\"{i}ss\'e's theory to the language of surfaces and homeomorphisms. To our knowledge, this has not been done before in the literature. For examples of such translations, see Sections~\ref{SSEC:Translation}
and~\ref{SSEC:Needed-Lemmas}.

\textbf{Ample generics and automatic continuity}. A Polish group $G$ has the \emph{automatic continuity property} if every homomorphism from $G$ to a separable group $H$ is continuous. In other words, any map preserving the algebraic structure must be continuous.
In~\cite{kechris-rosendal}, Kechris and Rosendal show that automatic continuity follows from a dynamical property called \emph{ample generics}: a Polish group $G$ has \emph{ample generic} elements (or simply \emph{ample generics}) if for every~$n\in\mathbb{N}$, there is an orbit for the diagonal conjugacy action of $G$ on $G^n$ whose complement is a countable union of nowhere dense subsets (\emph{i.e.}\ the orbit is comeager). Theorem~\ref{THM:All_CC_are_meager} implies that no big mapping class group has ample generics. Luckily, there are other ways to show that a group has the automatic continuity property. For example, K.~Mann uses the \emph{Steinhaus} property to show automatic continuity in the noncompact case~\cite{Mann_note}. The characterization of Steinhaus big mapping class groups is still open.

\textbf{Structure of this article}. Section~\ref{Section:Background} recalls the background from infinite-type surfaces, spaces of ends and Fra\"{i}ss\'e's theory. In particular, Theorem~\ref{Theorem:Criteria} summarizes the criteria we extract from~\cite{kechris-rosendal} to obtain our results. The process of Fra\"{i}ss\'efication applied to the curve graph is detailed in Section~\ref{SSEC:Fraissefication}. In Section~\ref{section:meager}, we begin by translating a property of Fra\"{i}ss\'e classes called the \textbf{local Weak Amalgamation Property} to the language of surfaces and homeomorphisms and use this to prove Theorem~\ref{THM:All_CC_are_meager}. Section~\ref{SEC:dense-conjugacy-classes} deals with the characterization of big mapping class groups having a dense conjugacy class. In the language of Fra\"{i}ss\'e classes, the existence of dense conjugacy classes is characterized by the  \textbf{Joint Embedding Property}. After translating this property to the language of surfaces and homeomorphisms, we dicuss a simple geometric trick to disprove it (we call this the Dehn twist trick) and we present a simple illustrative instance of Theorem~\ref{THM:Charact-Dense-Conjugacy-Classes-simple-version}. This particular case condenses all the arguments needed to understand the general case. In Section~\ref{SEC:Nowhere-Dense-Conjugacy-Classes}, we address the proofs of (a general version of)
Theorem~\ref{THM:non-displaceable-the-nowhere-dense-versionintro}, Theorem~\ref{thm:Somewhere-dense} and Theorem~\ref{THM:CC_for_PMap}.
In Section~\ref{SEC:questions}, we show that two properties that are a test for ample generics fail for all big mapping class groups. These properties are Roelcke precompactness and oligomorphy. Moreover, we present a list of open questions.

\textbf{Acknowledgements}. We want to thank C.~Rosendal for fruitful discussions that guided us to start this article and the anonymous referee for all corrections and remarks. Special thanks to J.~Lanier and N.~Vlamis for sharing with us early versions of their manuscript. J. Hern\'andez Hern\'andez was supported during the creation of this article by the UNAM-PAPIIT research grants IA104620 and IN102018, and by the CONACYT research grant Ciencia de Frontera 2019 CF 217392. The research of M.~Hru\v{s}\'ak was partially supported by PAPIIT grant IN104220 and CONACyT grant A1-S-16164.  I.~Morales and F.~Valdez were partially supported by PAPIIT grant IN102018. I.~Morales was partially supported by Proyecto CONACyT Ciencia de Frontera 217392 cerrando brechas y extendiendo puentes en Geometr\'ia y Topolog\'ia. A.~Randecker acknowledges support of the Deutsche Forschungsgemeinschaft within the Priority Program SPP 2026 ``Geometry at Infinity'' and the Heidelberg STRUCTURES Excellence Cluster. M.~Sedano's research was supported by DGAPA's Posdoctoral grant. F.~Valdez was also supported by CONACYT grant Ciencia B\'asica CB-2016-01 283960.

%% file: background.tex
\section{Background}
    \label{Section:Background}

\subsection{Infinite-type surfaces}\label{subsection:ends} In this article, by \emph{surface} we mean an orientable manifold of real dimension $2$, that is, a second countable, Hausdorff topological space which is locally homeomorphic to the plane. A connected surface $\Sigma$ is said to be of \emph{finite type} if its fundamental group is finitely generated. In all other cases, we say that $\Sigma$ is of \emph{infinite type}.  Finite-type surfaces with empty boundary are classified, up to homeomorphisms, by their genus and number of punctures. Infinite-type surfaces with empty boundary are classified, up to homeomorphisms, by their genus (which now can be infinite) and a pair of nested topological spaces $E_\infty(\Sigma)\subset E(\Sigma)$. The space $E(\Sigma)$ is called the \emph{space of ends of $\Sigma$} and, roughly speaking, encodes the different ways in which one can escape to infinity within $\Sigma$. This space can be defined for a large class of topological spaces, see for example the work of F.~Raymond~\cite{Raymond60}. For surfaces, a further distinction can be made among ends, namely those which are accumulated by genus: this is where $E_\infty(\Sigma)$ enters the picture. A detailed discussion about spaces of ends of surfaces can be found in the work of I.~Richards~\cite{Richards63}. In the rest of this section, we recall the definition and generalities of $E_\infty(\Sigma)\subset E(\Sigma)$ because we use some aspects of them in the proofs of our main results.

A \textit{pre-end} of a surface $\Sigma$ is a nested sequence $U_1\supset U_2\supset \cdots$ of connected open subsets of~$\Sigma$ such that the boundary of $U_n$ in $\Sigma$ is compact for every $n\in\mathbb{N}$ and for any compact subset $K$ of $\Sigma$ there exists $l\in\mathbb{N}$ such that $U_{l}\cap K=\emptyset$. We shall denote the pre-end $U_1\supset U_2\supset \cdots$ as $(U_n)_{n\in\mathbb{N}}$. Two such sequences $(U_{n})_{n\in\mathbb{N}}$ and $(U_{n}^{'})_{n\in \mathbb{N}}$ are said to be equivalent if for any~$i \in \mathbb{N}$, there exists $j \in \mathbb{N}$ such that $U_{j}'\subset U_i$ and vice versa. We denote by $E(\Sigma)$ the corresponding set of equivalence classes and call each equivalence class $[U_{n}]_{n\in\mathbb{N}}\in E(\Sigma)$ an \textit{end} of $\Sigma$. We endow $E(\Sigma)$ with a topology by specifying a pre-basis as follows: for any open subset $W\subset \Sigma$ whose boundary is compact, we define $W^{\ast}:=\{[U_{n}]_{n\in \mathbb{N}}\in E(\Sigma): W\supset U_{l}\hspace{2mm}\text{for $l$ sufficiently large}\}$. We call the corresponding topological space \emph{the space of ends~of~$\Sigma$}.

\begin{proposition}
\label{p:1.2}
\cite[Proposition 3]{Richards63} The space of ends of a connected orientable surface~$\Sigma$ is totally disconnected, compact and Hausdorff. In particular, $E(\Sigma)$ is homeomorphic to a closed subspace of the Cantor set.
\end{proposition}

The \textit{genus} of a surface $\Sigma$ is the supremum of the genera of its compact subsurfaces.
A surface is said to be \textit{planar} if all of its compact subsurfaces are of genus zero. An end~$[U_n]_{n\in\mathbb{N}}$ is called \textit{planar} if there exists $l\in\mathbb{N}$ such that $U_l$ is planar. We define $E_{\infty}(\Sigma)\subset E(\Sigma)$ as the set of all ends of $\Sigma$ which are not planar. It follows from the definition that $E_{\infty}(\Sigma)$ forms a closed subspace of $E(\Sigma)$. Every $f\in\Homeo(\Sigma)$ induces an element of $\Homeo(E(\Sigma))$ preserving $E_{\infty}(\Sigma)$ which we denote by $f^*$.

\begin{theorem}[Classification of orientable surfaces]
\label{t:1.1}
\cite[Chapter 5]{Kerekjarto23} Let $\Sigma$ and~$\Sigma^{'}$ be two orientable surfaces of the same genus. Then $\Sigma$ and $\Sigma^{'}$ are homeomorphic if and only if $E_{\infty}(\Sigma)\subset E(\Sigma) $ and $E_{\infty}(\Sigma^{'})\subset E(\Sigma^{'})$ are homeomorphic as nested topological spaces.
\end{theorem}

Moreover, as explained in~\cite{Richards63}, every pair of nested closed subsets of the Cantor set can be realized as the space of ends of an orientable surface.

If $S$ is a finite-type surface, then $E_\infty(S)=\emptyset$ and $E(S)$ is in natural bijection\,\footnote{In general, any isolated puncture of a surface corresponds to an isolated point in its space of ends (not accumulated by genus).} with the set of punctures of $S$. In other words, Theorem~\ref{t:1.1} above is valid for finite-type surfaces.
Even though there are uncountably many infinite-type surfaces, there are some which have recieved special names because they play a distinguished role in some mathematical context. We finish this section by recalling some of these names, mainly because we use them later in the article.
\begin{enumerate}
    \item If $E_\infty(\Sigma)=E(\Sigma)$ is a singleton, $\Sigma$ is called a \emph{Loch Ness monster},
    \item If $E_\infty(\Sigma)=E(\Sigma)$ consists of two points, $\Sigma$ is called a \emph{Jacob's ladder}.
    \item If the genus of $\Sigma$ is zero and $E(\Sigma)$ is homeomorphic to $\omega+1$, $\Sigma$ is called a \emph{flute surface}.
    \item If $E_\infty(\Sigma)=\emptyset$ and $E(\Sigma)$ is homeomorphic to a Cantor set, $\Sigma$ is called a \emph{Cantor tree}.
\end{enumerate}

This nomenclature can be traced back to ~\cite{Sullivan81} and~\cite{Ghys95}.

\subsection{Big mapping class groups}

To any orientable surface, we associate its \emph{extended mapping class group}, which is the group of all homeomorphisms (including those that reverse orientation) of the surface in question modulo isotopy\,\footnote{That is, two homeomorphisms are equivalent if and only if they are isotopic.}. If we restrict ourselves to the group of homeomorphisms of the surface preserving orientation and mod out by isotopy, the result is the \emph{mapping class group}, which is the main object of study in this article. Note that by definition, the mapping class group is a subgroup of index two of the extended mapping class group. The elements of the (extended) mapping class group are called mapping classes and we will in the following often implicitly choose a representative when refering to a mapping class.

If $\Sigma$ is a surface, we denote by $\MCG^*(\Sigma)$ and $\MCG(\Sigma)$ its extended mapping class group and mapping class group, respectively.

\subsubsection{Topology of mapping class groups and the curve graph} We endow $\Homeo(\Sigma)$ with the compact-open topology and $\MCG^*(\Sigma)$ with the quotient topology. We refer the reader to N.~Vlamis' note~\cite{Vlamis_note} on the topology of mapping class groups for details about the topics discussed in this section. Recall that a topological group is \emph{Polish} if it is separable and completely metrizable.

\begin{proposition}
    \label{Proposition:BigModsArePolish}
Let $\Sigma$ be an infinite-type surface. Then $\MCG^*(\Sigma)$ is a Polish group. Moreover, every closed subgroup of  $\MCG^*(\Sigma)$ is Polish.
\end{proposition}

We make extensive use of an alternative description for the topology of $\MCG^*(\Sigma)$, detailed in what follows.

\textbf{The curve graph}. Let $\Sigma$ be a surface. A simple closed curve on $\Sigma$ is called \emph{essential} if it is not homotopic to a point or a puncture. In the following, we often interchange an isotopy class of a curve for one of its representatives to keep the notation lighter. We define the curve graph $C(\Sigma)$ to be the graph whose vertices are (isotopy classes of) essential simple closed curves in $\Sigma$ and whose edges are defined by pairs of (isotopy classes of) curves having representatives which do not intersect. The curve graph gives a useful and concise way to think of the mapping class group:

\begin{theorem}~\cite{hernandez-etal,Bavard-Dowdall-Rafi}
  \label{THM:MCG_is_Aut}
Let $\Sigma$ be an infinite-type surface. Then the homomorphism $\MCG(\Sigma)^*\to \Aut(C(\Sigma))$ describing the natural action of $\MCG^*(\Sigma)$ on isotopy classes of curves is an algebraic isomorphism.
\end{theorem}

Let us denote by $C^0(\Sigma)$ the set of vertices of the curve graph. For every finite subset $A$ of $C^0(\Sigma)$ let
\begin{equation}
    \label{EQ:Basic-permutation-topo}
U_A:=\{f\in{\MCG^*(\Sigma)}: f(a)=a \text{ for all } a\in A\}.
\end{equation}
Then the set of all $f\cdot U_A$, where $A$ is as above and $f\in\MCG^*(\Sigma)$, generates the \emph{permutation topology} on $\MCG^*(\Sigma)$, which coincides with the compact-open topology defined above. The group  $\Aut(C(\Sigma))$ can also be topologized this way and $\MCG(S)^*\to \Aut(C(\Sigma))$ in Theorem~\ref{THM:MCG_is_Aut} is a homeomorphism.

The following proposition and corollary tells us that when it comes to searching for dense conjugacy classes, there is nothing to be found in extended big mapping class groups.

\begin{proposition}
  \label{normal} If a Polish group $G$ contains a closed (equivalently, open)  normal subgroup of countable (or finite) index, then $G$ does not have a dense conjugacy class.
\end{proposition}

\begin{proof}
Let $H$ be such a subgroup. Then $H$ is clopen, see~\cite{BeckerKechris}. A dense conjugacy class would have to intersect $H$ and, as $H$ is normal, hence would be contained in $H$. However, this is not possible as the dense conjugacy class would also have to intersect the open set~$gH$ for any $g \notin H$.
\end{proof}

Given that for any infinite-type surface $\Sigma$, $\MCG(\Sigma)$ is a closed, index $2$ normal subgroup of $\MCG^*(\Sigma)$, we have the following:

\begin{corollary}
  \label{BigExtendedMCGHaveNoDenseOrbits}
  Let $\Sigma$ be an infinite-type surface. Then the \emph{extended} mapping class group $\MCG^*(\Sigma)$ has no dense conjugacy classes.
\end{corollary}

It is because of this corollary that we center our attention henceforth on $\MCG(\Sigma)$ rather than $\MCG^*(\Sigma)$.

\subsubsection{Dehn twists}
    \label{SSEC:Dehn-Twists}
    Dehn twists form a class of homeomorphisms in $\MCG(\Sigma)$ that plays a distiguished role in our proofs. The idea to define a Dehn twist is the following. Let $\alpha\subset\Sigma$ be an essential simple closed curve and $A_\alpha$ a tubular neighbourhood of $\alpha$. Modulo taking a smaller neighbourhood, we can suppose that the closure of $A_\alpha$ in $\Sigma$ is homeomorphic to the annulus
\begin{equation*}
A=\{z\in\mathbb{C}: |z|=1\}\times[0,1]
\end{equation*}
where the image of the curve $\alpha$ corresponds to $\{z\in\mathbb{C}: |z|=1\}\times\{\frac{1}{2}\}$. Then the map $(z,t)\to(e^{2\pi i t}z,t)$ defines a homeomorphism $h$ of $A$ (called a twist) which restricts to the identity on $\partial A$. Using the fact that $A$ and $A_\alpha$ are homeomorphic, one can define $\tau_\alpha\in\Homeo^+(\Sigma)$, supported and topologically conjugated to $h$ in~$A_\alpha$, which is equal to the identity in $\Sigma\setminus A_\alpha$. We call $\tau_\alpha$ a Dehn twist\,\footnote{Formally speaking, there are two possible Dehn twists about $\alpha$, once we fix an orientation on $\Sigma$: the positive Dehn twist about $\alpha$, which twists the surface in a clockwise manner with respect to the chosen orientation. The other one is called the negative Dehn twist about $\alpha$. In this article, $\tau_\alpha$ will always denote the positive Dehn twist about $\alpha$.} about $\alpha$ and, abusing notation, we also write $\tau_\alpha$ for the corresponding isotopy class. Note that~$\tau_\alpha$ is non-trivial, torsion-free and independent of the representatives chosen in the isotopy classes of $\alpha$ and $A_\alpha$. For more details, we refer the reader to Section 3.1.1 in~\cite{FarbMargalit}.

We list some properties of Dehn twists that are used in the proofs of some of our main theorems:
\begin{enumerate}
    \item Let $\alpha\subset \Sigma$ be an essential simple closed curve and $h\in\Homeo^+(\Sigma)$. Then $\tau_{h(\alpha)}=h\circ \tau_\alpha\circ h^{-1}$.
    \item Let $\alpha,\beta\subset \Sigma$ be essential simple closed curves whose geometric intersection $i(\alpha,\beta)$ is different from zero. Then $i(\tau_{\beta}^k(\alpha),\beta)$ is different from zero for every $k\in\Z$.
\end{enumerate}

We conclude with the following definition, which is important for the proof of our main results. Here, instead of curves, we refer to \emph{multicurves}, that is, locally finite, pairwise disjoint, and pairwise non-isotopic collections of essential simple closed curves in $\Sigma$.

\begin{definition}
        \label{Def:minpowerandmultitwist}
 Let $\Gamma$ be a closed subgroup of $\MCG(\Sigma)$ that contains, for every essential curve $\alpha\subset\Sigma$, a non-trivial power of the Dehn twist $\tau_\alpha$. Define $N_{\alpha}$ as the minimal positive power of $\tau_{\alpha}$ that is contained in $\Gamma$.
 Then, for each multicurve $M = \{\alpha_{i}\}_{i \in I}$, we define
 $$\tau_{M}:= \prod_{i \in I} \tau_{\alpha_i}^{N_{\alpha_i}} \in \Gamma.$$
\end{definition}

\subsection{Fraïssé classes and Fraïssé limits}

 The backbone on which most of our proofs rely is the work of Kechris--Rosendal~\cite{kechris-rosendal} and Truss~\cite{truss} on automorphism groups of countable structures. The main strategy to study topological properties of  conjugacy classes of big mapping class groups is to find, for any closed subgroup $\Gamma$ of $\MCG(\Sigma)$, an (ultrahomogeneous) countable structure whose automorphism group is isomophic to $\Gamma$ and then use the formalism of Fra\"{\i}ss\'e classes. In what follows, we recall the generalities about structures, Fra\"{\i}ss\'e classes and the process of Fra\"{\i}ss\'efication.

\textbf{Structures}. Let $X$ be a set and $n \in \N$. An $n$-ary relation $R$ is a subset\,\footnote{$R$ can be the empty relation, i.e.\ $R  = \emptyset$.} $R \subset X^n$, where we say that $(a_1, \cdots, a_n)$ \textit{satisfies} $R$ if $(a_1, \cdots , a_n) \in R$. A relational\,\footnote{In this article, all our structures are \emph{relational}, i.e.\ $\sigma$ consists only of relations. In this kind of structures, all subsets are structures themselves. This is no longer the case in structures with operations, such as groups.}  structure $\mathbf{K}$  is determined by:
\begin{itemize}
 \item \textbf{ a universe}: this is a set $K = |\mathbf{K}|$,
 \item \textbf{ a signature}: this is a collection of ``abstract'' relations $\sigma  = \{(R_\lambda, n_\lambda) : \lambda \in \Lambda \}$, where $n_\lambda \in \N$ and $\Lambda$ is a set,
 \item \textbf{an interpretation}: this is an association that gives us an $n_\lambda$-ary relation $R_\lambda^{\mathbf{K}} \subset K^{n_\lambda}$ for every element $(R_\lambda, n_\lambda) \in \sigma$.
\end{itemize}

For a structure $\mathbf{K}$, a substructure $\mathbf{S}$ of $\mathbf{K}$ is determined by a subset~$S$ of the universe~$K$ where $\mathbf{S}$ has the same signature as $\mathbf{K}$ and whose interpretation function is given by the identity
 \[	R_\lambda^{\mathbf{S}} = (R_\lambda^{\mathbf{K}} \cap S^{n_\lambda})	\]
for every $(R_\lambda, n_\lambda) \in \sigma$.

\emph{Isomorphisms between structures}.  For two structures $\mathbf{K}$ and $\mathbf{T}$ with the same signature, an isomorphism between $\mathbf{K}$ and $\mathbf{T}$ is a bijection between their universes
 \[	F : |\mathbf{K}| \rightarrow |\mathbf{T}|	\]
such that
\begin{equation}
    \label{EQ:Relation-preservation}
 	F(R_\lambda^{\mathbf{K}}) = R_\lambda^{\mathbf{T}}\hspace{2mm}\text{and}\hspace{2mm}F^{-1}(R_\lambda^{\mathbf{T}}) = R_\lambda^{\mathbf{K}}
\end{equation}
for every $(R_\lambda, n_\lambda) \in \sigma$.
A structure $\mathbf{S}$ is \emph{embeddable} in another structure $\mathbf{K}$ if there is an isomorphism between $\mathbf{S}$ and a substructure of $\mathbf{K}$.

The following properties help to determine topological properties of the automorphism group of a countable structure. Here and in the following, when we consider sets of structures, all the structures are supposed to have the same given signature.

\begin{definition}
    \label{Def:JEP-H-A}
Let $\mathcal{K}$ be a countable set of finite non-isomorphic structures. We say that~$\mathcal{K}$ satisfies:
\begin{itemize}
  \item[\textbf{(HP)}] If $A \in \mathcal{K}$ and $B \subset A$ is a substructure, then $B$ is isomorphic to an element of $\mathcal{K}$. Here {\rm \textbf{HP}} abbreviates the \emph{hereditary property}.

  \item[\textbf{(JEP)}] If $A, B \in \mathcal{K}$ are two given elements, then there exists an element $C \in \mathcal{K}$ such that $A, B$ are embeddable in $C$. Here {\rm \textbf{JEP}} abbreviates the \emph{joint embedding property}.

  \item[\textbf{(AP)}] If $A,B,C \in \mathcal{K}$ are structures together with embeddings
        \[	f:A \rightarrow B, \qquad g: A \rightarrow C,	\]
    then there exist an element $D \in \mathcal{K}$ and embeddings
        \[	f': B \rightarrow D, \qquad g': C \rightarrow D,	\]
    such that $f' \circ f = g' \circ g$. That is, the following diagram commutes:

        \begin{center}
        \begin{tikzpicture}
         \matrix (m) [matrix of math nodes,row sep=1em,column sep=2em,minimum width=2em]{
           & B & \\
          A & & D \\
           & C & \\};
      \path[right hook->]
        (m-2-1) edge node[above]{$f$} (m-1-2)
        (m-1-2) edge node[above]{$f'$} (m-2-3)
        (m-2-1) edge node[below]{$g$} (m-3-2)
        (m-3-2) edge node[below]{$g'$} (m-2-3);
    \end{tikzpicture}
    \end{center}
Here {\rm {\rm \textbf{AP}}} abbreviates the \emph{amalgamation property}.

\end{itemize}
\end{definition}

Given a countable structure $\mathbf{K}$ (i.e.\ the universe and signature are countable), consider the class of finite substructures modulo isomorphism in $\mathbf{K}$, that is
 \[	\mathcal{K}_\mathbf{K} = \{ \mathbf{S}  \textrm { finite structure} : 	\mathbf{S}  \textrm { is embeddable in } \mathbf{K} \} / \textrm{isomorphism} .\]

\begin{theorem}\cite[Theorem 7.1.1]{Hu} \label{Theorem:finite_substructures}
If $\mathcal{K}$ is a countable set of finite non-isomorphic structures, then there exists a countable structure $\mathbf{K}$ such that $\mathcal{K} = \mathcal{K}_\mathbf{K}$ if and only if $\mathcal{K}$ satisfies~{\rm \textbf{HP}} and~{\rm \textbf{JEP}}.
\end{theorem}

\begin{proof}[Sketch of proof.]
If $\mathcal{K} = \mathcal{K}_\mathbf{K}$, it is clear from the definition of $\mathcal{K}_\mathbf{K}$ that it has to fulfill~\textbf{HP} and~\textbf{JEP}. For the other direction, enumerate the elements of~$\mathcal{K}$ as $(A_n : n \in \N )$  and construct a chain of structures
 \[	B_0 \subset B_1 \subset \cdots \subset B_m \subset ...	\]
where $B_0 = A_0$ and $B_{m+1}$ is an element of $\mathcal{K}$ that contains isomorphic copies of~$A_{m+1}$ and~$B_m$ (the existence of such an element is assured by {\rm \textbf{JEP}}). The desired structure is
 \[	\mathbf{K} = \bigcup_{n \in \N} B_n.\]
 By construction, every element of $\mathcal{K}$ is a substructure of $\mathbf{K}$ and by \textbf{HP}, every finite substructure of $\mathbf{K}$ is an element of $\mathcal{K}$, hence $\mathcal{K} = \mathcal{K}_\mathbf{K}$.
\end{proof}

\begin{example}
If $\mathbf{K}$ is either $\Z$ or $\Q$, considered as a substructure of $\R$ with the binary relation
 \[	R = \{ (a,b) : a \leq b \} \subset \R^2	\]
given by inequality, then $\mathcal{K}_{\mathbf{K}}$ is the class of finite linear orders (in both cases).
\end{example}

\textbf{Fra\"iss\'e's construction}. The amalgamation property {\rm \textbf{AP}} can be rephrased in words as follows: if two finite structures in $\mathcal{K}$ share a common substructure, then those structures can be embedded in a bigger structure in $\mathcal{K}$ such that the common structure is contained in the intersection. In this section we recall that using this property, \textbf{HP} and \textbf{JEP}, one can define the limit of a class of finite structures.

\begin{definition}
    \label{Def:Ultrahomogeneous}
Let $\mathbf{K}$ be a structure. We say that $\mathbf{K}$ is \emph{ultrahomogeneous} if an isomorphism of any two finite substructures can be extended to a global isomorphism of~$\mathbf{K}$.
\end{definition}

For example, the bijection
	\[	\{1,2,3\} \rightarrow \{1,2,4\}	\]
fixing $1$ and $2$, and sending $3$ to $4$, is an isomorphism of linear orderings that cannot be extended to a global isomorphism of $\Z$, so $\Z$ is not an ultrahomogeneous structure.

The amalgamation property \textbf{AP} of a set $\mathcal{K}$ of structures is strongly related to the corresponding structure $\mathbf{K}$ being ultrahomogeneous as can be seen in the following extension of Theorem~\ref{Theorem:finite_substructures}.

\begin{theorem}\cite[Theorem~7.1.2]{Hu}
  \label{Theorem:Fraisse_Limit}
	If $\mathcal{K}$ is a countable set of finite non-isomorphic structures, then there exists a countable, \textbf{ultrahomogeneous} structure $\mathbf{K}=\mathbf{K}_\mathcal{F}$ which is unique up to isomorphism such that $\mathcal{K} = \mathcal{K}_\mathbf{K}$ if and only if $\mathcal{K}$ satisfies {\rm \textbf{HP}} , {\rm \textbf{JEP}} and {\rm \textbf{AP}}.
\end{theorem}

\begin{definition}
  Let $\mathcal{F}$ be a countable (up to isomorphism) class that contains structures of arbitrarily large, but finite cardinality. The class $\mathcal{F}$ is called a \emph{Fra\"{\i}ss\'e class} if it satisfies {\rm \textbf{HP}}, {\rm \textbf{JEP}} and {\rm \textbf{AP}}. The corresponding ultrahomogeneous structure~$\mathbf{K}_\mathcal{F}$ given by Theorem~\ref{Theorem:Fraisse_Limit} is called the \emph{Fra\"{\i}ss\'e limit} of $\mathcal F$.
\end{definition}

Observe that the partial isomorphism
	\[	\{1,2,3\} \rightarrow \{1,2,4\}	\]
of linear orderings can be extended to the automorphism $F: \Q \rightarrow \Q$, defined by
	\[	F(x) = \begin{cases}
			x, & x \leq 2 \\
			2x - 2, & 2 \leq x \leq 3 \\
			x + 1, & 3 \leq x
		\end{cases}.	\]
In general, we can extend any isomorphism of finite linear orderings (which can be realized as finite substructures of $\Q$) to an automorphism of $\Q$, using as above piecewise linear maps. This gives us the ultrahomogeneity property and thus, the fact that $\Q$ is the Fra\"iss\'e limit of finite linear orders.

\subsection{Fra\"{\i}ss\'e classes and automorphism groups of countable structures} We outline below how Fra\"{\i}ss\'e theory can be used to analyze the automorphism groups of countable structures. As before, we follow ideas and notations from~\cite{kechris-rosendal, truss}.

\smallskip

\begin{definition}
    \label{DEF:Class-F_p}
    Given a Fra\"{\i}ss\'e class $\mathcal F$, we denote by $\mathcal F_p$ the class of all $ S = \langle A,\psi : B \to C\rangle $, where $A,B,C \in\mathcal F$, $B, C \subseteq A$  and $\psi$ is an isomorphism between~$B$ and $C$. We say that $ S = \langle A,\psi : B \to C\rangle $ \emph{embeds into} $ S' = \langle A',\psi' : B' \to C'\rangle $ if there is an embedding $f:A\to A'$ such that
    $f$ embeds $B$ into $B'$ and $C$ into $C'$ and $f\circ\psi\subseteq \psi'\circ f$, \emph{i.e.}\ the function $\psi'\circ f$ extends the function $f\circ\psi$.
    This can be summarized in the following diagram\,\footnote{Note that this diagram is not commutative.}:

 \begin{center}
 \begin{tikzpicture}
 \matrix (m) [matrix of math nodes,row sep=1em,column sep=2em,minimum width=2em]{
   B' & & & \\
   & B & & \\
    & & A & A' \\
   & C & & \\
   C' & & &\\};
 \path[right hook->]
   (m-2-2) edge node[below]{$f$} (m-1-1)
   (m-4-2) edge node[above]{$f$} (m-5-1)
   (m-3-3) edge node[above=-3pt]{$f$} (m-3-4);
 \path[->]
   (m-2-2) edge node[above]{$\subset$} (m-3-3)
   (m-4-2) edge node[above]{$\subset$} (m-3-3)
   (m-2-2) edge node[left]{$\psi$} (m-4-2)
   (m-1-1) edge node[left]{$\psi'$} (m-5-1);
 \path[->, bend left=15]
   (m-1-1) edge node[above]{$\subset$} (m-3-4);
 \path[->, bend right=15]
   (m-5-1) edge node[below]{$\subset$} (m-3-4);
 \end{tikzpicture}
 \end{center}
\end{definition}

\begin{remark}
    \label{rmk:A=BUC}
Given $\langle A,\psi : B \to C\rangle$ in a Fra\"{\i}ssé class, we can suppose without loss of generality that $A=B\cup C$. For this, we carefully choose representatives in the isomorphism classes $A$, $B$ and $C$ such that $\psi$ can be extended to $\psi' : A \to \psi'(A)$ by ultrahomogeneity. Then we can choose $A' = A \cup \psi'(A)$, $B'=A$ and $C'=\psi'(A)$.
\end{remark}

\begin{remark}
    \label{rk:JEP_in_Fp}
It should be clear from the definitions what {\bf JEP} for
$\mathcal F_p$ means but we shall spell it out to avoid confusion. The class $\mathcal F_p$ satisfies {\bf JEP} if for all $A,B,C, A',B',C' \in\mathcal F$ such that  $B, C\subseteq A$ and $B', C'\subseteq A'$ with isomorphisms $\psi:B \to C$ and $\psi':B' \to C'$, there is an $A''$ and embeddings $f:A \to A''$ and $f': A'\to A''$ such that
\[
\psi'': f(B)\cup f'(B') \to f(C)\cup f'(C')
\]
defined by
$\psi'' (f(b))= f(\psi(b))$ if $b\in B$ and
$\psi'' (f'(b))= f'(\psi'(b))$ if $b\in B'$,
is an isomorphism.
\end{remark}

We consider one last property for our study of conjugacy classes.

\begin{definition}
    \label{DEF:WAP}
We say that a countable set of finite non-isomorphic structures~$\mathcal{K}$ satisfies~{\bf WAP} if for any $ S\in\mathcal F $, there is a $ T \in \mathcal F$ and an embedding $e:S\to T$, such that for any pair of embeddings $f : T \to T_0$ and $g : T \to  T_1$, where $T_0,T_1 \in\mathcal F$, there is a $U \in\mathcal F$ and embeddings $r : T_0 \to U$, $s : T_1 \to U$ such that $r \circ f \circ e = s \circ g \circ e$.
\begin{center}
    \begin{tikzpicture}
 \matrix (m) [matrix of math nodes,row sep=1em,column sep=2em,minimum width=2em]{
  & & T_0 & \\
  S & T & & U \\
  & & T_1 & \\};
\path[right hook->]
(m-2-1) edge node[above]{$e$} (m-2-2)
(m-2-2) edge node[above]{$f$} (m-1-3)
(m-1-3) edge node[above]{$r$} (m-2-4)
(m-2-2) edge node[below]{$g$} (m-3-3)
(m-3-3) edge node[below]{$s$} (m-2-4);
\end{tikzpicture}
\end{center}
Here {\bf WAP} abbreviates the \emph{weak amalgamation property}. We shall say that $\mathcal{F}$ satisfies $\bf{WAP}$ \textbf{locally} if there exists $A\in\mathcal{F}$ such that $\bf{WAP}$ holds for all $B\in\mathcal{F}$ into which $A$ embeds. The definition of these properties for $\mathcal{F}_p$ is analogous.
\end{definition}

Our main results use the following theorem, which integrates several statements that can be found in~\cite{kechris-rosendal}.

\begin{theorem}
    \label{Theorem:Criteria}
  Let $\mathcal{F}$ be a Fra\"{\i}ss\'e class with Fra\"{\i}ss\'e limit $\mathbf K$. Then:
  \begin{enumerate}
      \item $\Aut(\mathbf K)$ has a dense conjugacy class iff $\mathcal F_p$ satisfies {\bf JEP},
      \item  $\Aut(\mathbf K)$ has a co-meager conjugacy class iff $\mathcal F_p$ satisfies
      {\bf JEP} and {\bf WAP},
      \item $\Aut(\mathbf K)$ has a non-meager conjugacy class iff $\mathcal F_p$ satisfies {\bf local WAP}.
  \end{enumerate}
\end{theorem}

\textbf{Fra\"{\i}ss\'efication}. There is a standard process to turn any countable structure $\mathbf K$ into a Fra\"{\i}ss\'e limit by enlarging the signature. We discuss the generalities of the process in what follows. For a given structure $\mathbf{K}$, denote the group of automorphisms of $\mathbf{K}$ by $\Aut(\mathbf{K})$. Construct a new countable structure $\widetilde{\mathbf{K}}$ having the same universe as $\mathbf{K}$, all the relations of~$\mathbf{K}$ and, for every $n \in \N$ and~$\overline{b} \in K^n$,
add to the signature
	\[	(R_{\overline{b}},n) = \{ \overline{a} \in K^n : \overline{a} \ E_{\mathbf{K}} \ \overline{b} \} \subset K^n,	\]
where the equivalence relation
	\[	(a_1, \ldots , a_n) E_{\mathbf{K}} (b_1, \ldots , b_n)	\]
holds if and only if there is an element $h \in \Aut(\mathbf{K})$ such that $h(a_j) = b_j$ for all $j\in \{1,\ldots,n\}$.
The interpretation of $(R_{\overline{b}},n)$, which we denote by $R_n(\overline{b})$, is the $\Aut(\mathbf{K})$-orbit\,\footnote{The $\Aut(\mathbf{K})$-action on $K^n$, which we consider here, is $g\cdot(b_1,\ldots,b_n):=(g(b_1),\ldots,g(b_n))$.} of $\overline{b}\in K^n$.

\begin{proposition}
  \label{Proposition:Fraissefication_Automorphism_Group}
If $\mathbf{K}$ is a countable structure and $\widetilde{\mathbf{K}}$ is the structure with the enlarged language as above, then
	\[	\Aut(\widetilde{\mathbf{K}}) = \Aut(\mathbf{K})	\]
and $\widetilde{\mathbf{K}}$ is ultrahomogeneous.
\end{proposition}

\begin{proof}
As we enlarge the signature from $\mathbf{K}$ to $\widetilde{\mathbf{K}}$, we immediately have that $\Aut(\widetilde{\mathbf{K}}) \subset \Aut(\mathbf{K})$. Take $h \in \Aut(\mathbf{K})$ and $\overline{b} \in K^n$, then we have that
	\[	R_n(\overline{b}) = \{ f(\overline{b}) : f \in \Aut(\mathbf{K}) \},	\]
which implies that
	\[	h(R_n(\overline{b})) = \{ h \circ f(\overline{b}) : f \in \Aut(\mathbf{K}) \} = R_n(\overline{b}),	\]
that is, $h$ preserves all the new relations and thus $h \in \Aut(\widetilde{\mathbf{K}})$, so we also have $\Aut(\mathbf{K}) \subset \Aut(\widetilde{\mathbf{K}})$.

To prove the ultrahomogeneity of $\widetilde{\mathbf{K}}$, let $g : A \rightarrow B$ be an isomorphism of finite substructures $A,B \subset \widetilde{\mathbf{K}}$.
Enumerate $A$ as $\langle a_1, \cdots , a_m \rangle$ and take the vector $\overline{a} = (a_1, \cdots , a_m) \in A^m$. We show that $\overline{a}$ and $g(\overline{a})$ are equivalent. Recall that, for every $n \in \N$ and $\overline{b} \in K^n$, the relation $R_n(\overline{b})$ interprets as
	\[	R_n(\overline{b}) \cap A^n \quad \textrm{and} \quad R_n(\overline{b}) \cap B^n	\]
in $A$ and $B$, respectively. As $g$ is an isomorphism of structures, we have that
	\[	g(R_n(\overline{b}) \cap A^n) = R_n(\overline{b}) \cap B^n,	\]
because it has to send the interpretation on $A$ to the interpretation on $B$ of the same relation.
By taking $n=m$ and $\overline{b}=\overline{a}$ we get
	\[	g(\overline{a}) \in R_m(\overline{a}) \cap B^m \subset R_m(\overline{a}),	\]
which tells us that there is an automorphism $h \in \Aut(\mathbf{K})$ such that $g(\overline{a}) = h(\overline{a})$.
As $g$ is a bijection, $\langle g(a_1), \cdots , g(a_m) \rangle$ is an enumeration of $B$ and $h \in \Aut(\mathbf{K}) = \Aut(\widetilde{\mathbf{K}})$ is an automorphism that extends $g$.
\end{proof}

\begin{remark}
 If $\mathbf{K}$ is already ultrahomogeneous, then it is in general false that $\widetilde{\mathbf{K}} \cong \mathbf{K}$. As an example of this, we can consider $(\Q, \leq)$ which is already ultrahomogeneous but only has binary relations, whereas $\widetilde{\Q}$ will have $n$-ary relations, for larger $n$.
\end{remark}

\begin{example}
Consider $\Z$ as a structure with the binary relations given by inequality, then $\Aut(\Z) \cong \Z$. The new relations in $\widetilde{\Z}$ are just sets of arithmetic progressions
	\[	\{ (a_1 + k, \cdots , a_n + k) : k \in \Z \} \subset \Z^n.	\]
Hence, we can see this enriched structure on $\Z$ as a Fra\"iss\'e limit of finite linear orders, but these finite structures are also enriched by the relations of finite arithmetic progressions.
\end{example}

\subsection{Fra\"{i}ss\'efying the curve graph}\label{SSEC:Fraissefication} Any countable graph is a structure: the universe is the vertex set, signature and interpretation are given by the edges. In particular, for any surface $\Sigma$, the curve graph $C(\Sigma)$ is a structure\,\footnote{Edges give only binary relations. If we want to consider $C(\Sigma)$ as an abstract simplicial complex, we have to consider higher order relations. However, since the curve complex is a flag complex, all relations are induced by the binary relations corresponding to edges.}, but it is not ultrahomogeneous in general. This is due to Theorem~\ref{THM:MCG_is_Aut}, which tells us that automorphisms of $C(\Sigma)$ are geometric, \emph{i.e.}\ are induced by mapping classes. Indeed, consider for example the Cantor tree surface  $\Sigma:=\mathbb{S}^2\setminus Cantor$; it is not difficult to find two triangles $\{a_i,b_i,c_i\}_{i=1,2}$ in $C(\Sigma)$ such that, on the one hand, $\Sigma\setminus \{a_1,b_1,c_1\}$ has a connected component homeomorphic to a pair of pants and, on the other, all components $\Sigma \setminus \{a_2,b_2,c_2\}$ are surfaces of infinite type. Any isomorphism between these triangles cannot be promoted to a global automorphism of $C(\Sigma)$ because such an automorphism would then be induced by a homeomorphisms of $\Sigma$ sending a pair of pants to an infinite-type surface and this is absurd. In what follows, we describe a process that, given a closed subgroup $\Gamma<\MCG^*(\Sigma)$, uses the curve graph to produce an ultrahomogeneous structure whose automorphism group is isomorphic to $\Gamma$.

Let $\Gamma<\MCG^*(\Sigma)$ be a closed subgroup. We add to the signature of $C(\Sigma)$, for every $n\in\mathbb{N}$, the relations
	\[	(R_{\overline{b}},n) = \{ \overline{a} \in  (C^0(\Sigma))^n : \overline{a} \ E_{\mathbf{K}} \ \overline{b} \}.
    \]
where $ (C^0(\Sigma))^n$ denotes the set of $n$-tuples of vertices of $C(\Sigma)$, $\overline{b} \in (C^0(\Sigma))^n$ and the equivalence relation
	\[	(a_1, \cdots , a_n) E_{\mathbf{K}} (b_1, \cdots , b_n)	\]
holds if and only if there is an element $h \in\Gamma$ such that $h(a_j) = b_j$ for all $j\in\{1,\ldots,n\}$.
The interpretation $R_n(\overline{b})$ of $(R_{\overline{b}},n)$ is the $\Gamma$-orbit of $\overline{b}\in (C^0(\Sigma))^n$.

\begin{definition}
    \label{DEF:FullCurveComplex}
Let $\Sigma$ be an infinite-type surface, $\Gamma<\MCG^*(\Sigma)$ a closed subgroup and~$C_\Gamma(\Sigma)$ the ultrahomogeneous countable structure obtained from the curve graph by the process described above. We call~$C_\Gamma(\Sigma)$ the \emph{full curve graph} of~$\Sigma$ (w.r.t.\ to $\Gamma$) and $\mathcal F_{C_\Gamma(\Sigma)}$ the Fra\"{\i}ss\'e class of finite structures embeddable in $C_\Gamma(\Sigma)$.
\end{definition}

Note that by definition, $C_\Gamma(\Sigma)$ is the Fra\"{\i}ss\'e limit of $\mathcal F_{C_\Gamma(\Sigma)}$, and hence it is an ultrahomogeneous structure.
The following lemma is fundamental for our proofs because it allows us to use the technology given by Theorem~\ref{Theorem:Criteria} to study conjugacy classes in any closed subgroup $\Gamma<\MCG^*(\Sigma)$.

\begin{lemma}
        \label{Lemma:Aut-Gamma-AreIsomorphic}
Let $\Sigma$ be an infinite-type surface, $\Gamma<\MCG^*(\Sigma)$ a closed subgroup and~$C_\Gamma(\Sigma)$ the full curve graph of $\Sigma$ w.r.t.~$\Gamma$. Then
the morphism $\Gamma\to\Aut(C_\Gamma(\Sigma))$ which associates to a mapping class $f\in \Gamma$ the automorphism $\alpha \mapsto f(\alpha)$ is an isomorphism.
\end{lemma}

\begin{proof}
Theorem~\ref{THM:MCG_is_Aut} guarantees the injectivity of $\Gamma\to\Aut(C_\Gamma(\Sigma))$ and that any $\Phi\in\Aut(C_\Gamma(\Sigma))$ is induced by a mapping class $h\in\MCG^*(\Sigma)$. In what follows, we show that~$h$ is indeed in $\Gamma$. Let $(A_i)_{i\geq 0}$ be an exhaustion of $C^0(\Sigma)$ by finite sets, that is, $A_i\subset A_{i+1}$ for all $i\geq 0$ and $\cup_{i\geq 0}A_i=C^0(\Sigma)$. Define $\Phi_i:=\Phi \restriction A_i$. It is then sufficient to show that for every $i\geq 0$, we have that $\Phi_i=\gamma_i \restriction A_i$ for some $\gamma_i\in\Gamma$. Indeed, since $\Gamma$ is closed and $(A_i)$ is an exhaustion, the Alexander method (see~\cite{hernandez-etal-2} for details) implies that $(\gamma_i)$ converges to $h$ and therefore  $\Gamma\to\Aut(C_\Gamma(\Sigma))$ is surjective. Let now $A_i=\{a_1,\ldots,a_{m_i}\}$ and $\overline{a_i}=(a_1,\ldots,a_{m_i})\in A_i^{m_i}$. Given that $\Phi\in\Aut(C_\Gamma(\Sigma))$, we have:
$$
\Phi(\Gamma\overline{a_i}\cap A_i^{m_i})=\Gamma\overline{a_i}\cap \Phi(A_i^{m_i}).
$$
In particular, this equation implies that $\Phi_i=\gamma_i \restriction A_i$ for some $\gamma_i\in\Gamma$.
\end{proof}

\begin{remark}
The proof of Lemma~\ref{Lemma:Aut-Gamma-AreIsomorphic} is similar in spirit to the proof of Proposition~\ref{Proposition:Fraissefication_Automorphism_Group}. However, these results are different. In the context of the curve graph, Proposition~\ref{Proposition:Fraissefication_Automorphism_Group} tells us that by enriching $C(\Sigma)$ with $\Aut(C(\Sigma))\simeq\MCG^*(\Sigma)$, we get an ultrahomogeneous structure whose automorphism group is naturally isomorphic to $\MCG^*(\Sigma)$. Lemma~\ref{Lemma:Aut-Gamma-AreIsomorphic} tells us that if we enrich $C(\Sigma)$ only with elements of a closed subgroup $\Gamma<\MCG^*(\Sigma)$, then we get an ultrahomogeneous structure whose automorphism group is naturally isomorphic to $\Gamma$. 
\end{remark}

%% file: meager.tex
\section{Meager conjugacy classes}
	\label{section:meager}

Recall that, given a closed subgroup $\Gamma$ of $\MCG(\Sigma)$, we say that a conjugacy class of $\Gamma$ is meagre if it is a countable union of nowhere dense subsets of $\Gamma$. In this section, we prove the following:

\begin{theorem}[All conjugacy classes are meager]
	\label{THM:All_CC_are_meager_full_statement}
Let $\Sigma$ be an infinite-type surface and $\Gamma$ a closed subgroup of $\MCG(\Sigma)$ which contains, for every essential curve $\alpha\subset\Sigma$, a non-trivial power of the Dehn twist $\tau_\alpha$. Then all conjugacy classes of $\Gamma$ are meager in $\Gamma$.
 \end{theorem}

\subsection{Translating local WAP into the language of surfaces}
	\label{SSEC:Translation}

Let $\Sigma$ be an infinite-type surface and $\Gamma$ a closed subgroup of $\MCG^*(\Sigma)$. In this section, $\mathcal{F}_{p}$ is defined with respect to the Fra\"iss\'e class $\mathcal{F}_{C_\Gamma(\Sigma)}$, where $C_\Gamma(\Sigma)$ is the full curve graph of $\Sigma$ w.r.t.~$\Gamma$, see Definitions~\ref{DEF:Class-F_p} and~\ref{DEF:FullCurveComplex}. For the proof of Theorem~\ref{THM:All_CC_are_meager_full_statement}, we need to translate some concepts from the language of Fraïssé classes to the language of surfaces and homeomorphisms. This is done in the following paragraphs and summarized in Lemma~\ref{lemma:GeomLocalWAP}.

Let $\langle A, \psi: B \to C\rangle \in \mathcal{F}_{p}$.
As discussed in Remark~\ref{rmk:A=BUC}, we may assume that $A = B \cup C$.
Given that $C_\Gamma(\Sigma)$ is ultrahomogeneous, $\psi$ is induced by at least one element of $\Aut(C_\Gamma(\Sigma))$, which by Lemma~\ref{Lemma:Aut-Gamma-AreIsomorphic} is isomorphic to $\Gamma$. Hence, a way to translate $\langle A, \psi\rangle \in \mathcal{F}_{p}$ to the language of surfaces is to consider a pair $(h,S)$ with the property that $h: \Sigma \to \Sigma$ is an element of~$\Gamma$ that induces $\psi$, and $S$ is a finite-type subsurface of $\Sigma$ that contains the curves in $A$. In particular, this implies that $h(S)$ contains the curves in~$C$. Note that $S$ could be a surface with non-empty boundary, given that $\Sigma$ is of infinite type. Note also that given $\langle A, \psi\rangle\in \mathcal{F}_{p}$, there exist infinitely many pairs $(h,S)$ satisfying the previous property.

Let $\langle A, \psi: B \to C\rangle, \langle  A', \psi': B' \to C' \rangle \in \mathcal{F}_{p}$ be such that $\langle A, \psi\rangle$ embeds into $\langle A', \psi'\rangle$. A way to translate this to the language of surfaces is having two pairs $(h,S)$ and $(h',S')$ as in the preceding paragraph, corresponding to $\langle A, \psi\rangle$ and $\langle A', \psi'\rangle$, and $f\in\Gamma$, such that $f(S) \subset S'$ and $f \circ h \restriction S = h' \circ f \restriction S$ up to isotopy. In this case, we say that the pair $(h',S')$ \emph{extends the pair $(h,S)$} and we call $f(S)$ \emph{the copy of $S$ inside $S'$}. In case we need to specify~$f$, we say \emph{the pair $(h,S)$ is extended to $(h',S')$ by $f$}. Note that in this case, $f \circ h \circ f^{-1}$ has to coincide with $h'$ on the copy of $S$ inside $S'$, modulo isotopy.

Let $\langle A, \eta\rangle, \langle B, \psi \rangle, \langle C, \psi' \rangle, \langle D, \phi \rangle \in \mathcal{F}_{p}$ be such that $\langle A, \eta \rangle$ embeds into both $\langle B, \psi \rangle$ and $\langle C, \psi'\rangle$ via $f$ and $g$ respectively, and $\langle D, \phi \rangle$ is an amalgamation of these two embeddings, i.e.\ both $\langle B, \psi \rangle$ and $\langle C, \psi' \rangle$ embed into $\langle D, \phi \rangle$ via $f'$ and $g'$ respectively, and $f' \circ f = g' \circ g$. This translates as having pairs\,\footnote{As above, all surfaces involved are of finite type and the homeomorphisms in question belong to $\Gamma$.} $(h_{1}, S_{1}), (h_{2}, S_{2}), (h_{3}, S_{3}), (h_{4}, S_{4})$ corresponding to  $\langle A, \eta\rangle, \langle B, \psi \rangle, \langle C, \psi' \rangle, \langle D, \phi \rangle$ respectively and such that:
\begin{itemize}
 \item the pair $(h_{1},S_{1})$ is extended to both $(h_{2}, S_{2})$ and $(h_{3}, S_{3})$ by elements $f_{1}$ and $g_{1}$ of $\Gamma$ respectively, and
 \item the pairs $(h_{2},S_{2})$ and $(h_{3},S_{3})$ are extended to $(h_{4}, S_{4})$ by elements $f_{2}$ and $g_{2}$ of $\Gamma$, such that
  \begin{equation}
	  \label{EQ:ExtensionAP1}
  f_2 \circ f_{1} \restriction S_{1} = g_{2} \circ g_1 \restriction S_{1},\hspace{2mm} g_{2} \circ g_1(S_1) \subset S_{4}
\end{equation}
  and
  \begin{equation}
	  	\label{EQ:ExtensionAP2}
  f_2 \circ f_{1} \circ h_{1} \restriction S_{1} = h_{4} \circ f_2 \circ f_{1} \restriction S_{1} = g_{2} \circ g_1 \circ h_{1} \restriction S_{1} = h_{4} \circ g_{2} \circ g_1 \restriction S_{1}
\end{equation}
  modulo isotopy.
\end{itemize}
 In this case, we say that \emph{the pair $(h_{4}, S_{4})$ amalgamates the pairs $(h_{2},S_{2})$ and $(h_{3},S_{3})$ from~$(h_{1},S_{1})$}. Note that the extensions from $(h_{1},S_{1})$ to $(h_{2},S_{2})$ and $(h_{3},S_{3})$ might be a composition of finitely many extensions. As we see below, this is important to phrase properly the \textbf{local WAP} property in this language.

We integrate the discussion above, Definition~\ref{DEF:WAP} and Theorem~\ref{Theorem:Criteria} in the following:

\begin{lemma}\label{lemma:GeomLocalWAP}
Let $\Sigma$ be an infinite-type surface and $\Gamma$ be a closed subgroup of $\MCG(\Sigma)$. Then~$\Gamma$ has a non-meager conjugacy class if and only if the following geometric version of \textbf{local WAP} holds:

There exists a pair $(h_{-1},S_{-1})$ such that for all pairs $(h_{0},S_{0})$ that extend \linebreak $(h_{-1},S_{-1})$, there exists a pair $(h_{1},S_{1})$ which extends $(h_0, S_0)$ and satisfies that for all pairs $(h_{2},S_{2})$ and $(h_{3},S_{3})$ that extend $(h_{1},S_{1})$, there exists a pair $(h_{4},S_{4})$ that amalgamates $(h_{2},S_{2})$ and $(h_{3},S_{3})$ from $(h_{-1},S_{-1})$.

This can be summarized in the following diagram (read from left to right):
\begin{center}
\begin{tikzpicture}
 \matrix (m) [matrix of math nodes,row sep=1em,column sep=2em,minimum width=2em]{
  & & & (h_{2},S_{2}) & \\
  (h_{-1},S_{-1}) & (h_{0},S_{0}) & (h_{1},S_{1}) & & (h_{4},S_{4}) \\
  & & & (h_{3},S_{3}) & \\};
\path[right hook->]
(m-2-1) edge node[above]{$\forall$} (m-2-2)
(m-2-2) edge node[above]{$\exists$} (m-2-3)
(m-2-3) edge node[above]{$\forall$} (m-1-4)
(m-2-3) edge node[below]{$\forall$} (m-3-4)
(m-1-4) edge node[above]{$\exists$} (m-2-5)
(m-3-4) edge node[below]{$\exists$} (m-2-5);
\end{tikzpicture}
\end{center}
\end{lemma}

The proof of this lemma follows directly from the definitions and is left to the reader.

\subsection{Proof of Theorem~\ref{THM:All_CC_are_meager_full_statement}}

We show that \textbf{local WAP} always fails for $\mathcal{F}_{p}$, where this class is defined with respect to the Fra\"iss\'e class $\mathcal{F}_{C_\Gamma(\Sigma)}$, see Definitions~\ref{DEF:Class-F_p} and~\ref{DEF:FullCurveComplex}.

In the language of pairs formed by elements of $\Gamma$ and  finite-type subsurfaces of $\Sigma$ (see Section~\ref{SSEC:Translation} above), the failure of \textbf{local WAP} (see Definition~\ref{DEF:WAP}) can be phrased as follows. For every pair $(h_{-1},S_{-1})$, there exists a pair $(h_{0},S_{0})$ that extends $(h_{-1},S_{-1})$, so that for any pair $(h_{1},S_{1})$ that extends $(h_{0},S_{0})$ there exist two pairs $(h_{2},S_{2})$ and $(h_{3},S_{3})$ extending $(h_{1},S_{1})$ such that there does not exist a pair $(h_{4},S_{4})$ that amalgamates $(h_{2},S_{2})$ and $(h_{3},S_{3})$ from $(h_{-1},S_{-1})$.

Let $(h_{-1}, S_{-1})$ now be a given pair. Define $h_0=h_{-1}$ and $S_0$ as the closure of $S_{-1}$ inside~$\Sigma$. Suppose that $(h_1,S_1)$ is an extension of $(h_0,S_0)$ by~$g\in\Gamma$. In the rest of the proof, we:
\begin{enumerate}
	\item enlarge the surface $S_1$ in a convenient manner to obtain a surface $S_2=S_3$ containing $S_1$ and
	\item define homeomorphisms $h_2,h_3\in\Gamma$ such that $(h_2,S_2)$ and $(h_3,S_3)$ extend~$(h_1,S_1)$ but cannot be amalgamated into a pair $(h_4,S_4)$ from $(h_{-1},S_{-1})$.
\end{enumerate}

Pick $x\in E(\Sigma)$  an end accumulated by genus or a non-isolated planar end. Such an end exists because if all ends of $\Sigma$ are planar and isolated then $\Sigma$ is a finite-type surface. This implies, loosely speaking, that any neighbourhood of $x$ in $\Sigma$ is a subsurface of infinite type. Let $U$ be the connected component of $\Sigma\setminus S_0$ such that the open set $U^*$ (see Section~\ref{subsection:ends}) is a neighbourhood  of $x$ and $M_0\subset S_0$ be the multicurve in $\Sigma$ forming $\partial U$. Recall from above that $g\in\Gamma$ denotes a
homeomorphism defining the extension of $(h_0,S_0)$ to $(h_1,S_1)$. Define $U_1=g(U)$, $x_1=g^*x$ and $M_1=g(M_0)$. In particular, $(U_1)^*$ is a neighbourhood of $x_1$. As~$U_1$ is of infinite type, we can assure the existence of a \emph{closed} finite-type subsurface $V\subset U_1$ satifying the following properties:
\begin{enumerate}
    \item $M_{1}$ consists of boundary curves of $V$ and all boundary curves of $V$ that are not in~$M_{1}$ are essential curves of $\Sigma$. Moreover, every puncture of $V$ is a puncture of $\Sigma$.
    \item $V$ contains all the connected components of $S_{1} \cap U_{1}$ so that if $\delta$ is a boundary curve of $S_{1} \cap U_{1}$, then either:\footnote{Note that if a connected component of $U_1 \setminus S_1$ was a disc then we could add it to $S_1$ without affecting the rest of the arguments in this proof. Similarly, if a connected component of $S_1 \cap U_1$ was a disk, we could remove it from $S_1$.}
    \begin{itemize}
     \item $\delta$ is an element of $M_{1}$, or
     \item $\delta$ is an essential curve of $V$, or
     \item $\delta$ bounds a once-punctured disc in $V$.
    \end{itemize}
    \item $V\setminus S_1$ admits a non-trivial pants decomposition. That is, it has sufficiently large topological complexity.\footnote{Recall that the \emph{complexity} of a surface of genus $g$ with $n$ planar ends and $b$ boundary components is equal to $3g-3+n+b$.}
\end{enumerate}

Let $B$ be the multicurve formed by all curves in $\partial V$ which do not belong to $M_1$. We can think of $B$ as the ``exterior'' boundary of $V$ (relative to the end $x_1$). Now consider a pants decomposition $P$ of the interior of $V\setminus S_1$ and define the multicurve $M:=P\sqcup B$.
See Figure~\ref{Fig:locWAP} for an illustration of the subsurface $V$ and the multicurve $B$.
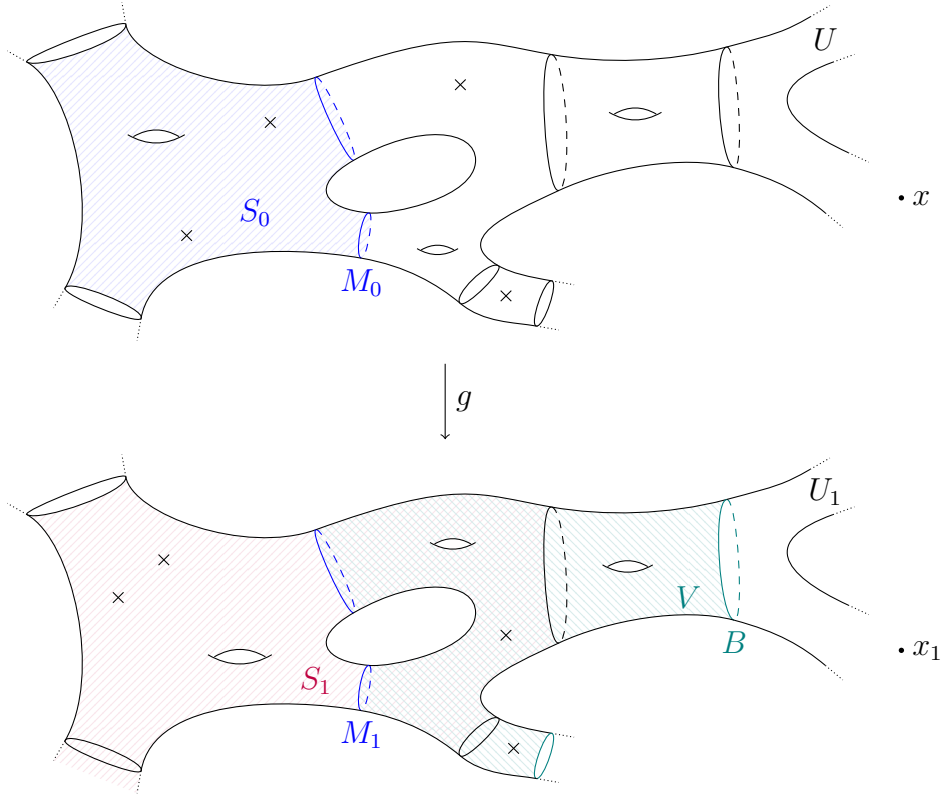
\begin{figure}[!ht]
\centering
\begin{tikzpicture}
\newcommand{\hole}[3]{
 \begin{scope}[xshift=#1,yshift=#2,scale=#3]
  \fill[color=white] (-1.2,0.05) to[bend left] (1.2,0.05) -- (1.5,0.15) to[bend left] (-1.5,0.15) -- (-1.2,0.05);
  \draw (-1.2,0.05) to[bend left] (1.2,0.05);
  \draw (-1.5,0.15) to[bend right] (1.5,0.15);
 \end{scope}
 }

\newcommand{\puncture}[2]{
 \begin{scope}[xshift=#1, yshift=#2]
  \draw (-0.07,0.07) -- (0.07,-0.07);
  \draw (-0.07,-0.07) -- (0.07,0.07);
 \end{scope}
}

 \path[pattern color=blue!10, pattern=north east lines] (0,0) .. controls +(60:1cm) and +(-30:0.7cm) .. (-0.5,3)
 .. controls +(-30:0.2cm) and +(-80:0.2cm) .. (0.8,3.5)
 .. controls +(-80:0.5cm) and +(-160:1cm) .. (3.3,2.8)
 .. controls +(20:0.1cm) and +(30:0.1cm) .. (3.8,1.7)
 .. controls +(-150:0.7cm) and +(175:0.6cm) .. (4,1)
 .. controls +(-5:0.07cm) and +(-10:0.07cm) .. (3.9,0.4)
 .. controls +(170:1cm) and +(80:1cm) .. (1,-0.4)
 .. controls +(60:0.15cm) and +(80:0.15cm) .. (0,0);
 \draw[color=blue] (2.5,1) node{$S_0$};

 \draw[densely dotted] (0,0) -- +(-120:0.3cm);
 \draw (0,0) .. controls +(60:1cm) and +(-30:0.7cm) .. (-0.5,3);
 \draw[densely dotted] (-0.5,3) -- +(150:0.3cm);

 \draw (-0.5,3) .. controls +(-30:0.2cm) and +(-80:0.2cm) .. (0.8,3.5);
 \draw (-0.5,3) .. controls +(150:0.1cm) and +(100:0.1cm) .. (0.8,3.5);

 \draw[densely dotted] (0.8,3.5) -- +(100:0.3cm);
 \draw (0.8,3.5) .. controls +(-80:0.5cm) and +(-160:1cm) .. (3.3,2.8);

 \draw[color=blue, dashed] (3.3,2.8) .. controls +(20:0.1cm) and +(30:0.1cm) .. (3.8,1.7);
 \draw[color=blue] (3.3,2.8) .. controls +(-160:0.1cm) and +(-150:0.1cm) .. (3.8,1.7);

 \draw (3.8,1.7) .. controls +(-150:0.7cm) and +(175:0.6cm) .. (4,1);

 \draw[color=blue, dashed] (4,1) .. controls +(-5:0.07cm) and +(-10:0.07cm) .. (3.9,0.4);
 \draw[color=blue] (4,1) .. controls +(175:0.1cm) and +(170:0.1cm) .. (3.9,0.4) node[below]{$M_0$};

 \draw (3.9,0.4) .. controls +(170:1cm) and +(80:1cm) .. (1,-0.4);
 \draw[densely dotted] (1,-0.4) -- +(-100:0.3cm);

 \draw (1,-0.4) .. controls +(60:0.15cm) and +(80:0.15cm) .. (0,0);
 \draw (1,-0.4) .. controls +(-120:0.1cm) and +(-100:0.1cm) .. (0,0);

 \hole{1.2cm}{2cm}{0.25}
 \puncture{1.6cm}{0.7cm}
 \puncture{2.7cm}{2.2cm}

 \draw (3.3,2.8) .. controls +(20:2cm) and +(170:1cm) .. (6.4,3.1);

 \draw (3.8,1.7) .. controls +(30:1cm) and +(90:0.4cm) .. (5.4,1.7)
 .. controls +(-90:0.5cm) and +(-5:0.3cm) .. (4,1);

 \draw (6.4,3.1) .. controls +(170:0.2cm) and +(-155:0.2cm) .. (6.5,1.3);
 \draw[dashed] (6.4,3.1) .. controls +(-10:0.2cm) and +(25:0.2cm) .. (6.5,1.3);

 \draw (6.5,1.3) .. controls +(-155:0.8cm) and +(150:0.7cm) .. (5.7,0.3);

 \draw (5.7,0.3) .. controls +(150:0.15cm) and +(140:0.15cm) .. (5.2,-0.2);
 \draw (5.7,0.3) .. controls +(-30:0.1cm) and +(-40:0.1cm) .. (5.2,-0.2);

 \draw (5.2,-0.2) .. controls +(140:0.4cm) and +(-10:0.6cm) .. (3.9,0.4);

 \hole{4.9cm}{0.5cm}{0.18}
 \puncture{5.2cm}{2.7cm}

 \draw (6.4,3.1) .. controls +(-10:0.7cm) and +(-165:0.8cm) .. (8.7,3.2);

 \draw (8.7,3.2) .. controls +(-165:0.2cm) and +(165:0.2cm) .. (8.8,1.6);
 \draw[dashed] (8.7,3.2) .. controls +(15:0.15cm) and +(-15:0.15cm) .. (8.8,1.6);

 \draw (6.5,1.3) .. controls +(25:0.7cm) and +(165:0.7cm) .. (8.8,1.6);

 \hole{7.5cm}{2.3cm}{0.23}

 \draw (5.7,0.3) .. controls +(-30:0.2cm) and +(170:0.3cm) .. (6.4,0.1);
 \draw[densely dotted] (6.4,0.1) -- +(-10:0.3cm);

 \draw (6.4,0.1) .. controls +(170:0.1cm) and +(170:0.1cm) .. (6.2,-0.5);
 \draw (6.4,0.1) .. controls +(-10:0.1cm) and +(-10:0.1cm) .. (6.2,-0.5);

 \draw[densely dotted] (6.2,-0.5) -- +(-10:0.3cm);
 \draw (6.2,-0.5) .. controls +(170:0.3cm) and +(-40:0.4cm) .. (5.2,-0.2);

 \puncture{5.8cm}{-0.1cm}

 \draw (8.7,3.2) .. controls +(15:0.7cm) and +(-150:0.4cm) .. (9.8,3.6);
 \draw[densely dotted] (9.8,3.6) -- +(30:0.3cm);

 \draw[densely dotted] (10.1,3) -- +(20:0.3cm);
 \draw (10.1,3) .. controls +(-160:1cm) and +(155:1cm) .. (10.3,1.8);
 \draw[densely dotted] (10.3,1.8) -- +(-25:0.3cm);

 \draw[densely dotted] (10,1) -- +(-40:0.3cm);
 \draw (10,1) .. controls +(140:0.4cm) and +(-15:0.6cm) .. (8.8,1.6);

 \draw (10,3.3) node{$U$};

 \fill (11,1.2) node[right]{$x$} circle (1pt);

 \draw[->] (5,-1) -- node[right]{$g$} (5,-2);

 \begin{scope}[yshift=-6cm]
 \path[pattern color=purple!10, pattern=north east lines] (0,0) .. controls +(60:1cm) and +(-30:0.7cm) .. (-0.5,3)
 .. controls +(-30:0.2cm) and +(-80:0.2cm) .. (0.8,3.5)
 .. controls +(-80:0.5cm) and +(-160:1cm) .. (3.3,2.8)
 .. controls +(20:2cm) and +(170:1cm) .. (6.4,3.1)
 .. controls +(170:0.2cm) and +(-155:0.2cm) .. (6.5,1.3)
 .. controls +(-155:0.8cm) and +(150:0.7cm) .. (5.7,0.3)
 .. controls +(150:0.15cm) and +(140:0.15cm) .. (5.2,-0.2)
 .. controls +(140:0.4cm) and +(-10:0.6cm) .. (3.9,0.4)
 .. controls +(170:1cm) and +(80:1cm) .. (1,-0.4)
 -- +(-100:0.3cm) -- (-120:0.3cm) -- (0,0);

 \path[pattern color=teal!15, pattern=north west lines] (3.9,0.4)
 .. controls +(170:0.1cm) and +(175:0.1cm)  .. (4,1)
 .. controls +(175:0.6cm) and +(-150:0.7cm) .. (3.8,1.7)
 .. controls +(-150:0.1cm) and +(-160:0.1cm) ..
 (3.3,2.8)
 .. controls +(20:2cm) and +(170:1cm) .. (6.4,3.1)
 .. controls +(-10:0.7cm) and +(-165:0.8cm) .. (8.7,3.2)
 .. controls +(-165:0.2cm) and +(165:0.2cm) .. (8.8,1.6)
 .. controls +(165:0.7cm) and +(25:0.7cm) .. (6.5,1.3)
 .. controls +(-155:0.8cm) and +(150:0.7cm) .. (5.7,0.3)
 .. controls +(-30:0.2cm) and +(170:0.3cm) .. (6.4,0.1)
 .. controls +(170:0.1cm) and +(170:0.1cm) .. (6.2,-0.5)
 .. controls +(170:0.3cm) and +(-40:0.4cm) .. (5.2,-0.2)
 .. controls +(140:0.4cm) and +(-10:0.6cm) .. (3.9,0.4);
 \draw[color=teal] (8.2,1.9) node{$V$};
 \draw[color=purple] (3.3,0.8) node{$S_1$};

 \fill[color=white] (3.8,1.7)
 .. controls +(30:1cm) and +(90:0.4cm) .. (5.4,1.7)
 .. controls +(-90:0.5cm) and +(-5:0.3cm) .. (4,1)
 .. controls +(175:0.6cm) and +(-150:0.7cm) .. (3.8,1.7);

 \draw[densely dotted] (0,0) -- +(-120:0.3cm);
 \draw (0,0) .. controls +(60:1cm) and +(-30:0.7cm) .. (-0.5,3);
 \draw[densely dotted] (-0.5,3) -- +(150:0.3cm);

 \draw (-0.5,3) .. controls +(-30:0.2cm) and +(-80:0.2cm) .. (0.8,3.5);
 \draw (-0.5,3) .. controls +(150:0.1cm) and +(100:0.1cm) .. (0.8,3.5);

 \draw[densely dotted] (0.8,3.5) -- +(100:0.3cm);
 \draw (0.8,3.5) .. controls +(-80:0.5cm) and +(-160:1cm) .. (3.3,2.8);

 \draw[color=blue, dashed] (3.3,2.8) .. controls +(20:0.1cm) and +(30:0.1cm) .. (3.8,1.7);
 \draw[color=blue] (3.3,2.8) .. controls +(-160:0.1cm) and +(-150:0.1cm) .. (3.8,1.7);

 \draw (3.8,1.7) .. controls +(-150:0.7cm) and +(175:0.6cm) .. (4,1);

 \draw[color=blue, dashed] (4,1) .. controls +(-5:0.07cm) and +(-10:0.07cm) .. (3.9,0.4);
 \draw[color=blue] (4,1) .. controls +(175:0.1cm) and +(170:0.1cm) .. (3.9,0.4) node[below]{$M_1$};

 \draw (3.9,0.4) .. controls +(170:1cm) and +(80:1cm) .. (1,-0.4);
 \draw[densely dotted] (1,-0.4) -- +(-100:0.3cm);

 \draw (1,-0.4) .. controls +(60:0.15cm) and +(80:0.15cm) .. (0,0);
 \draw (1,-0.4) .. controls +(-120:0.1cm) and +(-100:0.1cm) .. (0,0);

 \hole{2.3cm}{1.1cm}{0.28}
 \puncture{0.7cm}{1.9cm}
 \puncture{1.3cm}{2.4cm}

 \draw (3.3,2.8) .. controls +(20:2cm) and +(170:1cm) .. (6.4,3.1);

 \draw (3.8,1.7) .. controls +(30:1cm) and +(90:0.4cm) .. (5.4,1.7)
 .. controls +(-90:0.5cm) and +(-5:0.3cm) .. (4,1);

 \draw (6.4,3.1) .. controls +(170:0.2cm) and +(-155:0.2cm) .. (6.5,1.3);
 \draw[dashed] (6.4,3.1) .. controls +(-10:0.2cm) and +(25:0.2cm) .. (6.5,1.3);

 \draw (6.5,1.3) .. controls +(-155:0.8cm) and +(150:0.7cm) .. (5.7,0.3);

 \draw (5.7,0.3) .. controls +(150:0.15cm) and +(140:0.15cm) .. (5.2,-0.2);
 \draw (5.7,0.3) .. controls +(-30:0.1cm) and +(-40:0.1cm) .. (5.2,-0.2);

 \draw (5.2,-0.2) .. controls +(140:0.4cm) and +(-10:0.6cm) .. (3.9,0.4);

 \hole{5.1cm}{2.6cm}{0.2}
 \puncture{5.8cm}{1.4cm}

 \draw (6.4,3.1) .. controls +(-10:0.7cm) and +(-165:0.8cm) .. (8.7,3.2);

 \draw[color=teal] (8.7,3.2) .. controls +(-165:0.2cm) and +(165:0.2cm) .. (8.8,1.6);
 \draw[color=teal,dashed] (8.7,3.2) .. controls +(15:0.15cm) and +(-15:0.15cm) .. (8.8,1.6) node[below]{$B$};

 \draw (6.5,1.3) .. controls +(25:0.7cm) and +(165:0.7cm) .. (8.8,1.6);

 \hole{7.4cm}{2.3cm}{0.22}

 \draw (5.7,0.3) .. controls +(-30:0.2cm) and +(170:0.3cm) .. (6.4,0.1);
 \draw[densely dotted] (6.4,0.1) -- +(-10:0.3cm);

 \draw[color=teal] (6.4,0.1) .. controls +(170:0.1cm) and +(170:0.1cm) .. (6.2,-0.5);
 \draw[color=teal] (6.4,0.1) .. controls +(-10:0.1cm) and +(-10:0.1cm) .. (6.2,-0.5);

 \draw[densely dotted] (6.2,-0.5) -- +(-10:0.3cm);
 \draw (6.2,-0.5) .. controls +(170:0.3cm) and +(-40:0.4cm) .. (5.2,-0.2);

 \puncture{5.9cm}{-0.1cm}

 \draw (8.7,3.2) .. controls +(15:0.7cm) and +(-150:0.4cm) .. (9.8,3.6);
 \draw[densely dotted] (9.8,3.6) -- +(30:0.3cm);

 \draw[densely dotted] (10.1,3) -- +(20:0.3cm);
 \draw (10.1,3) .. controls +(-160:1cm) and +(155:1cm) .. (10.3,1.8);
 \draw[densely dotted] (10.3,1.8) -- +(-25:0.3cm);

 \draw[densely dotted] (10,1) -- +(-40:0.3cm);
 \draw (10,1) .. controls +(140:0.4cm) and +(-15:0.6cm) .. (8.8,1.6);

 \draw (10,3.3) node{$U_1$};

 \fill (11,1.2) node[right]{$x_1$} circle (1pt);
 \end{scope}

\end{tikzpicture}

    \caption{The subsurface $V$ and the multicurve $B$.}
	\label{Fig:locWAP}
\end{figure}

Define the finite-type subsurfaces $S_{2} = S_{3} = S_{1} \cup V$, and the homeomorphisms:
$$
h_{2} = \tau_{h_{1}(M)}^{2} \circ h_{1} = h_{1} \circ \tau_{M}^{2} \circ h_{1}^{-1} \circ h_{1} = h_{1} \circ \tau_{M}^{2}
$$
and
$$
h_{3} = \tau_{h_{1}(M)}^{3} \circ h_{1} = h_{1} \circ \tau_{M}^{3} \circ h_{1}^{-1} \circ h_{1} = h_{1} \circ \tau_{M}^{3}.
$$
Note that the identity on $\Sigma$ provides an extension of $(h_{1},S_{1})$ to the pairs $(h_{2},S_{2})$ and $(h_{3},S_{3})$. We claim that there does not exist a pair $(h_{4}, S_{4})$ that amalgamates the pairs $(h_{2},S_{2})$ and $(h_{3},S_{3})$ from $(h_{-1}, S_{-1})$.

If it existed then there would be copies $W_{2}$ and $W_{3}$ of $S_{2}$ and $S_{3}$, respectively, contained in~$S_{4}$. This is because $(h_{4},S_{4})$ has to extend both $(h_{2},S_{2})$ and $(h_{3},S_{3})$. Moreover, since~$(h_{4},S_{4})$ amalgamates these two pairs from $(h_{-1},S_{-1})$, we have that $W_{2}$ and $W_{3}$ (along with the extensions of the homeomorphisms $h_{2}$ and $h_{3}$) have to coincide in the copy of $S_{-1}$ inside them. Let $W$ denote this copy of $S_{-1}$ that is inside both $W_{2}$ and $W_{3}$ and let:
\begin{itemize}
	\item $V_{2}$ and $V_{3}$ be the copies of $V$ inside $W_{2}$ and $W_{3}$ respectively, (note that $V_2 \subset W_2 \setminus W$ and $V_3 \subset W_3 \setminus W$)
	\item $M_2=P_2\cup B_2\subset V_2$ and $M_3=P_3\cup B_3\subset V_3$ be the copies in $S_4$ of the multicurve $M=P\cup B$ defined above, and
    \item given that $S_{0}$ is the closure of $S_{-1}$, we have that $W_{2}$ and $W_{3}$ have to coincide in the copy of the multicurve $M_{1}$. Let $M_{1}^{\prime}$ be this copy inside $S_4$.

\end{itemize}
Given that $V_2$ and $V_3$ are homeomorphic and coincide in $M_1'$, there exist $\delta_2\in M_2$ and~$\delta_3\in M_3$ such that either (i) $\delta_2=\delta_3$ (mod isotopy) or (ii) $i(\delta_2,\delta_3)\neq 0$.  In both cases, we derive a contradiction. Indeed, suppose that $(h_2,S_2)$ and $(h_3,S_3)$ are extended to $(h_4,S_4)$ via $f_2$ and $g_2$ respectively.
If (i) happens then we would have that $h_{4}$ restricted to $M_{2}$ has to be conjugated via $f_2$ to $\tau^2_{h_1(M)}\circ h_1$. On the other hand, we would have that $h_{4}$ restricted to~$M_{3}$ has to be conjugated via $g_2$ to $\tau^3_{h_1(M)}\circ h_1$. This is absurd because different powers of a Dehn twist cannot be conjugated. Now suppose that (ii) happens. Then on one hand, we have that~$h_{4}(\delta_{2})$ is a curve in $h_{4}(M_{2})$, thus $i(h_{4}(\delta_{2}), \gamma) = 0$ for every $\gamma \in h_4(M_2)$. On the other hand, $h_{4}$ extends $h_{3} = h_{1} \circ \tau_{M}^{3}$. Therefore inside $V_{3}$, $h_{4}$ first acts as $\tau_{M_{3}}^{3}$ and then maps to $h_{4}(V_{3})$. As $\delta_2$ intersects a curve in $M_3$, we have that $\tau_{M_{3}}^{3}(\delta_{2})$ intersects $\delta_{2}$ non-trivially. In particular, $\tau_{M_{3}}^{3}(\delta_{2})$ intersects a curve in $M_2$ nontrivially. Hence, $h_{4}(\delta_{2})$ would intersect~$h_4(M_2)$ non-trivially, which is a contradiction.

Therefore, the pairs $(h_{2},S_{2})$ and $(h_{3},S_{3})$ cannot be amalgamated from $(h_{-1},S_{-1})$.
\qed

%% file: dense.tex
\section{Dense conjugacy classes}
    \label{SEC:dense-conjugacy-classes}

In this section, we give a characterization of big mapping class groups having dense conjugacy classes. Our characterization uses the (partial) order on the space of ends~$E(\Sigma)$ defined by Mann and Rafi, and the notions of self-similar spaces of ends and non-displaceable subsurfaces which we recall here. We refer the reader to~\cite{mann-rafi} for details.

\begin{definition}
Let $\Sigma$ be an infinite-type surface and $x,y\in E(\Sigma)$. Let $\preccurlyeq$ be the binary relation on $E(\Sigma)$ where $y\preccurlyeq x$ if, for every neighbourhood $U$ of $x$, there exists a neighbourhood~$V$ of $y$ and $f\in\MCG(\Sigma)$ such that $f(V)\subset U$.
\end{definition}

A key point in the above definition is that $x$ does not have to be contained in $f(V)$. The binary relation defined above can be easily promoted to a partial order by considering an adequate quotient of the space of ends. Indeed, if for $x,y\in E(\Sigma)$ we say that $x$ and $y$ are of the same type if $y\preccurlyeq x$ and $x\preccurlyeq y$, then we obtain an equivalence relation on $E(\Sigma)$. Define $x\prec y$ if $x\preccurlyeq y$ but $x$ and $y$ are not of the same type. Then $\prec$ defines a partial order on the set of equivalence classes of ends.

\begin{proposition}~\cite[Proposition 4.7]{mann-rafi}
    \label{prop:topology_maximal_elements}
The partial order $\prec$ has maximal elements. Moreover, for every maximal element $x$, its equivalence class is either finite or a Cantor set.
\end{proposition}

Following $[ibid.]$, we denote by $E(x)$ the equivalence class of $x\in E(\Sigma)$ and $\mathcal{M}=\mathcal{M}(\Sigma)$ the set of maximal ends for $\prec$.

As an example, consider the flute surface $\Sigma=\mathbb{R}^2 \setminus \mathbb{Z}^2$. The space of ends of $\Sigma$ is homeomorphic to $\{0\} \cup \{\frac{1}{n}:n\in\mathbb{N}\}\subset\R$. Isolated ends are in bijection with $\{\frac{1}{n}:n\in\mathbb{N}\}$ and the point at infinity corresponds to $0$ and is the unique maximal end of $\Sigma$. It is rather easy to manipulate this example to create examples of infinite-type surface with two or more different classes of maximal ends.

\begin{definition}
    \label{def:non-displaceable}
    Let $\Sigma$ be an infinite-type surface.
\begin{itemize}
    \item Let $S\subset\Sigma$ be a connected subsurface\,\footnote{We do not assume here that $S$ is of finite type.}. We say that $S$ is \emph{non-displaceable} if for every $f\in\Homeo(\Sigma)$, we have that $f(S)\cap S\neq\emptyset$.
    \item The space of ends $(E(\Sigma),E_\infty(\Sigma))$ is called \emph{self-similar} if for any decomposition of $E(\Sigma)$ into pairwise disjoint clopen sets
    \[
    E(\Sigma)=E_1\sqcup E_2 \sqcup\ldots\sqcup E_n,
    \]
    there exists a clopen set $D\subset E_i$, for some $i\in\{1,\ldots,n\}$, such that $(D,D\cap E_\infty(\Sigma))$ is homeomorphic to $(E(\Sigma),E_\infty(\Sigma))$.
\end{itemize}
\end{definition}

Any infinite-type surface $\Sigma$ of positive, finite genus has an infinite space of ends and hence infinitely many different (\emph{i.e.}\ non-homeomorphic) non-displaceable finite-type subsurfaces. Indeed, any subsurface $S\subset\Sigma$ whose genus is equal to the genus of $\Sigma$ is non-displaceable. For an example of a surface without finite-type non-displaceable subsurfaces, consider again $\Sigma=\mathbb{R}^2\setminus\mathbb{Z}^2$. There is a natural action by translations of $\mathbb{Z}^2$ which, when extended to~$E(\Sigma)$, has no fixed points except for the unique maximal end and leaves the set of punctures invariant. As we see later, this surface has a very nice property: for any finite-type subsurface~$S\subset\Sigma$ and any compact subset $K\subset\mathbb{R}^2$ there exists $(n,m)\in\mathbb{Z}^2$ for which the corresponding translation $T_{(n,m)}\in\MCG(\Sigma)$ satisfies $T_{(n,m)}(S)\cap(K\cap\Sigma)=\emptyset$. In other words, not only that there are no non-displaceable subsurfaces of finite type but any finite-type subsurface can be displaced as far as desired.
    \begin{remark}
        \label{Rk:Interesting-remark}
    Every infinite-type surface $\Sigma$ has a (not necessarily connected\footnote{The definition of non-displaceable subsurface above can be extended to disconnected surfaces as follows: $S\subset\Sigma$ is non-displaceable if for every $f\in\Homeo(\Sigma)$ and every connected component $S_1$ of $S$, there exists a connected component $S_2$ of $S$ such that $f(S_1)\cap S_2\neq \emptyset$.}) non-dis\-place\-able infinite-type subsurface $\Sigma'$. This follows from the fact that we can always find a multicurve $M\subset\Sigma$ such that $\Sigma\setminus M=S\sqcup\Sigma'$ with $S$ a finite-type subsurface. Because of this rather trivial fact, we stress most of the time that non-displaceable subsurfaces under consideration are of finite type.
    \end{remark}

The main result of this section is the following:

\begin{theorem}[Characterization of big mapping class groups with dense conjugacy classes]
	\label{THM:Charact-Dense-Conjugacy-Classes}
 Let $\Sigma$ be an infinite-type surface. Then the following are equivalent:
 \begin{enumerate}
  \item $\MCG(\Sigma)$ has a dense conjugacy class.
  \item The surface $\Sigma$ has no non-displaceable finite-type subsurfaces and there exists a unique\,\footnote{That is, $\mathcal{M}$ has only one equivalence class and this class is a singleton formed by the point~$x_\infty$.} maximal end $x_\infty\in E(\Sigma)$.
  \item The surface $\Sigma$ has no non-displaceable finite-type subsurfaces, there exists a unique maximal end $x_\infty\in E(\Sigma)$ and every isomorphism $\phi:A\to B$ between finite substructures of the full curve graph $C_{\MCG(\Sigma)}(\Sigma)$ is induced by a homeomorphism $h\in Homeo^+(\Sigma)$ whose support is contained in $\Sigma\setminus U$, for some subsurface $U\subset \Sigma$ such that $U^\ast$ is a neighborhood of $x_\infty$.
 \end{enumerate}
\end{theorem}

\textbf{Structure of this section.} We first present in Section~\ref{SSEC:Needed-Lemmas} technical lemmas needed for the proof of Theorem~\ref{THM:Charact-Dense-Conjugacy-Classes}. The usefulness of these lemmas goes beyond the aforementioned proof.
We discuss in Section~\ref{SSEC:An-Illustrative-Example} an illustrative example of a surface which satisfies the hypothesis of  Theorem~\ref{THM:Charact-Dense-Conjugacy-Classes}. We believe the discussion of this example makes the proof of  Theorem~\ref{THM:Charact-Dense-Conjugacy-Classes} more transparent. We then proceed with the proof of this theorem in Section~\ref{SSEC:Proof-Theorem-Dense-CC}.

\subsection{Useful lemmas}
    \label{SSEC:Needed-Lemmas}

Lemma~\ref{Lemma:Aut-Gamma-AreIsomorphic} gives two different points of view for any closed subgroup of the mapping class group of an infinite-type surface: as group of homeomorphisms or as group of automorphisms of an ultrahomogeneous structure. By Theorem~\ref{Theorem:Criteria}, given a closed subgroup $\Gamma<\MCG(\Sigma)$, the existence of a dense conjugacy class in $\Gamma$ is determined by whether the class $\mathcal{F}_p$ (w.r.t.~$\mathcal{F}_{C_{\Gamma}(\Sigma)}$, see Definition~\ref{DEF:FullCurveComplex}) satisfies or not the joint embedding property \textbf{JEP}, see Remark~\ref{rk:JEP_in_Fp}. We begin this section by explaining how to translate \textbf{JEP} in the context of $\mathcal{F}_p$ to the language of finite-type subsurfaces and homeomorphisms.

\begin{lemma}
  \label{lemma:GEP}
Let $\Sigma$ be an infinite-type surface and $\Gamma$ be a closed subgroup of $\MCG(\Sigma)$. Then~$\Gamma$ has a dense conjugacy class if and only if the following geometric version of {\bf JEP} holds:

\begin{itemize}
 \item[({\bf GEP})] given two finite-type subsurfaces $S, S'$ of $\Sigma$ and $h,h'\in \Gamma$, there are $g,g',H\in \Gamma$ such that
  \begin{enumerate}
   \item $g\circ h\restriction S$ is isotopic to $H\circ g\restriction S$, and
   \item $g'\circ h' \restriction S'$ is isotopic to $H\circ g'\restriction S'$.
  \end{enumerate}
\end{itemize}
Here {\bf GEP} stands for \emph{geometric embedding property}.
\end{lemma}

\begin{proof}
Henceforth, $\mathcal{F}_p$ is defined w.r.t.\ the Fra\"iss\'e class $\mathcal{F}_{C_{\Gamma}(\Sigma)}$, where $C_\Gamma(\Sigma)$ is the full curve graph of $\Sigma$ w.r.t.\ the closed subgroup $\Gamma$, see Definition~\ref{DEF:FullCurveComplex}.

Let us first prove ${\bf GEP} \Rightarrow {\bf JEP}$  for $\mathcal F_p$. Fix $\langle A, \varphi: B\to C\rangle, \langle A', \varphi': B'\to C'\rangle \in \mathcal F_p$. We assume that  $A=B\cup C$ and $A'=B'\cup C'$ (see Remark~\ref{rmk:A=BUC}). Let $S$ and $S'$ be small tubular neighbourhoods of $B$ and $B'$ respectively (thought of as sets of fixed essential simple closed curves in $\Sigma$). From Lemma~\ref{Lemma:Aut-Gamma-AreIsomorphic}, we deduce that there exist $h,h'\in \Gamma$ such that $h\restriction S$ and $h'\restriction S'$ induce $\varphi$ and $\varphi'$ respectively. Now apply {\bf GEP} to get $g,g', H\in \Gamma$ such that $g\circ h\restriction S$ is isotopic to $H\circ g\restriction S$, and $g'\circ h'\restriction S'$ is isotopic to $H\circ g'\restriction S'$.
Let $A_0$ be the $g$-image of the curves in~$A$ and $A_0'$ the $g'$-image of the curves in $A'$, and define $B_0, B'_0, C_0,C'_0$ similarly for $B, B', C,C'$. Denote by $H_*$ the automorphism of $C_\Gamma(\Sigma)$ induced by $H$. Then $\langle A_0\cup A_0', H_*\restriction B_0\cup B_0' \rangle\in\mathcal{F}_p$ and witnesses {\bf JEP} for $\langle A, \varphi: B\to C\rangle$ and $\langle A', \varphi': B'\to C'\rangle$.

We now prove ${\bf JEP} \Rightarrow {\bf GEP}$. Given that the mapping class group of the disc is trivial, we fix
two non-simply connected finite-type subsurfaces $S,S'$ of~$\Sigma$ and $h,h' \in \Gamma$. Let~$B$ and~$B'$ be Alexander systems\,\footnote{These are sets of curves with the property that any homeomorphism that fixes them (modulo isotopy) has to be homotopic to the identity. See Section 2.3 in~\cite{FarbMargalit} for more details.} for~$S$ and $S'$ respectively. Define $C=h_*(B)$, $C'=h_*(B')$ and $\varphi:=h_* \restriction B$, $\varphi':=h_ * \restriction B'$. Let $A=B\cup C$ and $A'=B'\cup C'$. By hypothesis, there exists $\langle D,\Psi:E\to F\rangle$ witnessing {\bf JEP} for $\langle A,\varphi:B\to C \rangle$ and $\langle A',\varphi:B'\to C' \rangle$, that is, there exist embeddings $j:A\to D$ and $j':A'\to D$ such that $j\circ \varphi\subset \Psi\circ j$ and $j'\circ\varphi'\subset \Psi\circ j'$. Since $j,j'$ are embeddings, there exist $g,g'\in \Gamma$ such that $g_* \restriction A =j$ and $g'_* \restriction A' =j'$. Also, there exists $H\in \Gamma$ such that $H_* \restriction E=\Psi$. Given that $(H\circ g)_*$ agrees with $(g \circ h)_\ast$ on an Alexander system in~$S$, we have that $g\circ h$ is isotopic to $H\circ g$ on $S$. By the same arguments, we have that  $g'\circ h'$ is isotopic to $H\circ g'$ on $S'$.
\end{proof}

In what follows, we present a simple trick to disprove the existence of a dense conjugacy class for certain closed subgroups of the extended mapping class group. These are closed subgroups $\Gamma$ that contain, for every essential curve $\alpha\subset\Sigma$, a non-trivial power of the Dehn twist $\tau_\alpha$. The rough idea is that Dehn twists about curves that cannot be separated by elements of $\Gamma$ cannot be jointly embedded, in the sense of \textbf{GEP} (see Lemma~\ref{lemma:GEP}).

\begin{definition}
    \label{DEF:Cannot-separate}
Let $\Sigma$ be an infinite-type surface and $\Gamma<\MCG^*(\Sigma)$ a closed subgroup. We say that two (not necessarily distinct) multicurves $A = \{\alpha_{1}, \ldots, \alpha_{n}\}$ and $B = \{\beta_{1}, \ldots, \beta_{m}\}$ cannot be separated by $\Gamma$ if for all $f,g\in\Gamma$ there exists a pair $(\alpha_i,\beta_j)\in A\times B$ such that either $f(\alpha_i)=g(\beta_j)$ or $i(f(\alpha_i),g(\beta_j))\neq 0$.
\end{definition}

\begin{lemma}[Dehn twist trick]
        \label{DehnTwistTrick}
    Let $\Sigma$ be an infinite-type surface, and $\Gamma<\MCG(\Sigma)$ be a closed subgroup that contains for every essential curve $\alpha\subset\Sigma$ a  non-trivial power of the Dehn twist $\tau_\alpha$. Let $A = \{\alpha_{1}, \ldots, \alpha_{n}\}$ and $B = \{\beta_{1}, \ldots, \beta_{m}\}$ be two multicurves that cannot be separated by $\Gamma$. Then the multitwists $\tau_A^2$ and $\tau_B^3$ (see Definition~\ref{Def:minpowerandmultitwist}) cannot be jointly embedded, i.e.\ do not fulfill {\bf GEP}. In particular, $\Gamma$ does not have a dense conjugacy class.
\end{lemma}

\begin{proof}
To prove that $\Gamma$ does not have a dense conjugacy class, we use $A$ and $B$ to define homeomorphisms and compact subsurfaces that do not satisfy {\bf GEP} (see Lemma~\ref{lemma:GEP}). Modulo passing to an appropriate power, we can assume that $\tau_{\alpha_i},\tau_{\beta_j}\in\Gamma$ for every  $(\alpha_i,\beta_j)\in A\times B$. Let $\tau_{A} = \tau_{\alpha_{1}} \circ \cdots \circ \tau_{\alpha_{n}}, \tau_{B} = \tau_{\beta_{1}} \circ \cdots \circ \tau_{\beta_{m}} \in \Gamma$, $S = N(A)$ and $S' = N(B)$, where~$N(A)$ and~$N(B)$ denote closed regular neighbourhoods of $A$ and $B$ containing the support of $\tau_A$ and $\tau_B$ respectively. We proceed by contradiction.

Suppose there exist $g,g', H \in \Gamma$ such that $H \restriction g(S) = g \circ \tau_{A}^{2} \circ g^{-1} \restriction g(S)$ and $H \restriction g'(S') = g' \circ \tau_{B}^{3} \circ (g')^{-1} \restriction g'(S')$. We have that:
\begin{equation}
    \label{EQ:Dehn-Twist-trick}
\begin{array}{ccccc}
H \restriction g(S) & = & g \circ \tau_{A}^{2} \circ g^{-1} \restriction g(S) & = & \tau_{g(A)}^{2} \restriction g(S)\\
\\
H \restriction g'(S') & = & g' \circ \tau_{B}^{3} \circ (g')^{-1} \restriction g'(S') & = & \tau_{g'(B)}^{3} \restriction g'(S')
\end{array}
\end{equation}

By hypothesis, there exists a pair $(\alpha_i,\beta_j)\in A\times B$ such that either (i) $g(\alpha_i)=g'(\beta_j)$ or (ii) $i(g(\alpha_i),g'(\beta_j))\neq 0$. In both cases, (\ref{EQ:Dehn-Twist-trick}) leds to a contradiction. Indeed, (i) would imply that $H$ restricted to a neighborhood of $\gamma=g(\alpha_i)=g'(\beta_j)$ is isotopic to $\tau^2_\gamma$ and~$\tau^3_\gamma$, which is absurd.
Also, as $H$ restricted to $g(S)$ is the multitwist $\tau^2_{g(A)}$ by (\ref{EQ:Dehn-Twist-trick}), we have that $H(g(\alpha_i))=g(\alpha_i)$ for every $\alpha_i\in A$. But this is impossible if (ii) happens because, also by~(\ref{EQ:Dehn-Twist-trick}), $H$ restricted to $g'(S')$ is the multitwist $\tau^3_{g'(B)}$.
\end{proof}

Applying the ideas in the proof of the Dehn twist trick, we get the following:

\begin{theorem}
    \label{Theo:NotJEP}
Let $\Sigma$ be an infinite-type surface with a non-displaceable subsurface $S$ (not necessarily of finite type). Consider the Dehn twist $h=\tau_{\partial S}$.
Let $\Gamma$ be a closed subgroup of $\MCG(\Sigma)$ that contains, for every essential curve $\alpha\subset\Sigma$, a non-trivial power of the Dehn twist~$\tau_\alpha$. Then there are no $g, g^\prime, H \in \Gamma$ such that $H\circ g \restriction S=g \circ h \restriction S$ and $H\circ g^\prime \restriction S=g^\prime \restriction S$.
\end{theorem}

\begin{proof}

By possibly enlarging $S$, we can assume that the boundary of $S$ consists of a disjoint union of non-isotopic essential simple closed curves in $\Sigma$. Moreover, we can assume that the topological complexity of $S$ is as large as needed and that the support of $h$ is actually contained in $S$.

We proceed by contradiction. Suppose that there exist $g, g^\prime, H \in \Gamma$ such that $H\circ g \restriction S=g\circ h \restriction S$ and $H\circ g^\prime \restriction S=g^\prime \restriction S$. Then we have
\begin{eqnarray}\label{Eq1}
H \restriction g(S) = g\circ h\circ g^{-1} \restriction g(S) & \mbox{ and } & H \restriction g^\prime(S) = Id \restriction g^\prime(S).
\end{eqnarray}

We divide the proof into two cases:

\textbf{Case 1}. There exist $\gamma \subseteq g(\partial S)$ and $\delta \subseteq g^\prime(\partial S)$ such that ${\rm{i}}(\gamma,\delta) \neq 0$. In other words, elements of $\partial S$ cannot be separated by $g$ and $g'$.
Then we get a contradiction using the ideas in the proof of the Dehn twist trick. Indeed, Equation~(\ref{Eq1}) tells us that $H \restriction g(S)$ is the multitwist around the curves in $\partial g(S)$ and on the other hand $H \restriction g^\prime(S) = Id \restriction g^\prime(S)$. But this is impossible, given that $\gamma$ and $\delta$ intersect.

\textbf{Case 2}. For all $\gamma \subseteq g(\partial  S)$ and $\delta \subseteq g'(\partial S)$, ${\rm{i}}(\gamma,\delta) = 0$.
If each boundary curve in $g(\partial S)$ is isotopic to a boundary curve in $g'(\partial S)$ then we obtain a contradiction from~(\ref{Eq1}) because~$H$ cannot simultaneously be isotopic to a Dehn multitwist and to the identity on $g(\partial S)$.
As~$S$ is non-displaceable, there exists $\gamma \subseteq g(\partial  S)$ which intersects $g^\prime(S)$ and is not isotopic to a curve in $g'(\partial S)$, hence has to be an essential curve of $g^\prime(S)$. Let $Q$ be a four-punctured subsurface of $g^\prime(S)$ which contains $\gamma$ in its interior as an essential curve. This is possible because, as remarked above, $S$ has enough topological complexity. Let~$\alpha$ be an essential curve of $Q$ with ${\rm{i}}(\alpha,\gamma)=2$, see Figure \ref{Fig:Case2}. Since $H \restriction g^\prime(S) = Id \restriction g^\prime(S)$ then $H(\alpha)=\alpha$. On the other hand, by Equation (\ref{Eq1}), $H$ restricted to a neighbourhood of $\gamma$ is a Dehn twist, hence $\alpha \neq H(\alpha)$ which is a contradiction.
\begin{figure}[!ht]
\centering
 \begin{tikzpicture}[scale=0.8]
  \draw (-4,3.5) .. controls +(-70:0.2cm) and +(-175:2cm) .. (-0.5,2.5)
  .. controls +(5:2cm) and +(-120:1.3cm) .. (3.5,4);
  \draw (-3.5,-1.5) .. controls +(50:0.2cm) and +(-170:2cm) .. node[below]{$Q$} (0,0)
  .. controls +(10:2cm) and +(140:0.5cm) .. (4,-1);
  \draw (-5,2.5) .. controls +(-35:0.5cm) and +(95:1cm) .. (-3.5,1)
  .. controls +(-85:0.8cm) and +(45:0.5cm) .. (-5,-0.8);
  \draw (4.5,3) .. controls +(-150:0.5cm) and +(85:1cm) .. (3,1.5)
  .. controls +(-95:0.8cm) and +(140:0.5cm) .. (4.5,-0.2);

  \draw[thick] (0,0) .. controls +(10:0.2cm) and +(5:0.2cm) .. node[pos=0.7, right]{$\gamma$} (-0.5,2.5);
  \draw[thick, dashed] (-0.5,2.5) .. controls +(-175:0.2cm) and +(-170:0.2cm) .. (0,0);

  \draw[thick, color=red] (-3.5,1) .. controls +(-85:0.3cm) and +(-95:0.3cm) .. node[pos=0.8, below]{$\alpha$} (3,1.5);
  \draw[thick, dashed, color=red] (3,1.5) .. controls +(85:0.3cm) and +(95:0.3cm) .. (-3.5,1);
 \end{tikzpicture}
	\caption{ }
	\label{Fig:Case2}
\end{figure}
\end{proof}

\begin{corollary}\label{Coro:MapNotJEP}
Let $\Sigma$ be an infinite-type surface with a finite-type non-displaceable subsurface $S$, and $\Gamma$ be a closed subgroup of $\MCG(\Sigma)$ that contains for every essential curve $\alpha\subset\Sigma$, a non-trivial power of the Dehn twist $\tau_\alpha$. Then $\Gamma$ does not have a dense conjugacy class.
\end{corollary}

\begin{proof}
Consider $S=S'$, $M=\partial S$, $h=\tau_{M}$ and $h^\prime=Id_\Sigma$. Then by Theorem~\ref{Theo:NotJEP}, the homeomorphisms $h$ and $h'$ do not satisfy \textbf{GEP} as stated in Lemma~\ref{lemma:GEP}.
\end{proof}

\subsection{An illustrative example}
    \label{SSEC:An-Illustrative-Example}
    Consider the flute surface $\Sigma = \R^2 \backslash \Z^2$, which has genus zero, no non-displaceable finite-type subsurfaces and its space of ends is homeomorphic to~$\omega + 1$. In what follows, we explain why $\MCG(\R^2 \backslash \Z^2)$ satifies \textbf{GEP} and thus has at least one dense conjugacy class. The proof of Theorem~\ref{THM:Charact-Dense-Conjugacy-Classes} in the next section is inspired by this example.

In the context of this surface, we say that a subset $X \subset \Sigma$ is \emph{bounded away from the maximal end} if there is a compact set $K \subset \R^2$ such that $X \subset K\cap\Sigma$.

Let $f, h \in \MCG(\Sigma)$ and $S, S' \subset \Sigma$ be finite-type subsurfaces. Suppose first that we are lucky and have that $S\cap {\supp}(h)=h(S')\cap {\supp}(f)=\emptyset$. Then $H = f \circ h$ satisfies
	\begin{equation}
        \label{EQ:Disjoint-support-case}
    	H \restriction S = f \restriction S, \qquad H \restriction S' = h \restriction S'
    \end{equation}
and the pairs $(f,S)$ and $(h,S')$ can be jointly embedded, in the sense of \textbf{GEP}, see Lemma~\ref{lemma:GEP}. In the context of $\Sigma = \R^2 \backslash \Z^2$, we can deal with the general case because we have the following two properties:
\begin{enumerate}
	\item \label{movable_support} If $X \subset \Sigma$ is a subsurface bounded away from the maximal end, then there exists an element $g \in \MCG(\Sigma)$ whose support is bounded away from the maximal end such that $X \cap g(X) = \emptyset$.

	\item \label{localization_of_homeos} If $S\subset \Sigma$ is a finite-type subsurface and $f \in \MCG(\Sigma)$, then there exists $f_0 \in \MCG(\Sigma)$ such that $f_0$ has support bounded away from the maximal end and
		\[	f \restriction S = f_0 \restriction S.	\]
\end{enumerate}

Using (2) above, we can change every $(f,S)$ and $(h,S')$ for pairs $(f_0,S)$ and $(h_0,S')$ such that $f \restriction S = f_0 \restriction S$, $h \restriction S' = h_0 \restriction S'$ and both $f_0$ and $h_0$ have supports which are bounded away from the maximal end.
We now consider a subsurface $X\subset\Sigma$ bounded away from the maximal end and containing ${\supp}(f_0)$, ${\supp}(h_0)$, $h_0(S')$ and $S$. By (1), there exists $g\in \MCG(\Sigma)$ with support bounded away from the maximal end and $X \cap g(X) = \emptyset$. In particular, ${\supp(f_0)}\cap (g\circ h_0(S'))=\emptyset$, hence $f_0\restriction g\circ h_0(S') = Id \restriction g\circ h_0(S')$, and $(g \circ h_0 \circ g^{-1})\restriction S = Id\restriction S$.

Define $H:=f_0 \circ (g \circ h_0 \circ g^{-1})$, then:
	\begin{equation}
        \label{EQ:Illustrative-GEP}
    	H \circ g \restriction S' = f_0\circ g \circ h_0 \restriction S'= g \circ h_0 \restriction S' = g \circ h \restriction S',
    	\qquad H \restriction S = f_0 \restriction S = f \restriction S.
    \end{equation}

This means that $(f,S)$ and $(h,S')$ satisfy \textbf{GEP}, thus by Lemma~\ref{lemma:GEP}, we have that $\MCG(\Sigma)$ satisfies \textbf{JEP} and hence it has a dense conjugacy class.

We now provide a proof of the properties listed above.

\begin{proof}[Proof of (\ref{movable_support})]
\begin{figure}[b]
\centering
 \begin{tikzpicture}[scale=0.8]
  \draw (0,0) circle (2.3cm);
  \draw (50:2.3) node[above right]{$c$};
  \draw[color=green!50!black] (-1,0) node{$S$} circle (0.6cm);
  \draw[color=green!50!black] (-1.4,0.3) node[above left]{$\alpha$};
  \draw[color=blue] (1,0) node{$S'$} circle (0.6cm);
  \draw[color=blue] (1.4,0.3) node[above right]{$\alpha'$};

  \draw (3.5,0) node{$\leadsto$};

  \begin{scope}[xshift=7cm]
   \draw (2,2) .. controls +(50:0.4cm) and +(130:0.4cm) .. node[pos=0.8, above]{$c$} (-2,2)
   .. controls +(-50:0.4cm) and +(-130:0.4cm) .. (2,2);
   \draw (-2,2) .. controls +(-50:0.7cm) and +(50:0.5cm) .. (-2,-2);
   \draw (2,2) .. controls +(-130:0.7cm) and +(130:0.5cm) .. (2,-2);
   \draw (-0.5,-2) .. controls +(100:2cm) and +(80:2cm) .. (0.5,-2);
   \draw[->] (-1.4,1) .. controls +(-30:0.4cm) and +(-150:0.4cm) .. node[below]{$g$} (1.4,1);

   \draw[color=green!50!black] (-2,-2) .. controls +(-50:0.2cm) and +(-80:0.2cm) .. node[pos=0.3, below]{$\alpha$} (-0.5,-2);
   \draw[dashed, color=green!50!black] (-0.5,-2) .. controls +(50:0.2cm) and +(80:0.2cm) ..  (-2,-2);

   \draw[color=blue] (0.5,-2) .. controls +(-50:0.2cm) and +(-130:0.2cm) .. node[pos=0.8, below]{$\alpha'$} (2,-2);
   \draw[dashed, color=blue] (2,-2) .. controls +(50:0.2cm) and +(130:0.2cm) ..  (0.5,-2);
  \end{scope}
 \end{tikzpicture}
	\caption{Homeomorphism of bounded support with $g(S) = S'$.}
	\label{Fig:Half_twist_pants}
\end{figure}
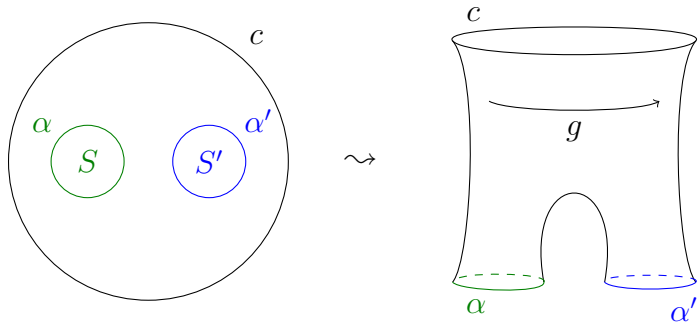
There is always a finite-type closed subsurface $S \subset \Sigma$ with $X \subset S$ such that the boundary of $S$ is a simple closed curve $\alpha \subset \Sigma$. If we translate $S$ along~$\R^2$, we can find another finite-type closed subsurface $S'$ homeomorphic to $S$ whose boundary is $\alpha'$ and such that $S \cap S' = \emptyset$. If we take a simple closed curve $c$ encircling~$\alpha$ and $\alpha'$ such that $\{\alpha,\alpha',c\}$ bounds a pair of pants, we may define $g$ as a half twist around the waist of the pair of pants so that it switches $\alpha$ with $\alpha'$ and $S$ with $S'$, and thus it is supported on a closed subsurface of $\Sigma$ whose boundary is $c$. See Figure~\ref{Fig:Half_twist_pants}.
\end{proof}

\begin{proof}[Proof of (\ref{localization_of_homeos})]
For every simple closed curve $\alpha \subset \Sigma$, denote by $S_\alpha$ the finite-type subsurface whose boundary is $\alpha$. Let $\alpha\subset\Sigma$ now be a curve such that $S \subset S_\alpha$ and observe that $f(S_\alpha) = S_{f(\alpha)}$.
Suppose first that $\alpha$ and $f(\alpha)$ are non-isotopic curves, so that $f(\alpha)$ cannot be contained in~$S_\alpha$ (this is because $S_\alpha$ and $S_{f(\alpha)}$ are punctured discs with the same number of punctures, so if one is contained in the other we have that their boundaries are isotopic). If $\alpha \cap f(\alpha) = \emptyset$ then there is a bigger circle $c$ bounding $\alpha$ and~$f(\alpha)$ such that $\{c,\alpha, f(\alpha) \}$ determines a pair of pants. In this case, we can define $f_0$ to be a homeomorphism whose support is in $S_c$ that acts as a half-twist in this pair of pants just as in the proof of (\ref{movable_support}) and such that it agrees with $f$ in the interior of $S_\alpha$. If $\alpha \cap f(\alpha) \neq \emptyset$, then a more complicated behaviour could happen, see Figure~\ref{Fig:Intersecting_curves}.
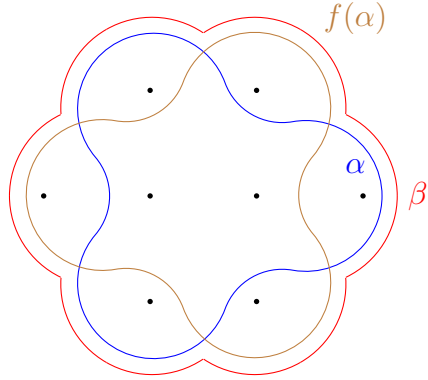
\begin{figure}
\centering
  \begin{tikzpicture}
   \draw[color=blue] (2,0.4) node{$\alpha$};
   \draw[color=blue] (0:2.35cm) arc (0:100:1cm)
    arc (-80:-160:0.8cm)
    arc (20:220:1cm)
    arc (40:-40:0.8cm)
    arc (-220:-20:1cm)
    arc (160:80:0.8cm)
    arc (-100:0:1cm);

   \draw[color=brown] (2,2) node[above]{$f(\alpha)$};
   \draw[color=brown] (60:2.35cm) arc (60:160:1cm)
    arc (-20:-100:0.8cm)
    arc (80:280:1cm)
    arc (100:20:0.8cm)
    arc (-160:40:1cm)
    arc (220:140:0.8cm)
    arc (-40:60:1cm);

   \draw[color=red] (0:2.55cm) node[right]{$\beta$} arc (0:64:1.2cm);
   \draw[color=red] (60:2.55cm) arc (60:-4:1.2);
   \draw[color=red] (60:2.55cm) arc (60:124:1.2);
   \draw[color=red] (120:2.55cm) arc (120:56:1.2);
   \draw[color=red] (120:2.55cm) arc (120:184:1.2);
   \draw[color=red] (180:2.55cm) arc (180:116:1.2);
   \draw[color=red] (180:2.55cm) arc (180:244:1.2);
   \draw[color=red] (240:2.55cm) arc (240:176:1.2);
   \draw[color=red] (240:2.55cm) arc (240:304:1.2);
   \draw[color=red] (300:2.55cm) arc (300:236:1.2);
   \draw[color=red] (300:2.55cm) arc (300:364:1.2);
   \draw[color=red] (0:2.55cm) arc (0:-64:1.2);

   \fill (-2.1,0) circle (1pt);
   \fill (-0.7,0) circle (1pt);
   \fill (0.7,0) circle (1pt);
   \fill (2.1,0) circle (1pt);
   \fill (-0.7,1.4) circle (1pt);
   \fill (0.7,1.4) circle (1pt);
   \fill (-0.7,-1.4) circle (1pt);
   \fill (0.7,-1.4) circle (1pt);
  \end{tikzpicture}
	\caption{Configuration of intersection between $\alpha$ and $f(\alpha)$.}
	\label{Fig:Intersecting_curves}
\end{figure}
In this case, we can take $\beta \subset \Sigma$ to be the boundary component of a regular neighborhood of $\alpha \cup f(\alpha)$ as in Figure~\ref{Fig:Intersecting_curves}. Note that $S_\beta \backslash S_\alpha$ and $S_\beta \backslash S_{f(\alpha)}$ are punctured annuli, each with the same amount of punctures. By the classification of finite-type surfaces, there exists a homeomorphism $h\in\Homeo^+(\Sigma)$, which is the identity in $\Sigma\setminus S_\beta$ (in particular it fixes the curve $\beta$), agrees with $f$ on the curve $\alpha$ and switches
the annuli $S_\beta \backslash S_\alpha$ and $S_\beta \backslash S_{f(\alpha)}$, see Figure~\ref{Fig:complement_of_curves}. We define then the homeomorphism:

\begin{equation}
       \label{EQ:homeo-property2}
       f_0(z) =
       \begin{cases}
       f(z)  & z \in \overline{S_\alpha} \\
       h(z)  & z \in S_\beta \backslash S_\alpha \\
       z     & z \in \Sigma  \backslash S_\beta
       \end{cases}
\end{equation}

\begin{figure}[!ht]
\centering
  \begin{tikzpicture}[scale=0.5]
   \path[pattern=north east lines, pattern color=blue] (0,0) circle (2.8cm);
   \fill[color=white] (0:2.35cm) arc (0:100:1cm)
    arc (-80:-160:0.8cm)
    arc (20:220:1cm)
    arc (40:-40:0.8cm)
    arc (-220:-20:1cm)
    arc (160:80:0.8cm)
    arc (-100:0:1cm);
   \draw (-90:4) node{$S_\beta \setminus S_\alpha$};

   \draw[color=blue] (1.8,0.4) node{$\alpha$};
   \draw[color=blue] (0:2.35cm) arc (0:100:1cm)
    arc (-80:-160:0.8cm)
    arc (20:220:1cm)
    arc (40:-40:0.8cm)
    arc (-220:-20:1cm)
    arc (160:80:0.8cm)
    arc (-100:0:1cm);

   \draw[color=red] (0,0) circle (2.8cm);
   \draw[color=red] (-2.8,0) node[left]{$\beta$};

   \fill (-2.1,0) circle (1pt);
   \fill (-0.7,0) circle (1pt);
   \fill (0.7,0) circle (1pt);
   \fill (2.1,0) circle (1pt);
   \fill (-0.7,1.4) circle (1pt);
   \fill (0.7,1.4) circle (1pt);
   \fill (-0.7,-1.4) circle (1pt);
   \fill (0.7,-1.4) circle (1pt);

   \draw (4.5,0) node{$\cong$};
   \draw (4.5,0.8) node{$h$};

   \begin{scope}[xshift=9cm]
   \path[pattern=north east lines, pattern color=green] (0,0) circle (2.8cm);
   \fill[color=white] (60:2.35cm) node[below]{$f(\alpha)$} arc (60:160:1cm)
    arc (-20:-100:0.8cm)
    arc (80:280:1cm)
    arc (100:20:0.8cm)
    arc (-160:40:1cm)
    arc (220:140:0.8cm)
    arc (-40:60:1cm);
   \draw (-80:3.8) node{$S_\beta \setminus S_{f(\alpha)}$};

   \draw[color=brown] (0.5,0.6) node {$f(\alpha)$};
   \draw[color=brown] (60:2.35cm)  arc (60:160:1cm)
    arc (-20:-100:0.8cm)
    arc (80:280:1cm)
    arc (100:20:0.8cm)
    arc (-160:40:1cm)
    arc (220:140:0.8cm)
    arc (-40:60:1cm);

   \draw[color=red] (0,0) circle (2.8cm);
   \draw[color=red] (2.8,0) node[right]{$\beta$};

   \fill (-2.1,0) circle (1pt);
   \fill (-0.7,0) circle (1pt);
   \fill (0.7,0) circle (1pt);
   \fill (2.1,0) circle (1pt);
   \fill (-0.7,1.4) circle (1pt);
   \fill (0.7,1.4) circle (1pt);
   \fill (-0.7,-1.4) circle (1pt);
   \fill (0.7,-1.4) circle (1pt);
   \end{scope}

  \end{tikzpicture}
	\caption{Homeomorphism $h$ switching the punctured annuli $S_\beta \backslash S_\alpha$ and $S_\beta \backslash S_{f(\alpha)}$.}
	\label{Fig:complement_of_curves}
\end{figure}
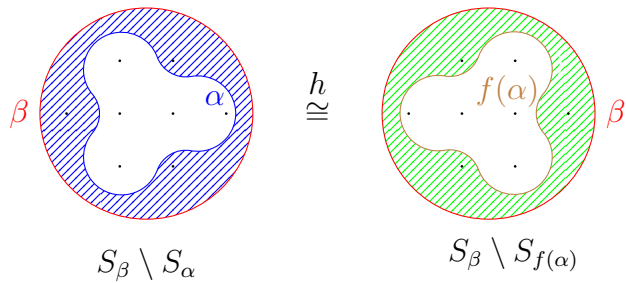
Observe that ${\supp}(f_0)\subset S_\beta$ and hence this support is bounded away from the maximal end.
The remaining case to consider is when $\alpha$ is isotopic to $f(\alpha)$. Here the proof becomes rather trivial because $f$ is isotopic to a function which is the identity on $\alpha$ and which sends the interior of $\alpha$ to itself, so we can define $f_0$ as $f$ in the interior of $S_\alpha$ and as the identity in the rest of $\Sigma$.
\end{proof}

\subsection{Proof of Theorem~\ref{THM:Charact-Dense-Conjugacy-Classes}}
    \label{SSEC:Proof-Theorem-Dense-CC}
    
$\mathbf{(3)\Rightarrow(1)}$. The hypothesis in (\textbf{3}) distills the properties of the surface $\mathbb{R}^2\setminus\mathbb{Z}^2$ needed to guarantee that the class $\mathcal{F}_p$ (w.r.t.\ the Fra\"iss\'e class $\mathcal{F}_{C_{\MCG(\Sigma)}(\Sigma)}$, see Definition~\ref{DEF:FullCurveComplex}) satisfies \textbf{JEP}. This part of the proof is very much inspired by the illustrative example presented in Section~\ref{SSEC:An-Illustrative-Example}.

\begin{definition}
    \label{DEF:Bounded-Away}
Let $\Sigma$ be an infinite-type surface with a unique maximal end $x_\infty$. For any given separating curve $\gamma\subset\Sigma$, let $I_{\gamma}$ be the closed subsurface of $\Sigma$ bounded by $\gamma$ whose space of ends contains $x_\infty$ and let $E_\gamma$ be the complement of the interior of $I_\gamma$ in $\Sigma$.
We say that $X\subset \Sigma$ is \emph{bounded away from the maximal end} if there exists a separating curve $\gamma$ such that~$X$ is contained in the interior of $E_\gamma$.
\end{definition}

We now show that the class $\mathcal{F}_p$ (w.r.t.\ the Fra\"iss\'e class $\mathcal{F}_{C_{\MCG(\Sigma)}(\Sigma)}$, see Definition~\ref{DEF:FullCurveComplex}) satisfies~$\textbf{JEP}$. The proof is similar to the proof of Lemma~\ref{lemma:GEP}. Let $\langle A, \varphi: B\to C\rangle, \langle A', \varphi': B'\to C'\rangle \in \mathcal F_p$. We assume that $A=B\cup C$ and $A'=B'\cup C'$ (see Remark~\ref{rmk:A=BUC}). Let~$S$ and $S'$ be small tubular neighbourhoods of $B$ and $B'$ respectively (thought of as sets of fixed essential simple closed curves in $\Sigma$).
From the hypothesis, we deduce that there exist $f,h\in \MCG(\Sigma)$ such that $f\restriction S$ and $h\restriction S'$ induce $\varphi$ and $\varphi'$ respectively. Moreover, we can suppose that both ${\supp(f)}$ and ${\supp(h)}$ are bounded away from the maximal end.

If $S\cap {\supp}(h)=h(S')\cap {\supp}(f)=\emptyset$ then $H = f \circ h$ satisfies $(\ref{EQ:Disjoint-support-case})$ and hence the pairs $(f,S)$ and $(h,S')$ can be jointly embedded, in the sense of \textbf{GEP}. Using the same arguments as the ones used in the proof of Lemma~\ref{lemma:GEP} to show that \textbf{GEP} implies \textbf{JEP}, we can conclude that the classes $\langle A, \varphi: B\to C\rangle, \langle A', \varphi': B'\to C'\rangle \in \mathcal F_p$ can be jointly embedded.

For the general case, we use the following:

\emph{Claim}: there exist separating curves $\alpha$ and $\gamma$ such that:
\begin{enumerate}
    \item ${\supp}(f)\cup{\supp}(h)\cup h(S')\cup S$ is contained in $ E_\alpha$,\\
    \item $E_\alpha\cap E_\gamma=\emptyset$, and \\
    \item there exists $g\in\Homeo(\Sigma)$ sending $E_\alpha$ to $E_\gamma$.
\end{enumerate}
If such curves exist then define $H:=f\circ (g\circ h\circ g^{-1})$.
The rest of the proof follows from the arguments explained in Section~\ref{SSEC:An-Illustrative-Example} around Equation (\ref{EQ:Illustrative-GEP}).

\emph{Proof of claim}. Part (1) follows from the fact that the supports of $f$ and $h$ are bounded away from the maximal end and $S$ and $S'$ are finite-type subsurfaces. If the closure of $E_\alpha$ in~$\Sigma$ is a compact subsurface, it cannot be non-displaceable. Hence, there exists $g \in \Homeo(\Sigma)$ such that $g(E_\alpha) \cap E_\alpha = \emptyset$ and $\gamma := g(\alpha)$ satisfies (2) and (3). If the closure of $E_\alpha$ in $\Sigma$ is not compact then the space of ends of $\overline{E_\alpha}$ is not empty. The existence of $\gamma$ satisfying~(2) and~(3) above is a consequence of the following two facts: (i) the space of ends of $\overline{E_\alpha}$ does not contain the maximal end $x_\infty$ and (ii) the space of ends of $\Sigma$ is self-similar. For this last point, see Proposition 4.8 in~\cite{mann-rafi}.

$\mathbf{(2)\Leftrightarrow(3)}$. Given that $(2)$ is contained in $(3)$, we only have to prove that $(2)\Rightarrow(3)$.  Let $\phi:B\to C$ be an isomorphism between finite substructures of the full curve graph of $\Sigma$ w.r.t.~$\MCG(\Sigma)$. In what follows, we show that there exists $h\in \Homeo^+(\Sigma)$ that induces\,\footnote{That is, if $B$ is formed by curves $\{\beta_1,\ldots,\beta_n\}$ then $h(\beta_i)=\phi(\beta_i)$ for all $i=1,\ldots,n$.}~$\phi$ and whose support is bounded away from the maximal end. The following lemma tells us that the real difficulty is when $B$ is a singleton.

\begin{lemma}
	\label{LEMMA:reduction_to_singletons}
 Let $B$ and $C$ be singletons and $\phi: B \to C$ an isomorphism.\footnote{Henceforth, all isomorphisms are between finite substructures  of the full curve graph of $\Sigma$ w.r.t.~$\MCG(\Sigma)$.} Then there exists a homeomorphism $g: \Sigma \to \Sigma$ that induces\,\footnote{In the context of this lemma, this means that $g(B)=C$.} $\phi$ and whose support is bounded away from the maximal end.
\end{lemma}

We postpone the proof of this lemma to explain first how it is used to obtain the desired conclusion.

If $B$ is not a singleton, let $U \subset \Sigma$ be a subsurface such that $U^\ast$ is a closed neighborhood of the unique maximal end, $U$ has a unique boundary curve $\beta$ and every element of $B$ is an essential curve of $\Sigma \setminus U$. Since $\phi$ is an isomorphism and the full curve graph is an ultrahomogeneous structure w.r.t.~$\MCG(\Sigma)$, there exists a homeomorphism $f: \Sigma \to \Sigma$ that realizes $\phi$. Define $B^{\prime} = B \cup \{\beta\}$ and $\varphi: B^{\prime} \to C\cup\{f(\beta)\}$. By Lemma~\ref{LEMMA:reduction_to_singletons}, we know that there exists a homeomorphism $g:\Sigma \to \Sigma$ whose support is bounded away from the maximal end and that realizes $\varphi \restriction \{\beta\}$. Note that since $g(\beta) = \varphi(\beta) = f(\beta)$ and $U^\ast$ is a neighbourhood of the unique maximal end, we have that $g(\Sigma \setminus U) = f(\Sigma \setminus U)$. Define the following homeomorphism of $\Sigma$:
$$
 h(p)=
 \begin{cases}
 g(p) & p \in U\\
 f(p)    & p \in \Sigma \setminus U
 \end{cases}
$$
Since $h(\beta) = g(\beta) = f(\beta)$ and $h \restriction U = g \restriction U$, we have that the support of $h$ is bounded away from the unique maximal end. Also, since every element of $B$ is an essential curve of $\Sigma \setminus U$ and $h \restriction \Sigma \setminus U = f \restriction \Sigma \setminus U$, we have that $h(\gamma) = f(\gamma) = \phi(\gamma)$ for every $\gamma \in B$. Therefore $h: \Sigma \to \Sigma$ realizes $\phi$.

\emph{Proof of Lemma~\ref{LEMMA:reduction_to_singletons}}.
Note first that $B$ and $C$ either both correspond to a non-separating curve or both to a separating curve. In the case of non-separating curves, the classification of surfaces yields a homeomorphism $g \colon \Sigma \to \Sigma$ which sends the curve corresponding to $B$ to the curve corresponding to $C$ and whose support is a compact subsurface that contains these two curves.

In the case of separating curves, we use the following general principle. Let $f\in\Homeo^+(\Sigma)$ and suppose that for a given separating curve $\alpha$, one can find another separating curve $\beta$ such that~$ I_\alpha\setminus I_\beta$ is homeomorphic to $ I_{f(\alpha)}\setminus I_\beta$ (see Definition~\ref{DEF:Bounded-Away}). Then there exists a homeomorphism
$$
h:\overline{ I_\alpha\setminus I_\beta}\to\overline{ I_{f(\alpha)}\setminus I_\beta}
$$
such that $h \restriction \{\alpha\}=f \restriction \{\alpha\}$, $h \restriction \{\beta\}=Id \restriction \{\beta\}$ and hence
$$
g(z)=
\begin{cases}
f(z) & z\in E_\alpha \\
z    & z\in I_\beta\\
h(z) & z\in\overline{ I_\alpha\setminus I_\beta}
\end{cases}
$$
is a homeomorphism whose support is bounded away from the maximal end and such that $g(\alpha)=f(\alpha)$. Given that the full curve graph is an ultrahomogeneous structure w.r.t.~$\MCG(\Sigma)$,  Lemma~\ref{LEMMA:reduction_to_singletons} follows.

\begin{remark}
This principle already appeared in the illustrative example discussed
in Section~\ref{SSEC:An-Illustrative-Example}. See Equation~\eqref{EQ:homeo-property2} and the paragraph preceding it for details.
\end{remark}

Let henceforth $\alpha$ denote a separating curve and $f\in\Homeo(\Sigma)$. We proceed considering several cases. In some of them, we find the curve $\beta$ satisfying the principle mentioned above, in others, we see that the conclusion of Lemma~~\ref{LEMMA:reduction_to_singletons} follows.

\emph{Case: $f(\alpha)=\alpha$ modulo isotopy}. This is the simplest case: any separating curve~$\beta$ which is non-isotopic to $\alpha$ and such that $ I_\beta\subset I_\alpha$ satisfies that $ I_\alpha\setminus I_\beta$ is homeomorphic to $ I_{f(\alpha)}\setminus I_\beta$. Henceforth, we suppose that the isotopy class of $\alpha$ is not fixed by $f$.

\emph{Case: $i(f(\alpha), \alpha) \neq 0$}. Let $X$ and $Y$ be the spaces of ends of $I_\alpha$ and $I_{f(\alpha)}$ respectively. By definition, $X$ and $Y$ are homeomorphic. Define $Z :=E(\Sigma)\setminus (X\cup Y)$. Then $Z$ is a clopen (possibly empty) subset of $E(\Sigma)$ which does not contain the maximal end $x_\infty$. Given that $E(\Sigma)$ is self-similar, there exists a curve $\delta$ such that $Z'$, the space of ends of~$E_{\delta}$, is homeomophic to $Z$ and is contained in~$V :=X \cap Y$. Let $\beta$ be the separating curve for which the space of ends of $I_\beta$ is $W := V \setminus Z'$. We claim that $\beta$ is the curve we are looking for. Indeed, if $f^*\in\Homeo(E(\Sigma))$ denotes the homeomorphism induced by $f$, then:
\begin{align*}
 X \setminus W
 &\ = (X \setminus V) \cup Z'
 \simeq (X \setminus V) \cup Z
 = E(\Sigma) \setminus Y
 = f^\ast(E(\Sigma) \setminus X) \\
 &\ \simeq E(\Sigma) \setminus X
 = (Y \setminus V) \cup Z'
 \simeq (Y \setminus V) \cup Z
 = Y \setminus W.
\end{align*}

\emph{Case: $i(f(\alpha), \alpha) = 0$}. We consider two subcases. The simplest one is when $E_{f(\alpha)}\cap E_\alpha =\emptyset$. In this situation, we do not need to find a curve $\beta$ such that $ I_\alpha\setminus I_\beta$ is homeomorphic to $ I_{f(\alpha)}\setminus I_\beta$. Indeed, from Figure~\ref{Fig:simple_homeo} it is easy to see that one can find a homeomorphism $g$ such that $g(\alpha)=f(\alpha)$ and whose support is bounded away from the maximal end.

\begin{figure}[!ht]
\centering
\begin{tikzpicture}
  \draw (0,0) .. controls +(90:0.8cm) and +(-15:0.5cm) .. (-0.9,1)
  .. controls +(165:0.5cm) and +(10:0.5cm) .. (-3,1)
  .. controls +(-170:0.5cm) and +(90:0.4cm) .. (-4,0.4)
  .. controls +(-90:0.4cm) and +(170:0.3) .. (-3,0)
  .. controls +(-10:0.1cm) and +(90:0.1cm) .. (-2.8,-0.2)
  .. controls +(-90:0.1cm) and +(40:0.1cm) .. (-2.9,-0.4)
  .. controls +(-140:0.3cm) and +(80:0.3cm) .. (-3.8,-0.8)
  .. controls +(-100:0.5cm) and +(-170:0.6cm) .. (-2.7,-1.3)
  .. controls +(10:0.5cm) and +(-150:0.7cm) .. (-0.7,-0.8)
  .. controls +(30:0.4cm) and +(-90:0.4cm) .. (0,0);

  \fill (0,0) circle (2pt) node[right]{$x_\infty$};

  \draw[thick, color=blue] (-0.9,1) node[above]{$\beta$} .. controls +(-15:0.3cm) and +(30:0.3cm) .. (-0.7,-0.8);
  \draw[thick, dashed, color=blue] (-0.9,1) .. controls +(165:0.3cm) and +(-150:0.3cm) .. (-0.7,-0.8);

  \draw[thick] (-3,1) node[above]{$\alpha$} .. controls +(-15:0.3cm) and +(-10:0.3cm) .. (-3,0);
  \draw[thick, dashed] (-3,1) .. controls +(170:0.3cm) and +(-150:0.3cm) .. (-3,0);

  \draw[thick] (-2.9,-0.4) .. controls +(40:0.3cm) and +(30:0.3cm) .. (-2.7,-1.3) node[below]{$f(\alpha)$};
  \draw[thick, dashed] (-2.9,-0.4) .. controls +(-140:0.3cm) and +(-150:0.3cm) .. (-2.7,-1.3);
 \end{tikzpicture}
    \caption{}
	\label{Fig:simple_homeo}
\end{figure}

Let us now consider the case when $E_{f(\alpha)} \subset E_\alpha$ (the proof for the case when the reverse inclusion holds is analogous). Once more, we do not seek to find a separating curve $\beta$ such that $ I_\alpha\setminus I_\beta$ is homeomorphic to $ I_{f(\alpha)}\setminus I_\beta$. Instead, we prove that the homeomorphism $g$ satisfying the conclusion of the lemma can be found.

Given that there are no non-displaceable finite-type subsurfaces and a unique maximal end, Proposition 4.8 in~\cite{mann-rafi} tells us that the space of ends of $\Sigma$ is self-similar. Therefore, we can find an essential separating curve $\tilde{\alpha}$ in $I_{\alpha}$ such that $E_\alpha$ is homeomorphic to $E_{\tilde{\alpha}}$. Let $a$ be a simple arc in $\Sigma$ joining $\alpha$ with $\tilde{\alpha}$ and define $\beta$ to be the boundary component in $I_\alpha$ of a closed regular neighborhood of $\alpha\cup a \cup \tilde{\alpha}$, see Figure~\ref{Fig:Construction}.

\begin{figure}[!ht]
\begin{center}
    \begin{center}
 \begin{tikzpicture}[scale=0.6]
  \draw (0,0) circle (2.5cm) +(45:2.5) node[right]{$\alpha$};
  \draw (10,0) circle (2.5cm) +(45:2.5) node[right]{$\tilde{\alpha}$};

  \draw[color=magenta] (2.5,0) -- node[above]{$a$} (7.5,0);

  \draw[color=red] (5,1) .. controls +(0:1.7cm) and +(-135:0.5cm) .. (7.525,2.475)
  arc (135:-135:3.5cm)
  .. controls +(135:0.5cm) and +(0:1.7cm) ..
  (5,-1) node[below]{$\beta$} .. controls +(180:1.7cm) and +(45:0.5cm) .. (-45:3.5cm)
  arc (-45:-315:3.5cm)
  .. controls +(-45:0.5cm) and +(180:1.7cm) .. (5,1);

  \fill (5,3) node[right]{$x_\infty$} circle (0.1cm);
 \end{tikzpicture}
\end{center}
    \caption{}
	\label{Fig:Construction}
\end{center}
\end{figure}

By the Classification Theorem of infinite-type surfaces, there exists a homeomorphism $h:E_\beta\setminus E_\alpha\to E_\beta\setminus E_{f(\alpha)}$ such that $f\restriction\alpha=h\restriction\alpha$ and $h\restriction\beta = Id\restriction\beta$. Then
\begin{equation}
       g(z)=
       \begin{cases}
       z  & z \in I_\beta \\
       h(z)  & z \in \overline{E_\beta\setminus E_\alpha} \\
       f(z)     & z \in E_\alpha
       \end{cases}
\end{equation}
is a homeomorphism with support bounded away from the maximal end and which coincides with $f$ on $\alpha$.

$\mathbf{(1)\Rightarrow(2)}$. First, suppose that $\Sigma$ contains a finite-type non-displaceable subsurface $S$. Then by Corollary~\ref{Coro:MapNotJEP}, we have that $\MCG(\Sigma)$ does not have a dense conjugacy class. We now suppose that $\MCG(\Sigma)$ has a dense conjugacy class and show that there exists a unique maximal end $x_\infty\in E(\Sigma)$. The rest of the proof is divided into two steps:

\underline{Step 1}: \emph{All maximal ends of $\Sigma$ are equivalent}. Indeed, if this was not the case, let $E(x_1)$ and~$E(x_2)$ be two different equivalence classes for maximal ends $x_1,x_2\in E(\Sigma)$.
By Proposition~\ref{prop:topology_maximal_elements}, each $E(x_i)$ is either finite or a Cantor set.
If both $E(x_1)$ and $E(x_2)$ are finite then
the set of mapping classes fixing each point in $E(x_1)\cup E(x_2)$ defines a closed normal subgroup of $\MCG(\Sigma)$ of (finite) index strictly bigger than 1. By Proposition~\ref{normal}, we get a contradiction. The remaining possibility is that at least one of $E(x_1)$ or~$E(x_2)$ is a Cantor set. In this case, it is rather easy to find a non-displaceable pair of pants in $\Sigma$, see Example~2.5 in~\cite{mann-rafi}. By Corollary~\ref{Coro:MapNotJEP}, we obtain a contradiction.

Before we show now in the second step that the unique equivalence class of maximal ends is a singleton, we first prove a lemma that is used to rule out any other possibility.

\begin{lemma}\label{splitting-cantor}
    Let $\Sigma$ be an infinite-type surface without non-displaceable subsurfaces.
	Suppose there exists a maximal end $x \in E(\Sigma)$ such that its equivalence class $E(x)$ coincides with the set of maximal ends and is homeomorphic to the Cantor set. Then, for every $n \in \N$, there exists a partition into disjoint clopen subsets $E(\Sigma) = E_1 \sqcup \cdots \sqcup E_n$, such that $E_i \cong E_j$ for all $i,j \in \{1,\ldots,n\}$.
\end{lemma}

\begin{proof}
	\textit{(1) Local splitting.}
	Note that any end $y \in E(x)$ has a stable neighborhood in the sense of \cite{mann-rafi} (see Remark 4.15 there). From Lemma 4.17 in \cite{mann-rafi}, it follows that for two given ends $y,z \in E(x)$, any two sufficiently small neighborhoods $V_y, V_z \subset E(\Sigma)$ are homeomorphic.
	As $E(x)$ is homeomorphic to a Cantor set, we can split $E(x) = E_1(x) \sqcup \cdots \sqcup E_n(x)$ where the~$E_i(x)$ are Cantor sets homeomorphic to $E(x)$. For a fixed $m \in \mathbb{N}$ big enough, we can cover each of the Cantor sets by disjoint stable clopen neighborhoods $E_j(x) \subset \bigsqcup_{s=1}^m V_{j,s} = X_j$ (this can be done by placing $m$ equally distributed points in each Cantor set with stable neighborhoods around them that can be made disjoint). Thus, we get a splitting into clopen subsets $E(\Sigma) = X_1 \sqcup \cdots \sqcup X_n \sqcup Z$, with $X_1 \cong X_j$ (because each of them is covered by the same amount of disjoint stable neighborhoods) and~$Z$ a clopen subset not containing maximal ends.
	
	\textit{(2) Absorbing non-maximal sets.} Given $U \subset E(\Sigma)$ a clopen subset containing $x \in E(x)$, consider a covering $E(x) \subset \bigsqcup_{j = 1}^l W_j$ with every $W_j$ being a clopen subset that embeds homeomorphically into~$U$ ($W_j$ can be taken the stable neighborhoods $V_{j,s}$ considered above). There exists a compact, connected subsurface $K \subset \Sigma$ such that the closed subsets $\{ W_1, ... , W_l, Z\}$ are contained in the spaces of ends of different connected components of $\Sigma \setminus K$. As~$\Sigma$ does not admit non-displaceable surfaces, there is a homeomorphism $f \in \Homeo(\Sigma)$ such that $f(K) \cap K = \emptyset$. By analyzing the connected components of $\Sigma \setminus K$ and the decomposition they induce on the space of ends, we see that the homeomorphism $f^*$ cannot send a clopen $W_j$ to the subspace containing $Z$, because the corresponding component does not have maximal ends. So what we have is that~$f^\ast(Z)$ lies inside one of the $W_j$ and in particular, $U$ contains a homeomorphic copy of~$Z$.
	
	\textit{(3) Global splitting.} Without loss of generality, suppose $x \in E_1(x) \subset X_1$ and consider~$\{U_n\}$ to be a countable local basis of clopen neighborhoods of $x$, such that $U_{n+1} \subset U_n \subset X_1$. Define $Z_1 = Z$ and recursively choose an element $U_{k_n}$ of the neighborhood basis which does not intersect $Z_n$ (this exists as $Z_n$ does not contain maximal elements) and then $Z_{n+1} \subset U_{k_n}$ a homeomorphic copy of~$Z$ inside $U_{k_n}$. As the elements of the sequence~$\{Z_n\}$ are pairwise disjoint, and the sequence Hausdorff-converges to $\{x\}$, we have a shift homeomorphism $\bigsqcup_{j=1}^\infty Z_n \cong \bigsqcup_{j=2}^\infty Z_n$ (see Observation 4.9 in \cite{mann-rafi}), which extends to the homeomorphism $ (X_1 \sqcup Z) \cong X_1$. So the splitting that we were looking for is $E_1 = (X_1 \sqcup Z)$, and $E_j = X_j$, for $j > 1$.
\end{proof}

\underline{Step 2}: Let $E(x)$ be the unique equivalence class that contains maximal ends.
Again, $E(x)$ is either finite or a Cantor set. If $E(x)$ is finite then it has only one element (if not, we can use Proposition~\ref{normal} again to get a contradiction as before), which is precisely what we wanted to show. Now suppose that $E(x)$ is Cantor set.
Consider the partition of the space of ends into three homeomorphic copies $E(\Sigma) = E_0 \sqcup E_1 \sqcup E_2$ given by Lemma~\ref{splitting-cantor}. Then there is a pair of pants $P \subset \Sigma$ such that:
\begin{equation*}
 \Sigma\setminus P = \Sigma_0\sqcup\Sigma_1\sqcup\Sigma_2, \qquad E_i \subset E(\Sigma_i).
\end{equation*}
As the three components of $\Sigma\setminus P$ are homeomorphic, there is an $R\in\MCG(\Sigma)$ leaving $P$ invariant and sending $\Sigma_i$ to $\Sigma_{i+1}$, for all $i\in\Z/3\Z$.

Consider now the pairs $(R,P)$ and $(Id_{\Sigma},P)$. We claim that these cannot be jointly embedded, and hence by Lemma~\ref{lemma:GEP}, we obtain a contradiction to the fact that $\MCG(\Sigma)$ has a dense conjugacy class. Suppose that there exist $g,g',H\in\Homeo^+(\Sigma)$ such that:
\begin{itemize}
\item [(i)] $g \circ R\restriction P = H \circ g\restriction P$ and
\item [(ii)] $g' \circ Id \restriction P = H \circ g'\restriction P$.
\end{itemize}
This is equivalent to $g \circ R \circ g^{-1} \restriction g(P) = H \restriction g(P)$ and $Id \restriction g'(P) = H \restriction g'(P)$.
Note that $g \circ R \circ g^{-1}$ permutes the connected components of $\Sigma \setminus g(P)$.
If $g'(P) \cap g(P) \neq \emptyset$, then $H$ acts at the same time as the identity and as a non-trivial rotation in $g'(P) \cap g(P)$, which is impossible. If $g'(P) \cap g(P) = \emptyset$, then $g'(P) \subset g(\Sigma_j)$ for some $j \in \{1,2,3\}$, but this is also impossible since, on the one hand, $H$ acts as the identity on $g'(P)$ and on the other hand, $H(g'(P)) \subset H(g(\Sigma_j)) = g(\Sigma_{j+1})$ with $g(\Sigma_{j+1}) \cap g(\Sigma_j) = \emptyset$. Hence the only possibility left is that $E(x)$ is a singleton as desired.

%% file: nowhere-dense.tex
\section{Nowhere dense and somewhere dense conjugacy classes}
    \label{SEC:Nowhere-Dense-Conjugacy-Classes}

In this section, we prove a generalized version of Theorem~\ref{THM:non-displaceable-the-nowhere-dense-versionintro} on nowhere dense conjugacy classes and Theorem~\ref{thm:Somewhere-dense} on somewhere dense conjugacy classes.
We then apply the results of the former to prove Theorem~\ref{THM:CC_for_PMap}.

\subsection{Nowhere dense conjugacy classes}
In this subsection, we study the conditions under which conjugacy classes of closed subgroups of a big mapping class group have \emph{nowhere dense} conjugacy classes. In what follows, $\Gamma$ is a closed subgroup of $\MCG(\Sigma)$ that contains, for each essential simple closed curve $\alpha\in\Sigma$, a non-trivial power of the Dehn twist $\tau_\alpha$.
Examples of such subgroups are $\MCG(\Sigma), \pmap{\Sigma}$ and $\overline{\pmapc{\Sigma}}$. Recall that:
\begin{itemize}
    \item for each curve $\alpha$ on $\Gamma$, the number $N_{\alpha}$ denotes the minimal positive power of the Dehn twist $\tau_{\alpha}$ that is contained in $\Gamma$,
    \item if $\alpha$ and $\beta$ are in the same $\Gamma$-orbit, then $N_{\alpha} = N_{\beta}$, and
    \item we say that a multicurve $M$ cannot be separated from itself by $\Gamma$ if for all $g,g' \in \Gamma$, there exist $\alpha \in g(M)$ and $\beta \in g'(M)$ such that either $\alpha = \beta$ or $i(\alpha, \beta) \neq 0$.
\end{itemize}

We also define a generalized notion of non-displaceable subsurfaces for this subsection: We say that a connected subsurface $S \subset \Sigma$ is \emph{$\Gamma$-non-displaceable} if for every representative~$f$ of a mapping class in $\Gamma$, we have that $f(S)\cap S\neq\emptyset$.

The main result that we prove in this subsection is the following:

\begin{theorem}
    \label{THM:non-displaceable-the-nowhere-dense}
 If $\Sigma$ has a $\Gamma$-non-displaceable finite-type subsurface $S$ and $W \subset \Gamma$ is a non-empty open set, then there exist $V_{1}, V_{2} \subset W$ non-empty and disjoint open sets satisfying that any conjugacy class with non-empty intersection with $V_{i}$, is disjoint from $V_{j}$, for $i \neq j$. In particular, all conjugacy classes of elements of $\Gamma$ are nowhere dense in $\Gamma$.
\end{theorem}

For the proof of this theorem, we need some technical lemmas.

\begin{lemma}\label{lemma:DehnTwistTrickNWD}
 Let $K$ be a compact subsurface of $\Sigma$, $f$ a homeomorphism of $\Sigma$ and $M$ be a multicurve of $\Sigma \setminus K$ that cannot be separated from itself by $\Gamma$. If $S$ is a connected subsurface of~$\Sigma$ (not necessarily of finite type) that contains $K$ and contains every element of $M$ as an essential or a boundary curve, then there exist two homeomorphisms $h, h'$ of $\Sigma$ such that:
 \begin{enumerate}
  \item $h \restriction K = h' \restriction K = f \restriction K$ up to isotopy.
  \item There do not exist $g,g',H \in \Gamma$ with $g \circ h \circ g^{-1} \restriction g(S) = H \restriction g(S)$ and $g' \circ h' \circ (g')^{-1} \restriction g'(S) = H \restriction g'(S)$ up to isotopy.
 \end{enumerate}
\end{lemma}
\begin{proof}
 For each $\alpha \in M$, let $b_{\alpha,f}: \overline{N(\alpha)} \to \overline{N(f(\alpha))}$ be the homeomorphism from the closed annulus $\overline{N(\alpha)}$ to the annulus $\overline{N(f(\alpha))}$ as defined in Figure \ref{figure:balpha}. Note that, when we fix the boundary, this means that $b_{\alpha,f}$ does not twist the annulus.

 \begin{figure}[ht]
     \begin{center}
      \begin{tikzpicture}
      \draw (0,0) circle (0.5cm);
      \draw (0,0) circle (1.5cm);
      \draw[thick, color=purple] (0,0) circle (1cm);
      \draw[color=purple] (1.1,-0.5) node{$\alpha$};
      \draw[thick, color=green] (-1.5,0) -- (-0.5,0);
      \draw (0,2) node{$\overline{N(\alpha)}$};

      \draw[->] (2,0) -- node[above]{$b_{\alpha,f}$} (3,0);

      \draw (5,0) circle (0.5cm);
      \draw (5,0) circle (1.5cm);
      \draw[thick, color=blue] (5,0) circle (1cm);
      \draw[color=blue] (7,-0.5) node{$f(\alpha)$};
      \draw[thick, color=green] (3.5,0) -- (4.5,0);
      \draw (5,2) node{$\overline{N(f(\alpha))}$};
      \end{tikzpicture}
     \end{center}

   \caption{In pink $\alpha$, in blue $f(\alpha)$, in green untwisted transversal arcs in both $N(\alpha)$ and $N(f(\alpha))$. Then, $b_{\alpha,f}$ maps one green transversal to the other. }
   \label{figure:balpha}
 \end{figure}
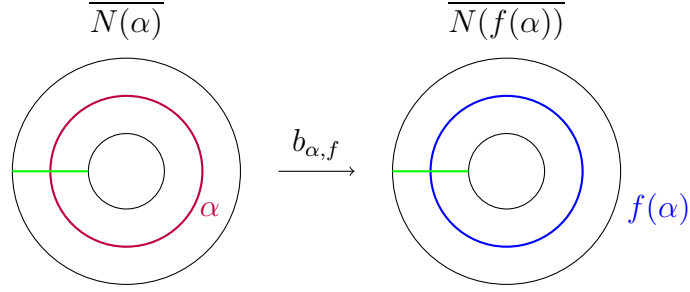

 Then there exists a unique $n(\alpha,f) \in \mathbb{Z}$ such that $f \restriction N(\alpha) = \tau_{f(\alpha)}^{n(\alpha,f)} \circ b_{\alpha,f}$.
 Note that $n(\alpha,f) = n(g(\alpha),g\circ f \circ g^{-1})$ for every $g \in \Gamma$.

 For each $\alpha \in M$, we define $k_{\alpha}$ and $k'_{\alpha}$ so that they satisfy $k'_{\alpha} \cdot N_{\alpha} + n(\alpha,f) > N_{\alpha} > k_{\alpha} \cdot N_{\alpha} + n(\alpha,f)$.

 We define $\displaystyle h := \left(\prod_{\alpha \in M} \tau_{f(\alpha)}^{k_{\alpha}\cdot N_{\alpha}}\right) \circ f = f \circ \left(\prod_{\alpha \in M} \tau_{\alpha}^{k_{\alpha}\cdot N_{\alpha}}\right)$ and $\displaystyle h' := \left(\prod_{\alpha \in M} \tau_{f(\alpha)}^{k'_{\alpha} \cdot N_{\alpha}}\right) \circ f = f \circ \left(\prod_{\alpha \in M} \tau_{\alpha}^{k'_{\alpha} \cdot N_{\alpha}}\right)$. By definition of $h$, $h'$ and $M$, we have that $h \restriction K = h' \restriction K = f \restriction K$.

 Suppose there exist $g,g',H \in \Gamma$ such that $g \circ h \circ g^{-1} \restriction g(S) = H \restriction g(S)$ and $g' \circ h' \circ (g')^{-1} \restriction g'(S) = H \restriction g'(S)$ up to isotopy. Then, we have the following facts for all $\alpha, \beta \in M$:

 \begin{enumerate}
  \item $H(g(\alpha)) = g \circ h \circ g^{-1} (g(\alpha)) = g \circ h (\alpha) = g \circ \tau_{f(\alpha)}^{k_{\alpha}\cdot N_{\alpha}} \circ f (\alpha) = g(f(\alpha))$.
  \item $H(g'(\beta)) = g' \circ h' \circ (g')^{-1} (g'(\beta)) = g' \circ h' (\beta) = g' \circ \tau_{f(\beta)}^{k'_{\beta} \cdot N_{\beta}} \circ f (\beta) = g'(f(\beta))$.
  \item On one hand, we have that:
  \begin{center}$\displaystyle \begin{array}{rcl}
  H \restriction N(g(\alpha)) & = & g \circ \tau_{f(\alpha)}^{k_{\alpha} \cdot N_{\alpha}} \circ f \circ g^{-1} \restriction N(g(\alpha))\\[0.2cm] & = & g \circ \tau_{f(\alpha)}^{k_{\alpha} \cdot N_{\alpha}} \circ \tau_{f(\alpha)}^{n(\alpha,f)} \circ b_{\alpha,f} \circ g^{-1} \restriction N(g(\alpha))\\[0.2cm] & = & g \circ \tau_{f(\alpha)}^{k_{\alpha} \cdot N_{\alpha} + n(\alpha,f)} \circ b_{\alpha,f} \circ g^{-1} \restriction N(g(\alpha))\\[0.2cm] & = & \tau_{g(f(\alpha))}^{k_{\alpha} \cdot N_{\alpha} + n(\alpha,f)} \circ g \circ b_{\alpha,f} \circ g^{-1} \restriction N(g(\alpha))\\[0.2cm] & = & \tau_{g(f(\alpha))}^{k_{\alpha} \cdot N_{\alpha} + n(\alpha,f)} \circ b_{g(\alpha),g \circ f \circ g^{-1}}
  \end{array}$\end{center}
  This means that in $N(g(\alpha))$, $H$ behaves as first mapping $N(g(\alpha))$ to $N(g(f(\alpha)))$ untwisted, and then twisting $N(g(f(\alpha)))$ strictly less than $N_{\alpha}$ times.
  \item On the other hand, we have that:
  \begin{center}$\displaystyle \begin{array}{rcl}
  H \restriction N(g'(\beta)) & = & g' \circ \tau_{f(\beta)}^{k'_{\beta} \cdot N_{\beta}} \circ f \circ (g')^{-1} \restriction N(g'(\beta))\\[0.2cm] & = & g' \circ \tau_{f(\beta)}^{k'_{\beta} \cdot N_{\beta}} \circ \tau_{f(\beta)}^{n(\beta,f)} \circ b_{\beta,f} \circ (g')^{-1} \restriction N(g'(\beta))\\[0.2cm] & = & g' \circ \tau_{f(\beta)}^{k'_{\beta} \cdot N_{\beta} + n(\beta,f)} \circ b_{\beta,f} \circ (g')^{-1} \restriction N(g'(\beta))\\[0.2cm] & = & \tau_{g'(f(\beta))}^{k'_{\beta} \cdot N_{\beta} + n(\beta,f)} \circ g' \circ b_{\beta,f} \circ (g')^{-1} \restriction N(g'(\beta))\\[0.2cm] & = & \tau_{g'(f(\beta))}^{k'_{\beta} \cdot N_{\beta} + n(\beta,f)} \circ b_{g'(\beta),g' \circ f \circ (g')^{-1}}
  \end{array}$\end{center}
  This means that in $N(g'(\beta))$, $H$ behaves as first mapping $N(g'(\beta))$ to $N(g'(f(\beta)))$ untwisted, and then twisting $N(g'(f(\beta)))$ strictly more than $N_{\beta}$ times.
 \end{enumerate}

 We now finish the proof by dividing it into two cases and reaching a contradiction in each~case:

 \textbf{Case 1}. If there exist $\alpha, \beta \in M$ with $g(\alpha) = g'(\beta)$ then $\alpha$ and $\beta$ are in the same $\Gamma$-orbit, which implies that $N_{\alpha} = N_{\beta}$. Also, by points (1) and (2) above, we have that $g(f(\alpha))$ is isotopic to $g'(f(\beta))$. Hence, using points (3) and (4), we reach a contradiction.

 \textbf{Case 2}. If there exist $\alpha, \beta \in M$ with $i(g(\alpha), g'(\beta)) \neq 0$ then $i(H(g(\alpha)), H(g'(\beta))) \neq 0$. By points (1) and (2) above, we have that:
 $$
 i(g(f(\alpha)), g'(f(\beta))) \neq 0.
 $$
 Therefore $N(g(f(\alpha)))$ and $N(g'(f(\beta)))$ intersect, even up to isotopy. Hence, we reach a contradiction by points (1) and (4) above, seeing that by (1) $H(g(\alpha)) = g(f(\alpha))$, while by~(4), we get that~$H$ sends subarcs of $g(\alpha)$ to the image of $g(f(\alpha))$ under a non-trivial power of $\tau_{g'(f(\beta))}$.
\end{proof}

 The following lemma can be thought of as a generalization of Theorem~\ref{Theo:NotJEP}.

\begin{lemma}\label{lemma:nondisp-nolocalGEP}
 Let $S$ be a $\Gamma$-non-displaceable subsurface of $\Sigma$ (not necessarily of finite topological type) and let $U \subset \Gamma$ be an open set. If $f \in U$, then there exist $h,h'\in U$ and a compact subsurface $K$ such that:
 \begin{enumerate}
  \item $h \restriction K = h' \restriction K = f \restriction K$ up to isotopy.
  \item There do not exist $g,g',H \in \Gamma$ with $g \circ h \circ g^{-1} \restriction g(S) = H \restriction g(S)$ and $g' \circ h' \circ (g')^{-1} \restriction g'(S) = H \restriction g'(S)$ up to isotopy.
 \end{enumerate}
\end{lemma}
\begin{proof}
 Since $f \in U$, there exists a compact subsurface $K$ such that $$V := \{g \in \Gamma: g \restriction K = f \restriction K\},$$ is an open neighbourhood of $f$ contained in $U$.

 Moreover, due to Alexander's Lemma, there exists a finite-type subsurface $\widetilde{K}$ such that all the boundary curves of $\widetilde{K}$ are essential, and for all $g \in \Gamma$ we have that $g \restriction K = f \restriction K$ if and only if $g \restriction \widetilde{K} = f \restriction \widetilde{K}$ up to isotopy.

 Without loss of generality, we can assume that $S$ contains $\widetilde{K}$ such that every boundary curve of $\widetilde{K}$ is an essential curve in $S$. Furthermore, we possibly enlarge $S$ such that every connected component of $S \setminus \widetilde{K}$ has complexity larger than the complexity of $\widetilde{K}$ or the boundary of the component is contained in the boundary of $\widetilde{K}$.

 Now, let $B$ be the multicurve composed of all the boundary curves of $S$, $P$ be a pants decomposition of $S \setminus \widetilde{K}$, and $M = B \cup P$.
 Let $g,g' \in \Gamma$. Since $S$ is $\Gamma$-non-displaceable,~$g(S)$ and~$g'(S)$ have non-trivial intersection. By the complexity condition above, it follows that also $g(S \setminus \widetilde{K})$ and $g'(S \setminus \widetilde{K})$ have non-trivial intersection. This implies that there exist $\alpha \in g(M)$ and $\beta \in g'(M)$ such that $\alpha = \beta$ or $i(\alpha,\beta) \neq 0$.
 In particular, $M$ cannot be separated from itself by $\Gamma$. Thus, by Lemma \ref{lemma:DehnTwistTrickNWD}, there exist $[h], [h'] \in V \subset U$ that satisfy~(1) and~(2), finishing the proof.
\end{proof}

\emph{Proof of Theorem~\ref{THM:non-displaceable-the-nowhere-dense}.}
 Let $f \in W$. Then there exists a compact subsurface $K$ such that $U = \{g \in \Gamma: g\restriction K = f \restriction K\}$ is an open neighbourhood of $f$ contained in $W$.

 Let $h,h' \in U$ be the mapping classes obtained by Lemma \ref{lemma:nondisp-nolocalGEP}, and let $g, g' \in \Gamma$ be two different elements. We define: $$V_{1} = \{\widetilde{f} \in \Gamma: \widetilde{f} \restriction g(S) = g \circ h \circ g^{-1} \restriction g(S)\},$$ $$V_{2} = \{\widetilde{f} \in \Gamma: \widetilde{f} \restriction g'(S) = g' \circ h' \circ (g')^{-1} \restriction g'(S)\}.$$

 By construction, $V_{1}$ and $V_{2}$ are non-empty open sets contained in $U$.

 If there existed $f_1, f_2, H \in \Gamma$ such that $f_{i} \circ H \circ f_{i}^{-1} \in V_{i}$ for each $i = 1,2$, we would have that $f_{1}^{-1} \circ g$, $f_{2}^{-1} \circ g'$ and $H$ contradict property~(2) of Lemma \ref{lemma:nondisp-nolocalGEP}. Thus, any conjugacy class with non-empty intersection with $V_{i}$ is disjoint from $V_{j}$ for $i \neq j$.

 Since any element is contained in a conjugacy class, it is obvious that $V_{1}$ is disjoint from~$V_{2}$.
\qed

\subsection{Somewhere dense conjugacy classes}
In this subsection, we address the proof of Theorem~\ref{thm:Somewhere-dense}, characterizing thus when a big mapping class group has a conjugacy class which is dense somewhere.

\emph{Proof of Theorem~\ref{thm:Somewhere-dense}}.
We show first that, if $\MCG(\Sigma)$ has a conjugacy class which is dense somewhere then $\Sigma$ has no non-displaceable finite-type subsurfaces and $\mathcal{M}$, the set of maximal ends of $\Sigma$, consists of at most two maximal ends. If $\Sigma$ has a non-displaceable finite-type subsurface, then by Theorem~\ref{THM:non-displaceable-the-nowhere-dense}, all conjugacy classes of $\MCG(\Sigma)$ are nowhere dense. If there are no non-displaceable finite-type subsurfaces and $\mathcal{M}$ has cardinality strictly larger than $2$ then $\mathcal{M}$ is a Cantor set. This follows from Example 2.5 in~\cite{mann-rafi}: if $Z$ is a finite and $\MCG(\Sigma)$-invariant set of ends of cardinality at least $3$, then any subsurface $S$ which separates all the ends in $Z$ into different complimentary regions is non-displaceable. We claim that an infinite-type surface $\Sigma$ with a Cantor set of maximal ends cannot have a conjugacy class that is somewhere dense.
Let $U$ be an open set of the form $f\cdot U_A$ for some $f\in \MCG(\Sigma)$ and~$A$ a finite set of simple closed curves where $U_A \subset \MCG(\Sigma)$ is defined analogously to~(\ref{EQ:Basic-permutation-topo})). We can extend $A$ to a larger, more convenient set of curves and by this pass to an open subset of $U$: As $\mathcal{M}$ is a Cantor set, there exists a separating curve $\alpha \subset \Sigma$ such that $\Sigma\setminus\alpha$ has a connected component $C$ whose set of ends contains a (not necessarily proper) subset $\mathcal{M}'\subset\mathcal{M}$ homeomorphic to the Cantor set. Furthermore, we can choose $\alpha$ such that no element of $A$ is contained in $C$ or intersects $\alpha$.
We enlarge $A$ by this curve $\alpha$.

Addititionally, let $\beta_1,\beta_2$ be two separating curves in $C$ which together with $\alpha$ bound a pair of pants $P$ such that $\Sigma\setminus P$ induces a partition of $\mathcal{M}'=\mathcal{M}'_1\sqcup \mathcal{M}'_2$ such that $\mathcal{M}'_i$ is a Cantor set for $i=1,2$. Now define $h_2=\tau^2\circ f$ and $h_3=\tau^3\circ f$, where $\tau$ is supported on a small regular neighbourhood of $\{f(\beta_1),f(\beta_2)\}$ and acts as a Dehn twist on these. Given that $\{\alpha,\beta_1,\beta_2\}$ do not intersect any curve in $A$, we have that $h_2,h_3\in U$. Moreover, by considering $A'=A\sqcup\{\beta_1,\beta_2\}$, the restriction of $h_i$ to $A'$ defines an open subset~$U_i$ of $U$, for $i=2,3$. Indeed, this follows from the fact that $h_i$ restricted to $A$ coincides with~$f$.
Note that a conjugacy class which intersects $U_2$ cannot intersect $U_3$. Indeed, take $g\in U_2$. Then~$g$ must act as $\tau^2$ on $\{f(\beta_1),f(\beta_2)\}$. On the other hand, any conjugate of $g$ which lies in~$U_3$ must act on a neighbourhood of $f(\beta_1) \cup f(\beta_2)$ on one hand as $\tau^3$ and on the other as $\tau^2$ (see properties of Dehn twists listed in Section~\ref{SSEC:Dehn-Twists}). This is absurd.

We now show that if $\Sigma$ has no non-displaceable finite-type subsurfaces and $\mathcal{M}$ consists of at most two maximal ends, then there exists a conjugacy class which is somewhere dense. The hypothesis on non-displaceable subsurfaces is necessary: Figure~\ref{Fig:two-ends-nondisplaceable} shows an example for which $\Sigma$ has exactly two maximal ends and the surface has non-displaceable finite-type subsurfaces.

\begin{figure}
\begin{center}
    \begin{center}
     \begin{tikzpicture}[scale=0.3]
     \path[pattern color=blue!10, pattern=north east lines] (17.5,2) .. controls +(180:0.5cm) and +(180:0.5cm) .. (17.5,-2)
      -- (27.5,-2) .. controls +(0:0.4cm) and +(0:0.4cm) .. (27.5,2);
     \fill[color=white] (18.8,0.05) to[bend left] (21.2,0.05) -- (21.5,0.15) to[bend left] (18.5,0.15) -- (18.8,0.05);

     \draw (-2.5,-2) -- (47.5,-2);
     \draw (-2.5,2) -- (47.5,2);

     \foreach \g in {0,5,...,20}{
      \begin{scope}[xshift=\g cm]
       \draw (-1.2,0.05) to[bend left] (1.2,0.05);
       \draw (-1.5,0.15) to[bend right] (1.5,0.15);
      \end{scope}
      }

     \foreach \g in {25,30,...,45}{
      \begin{scope}[xshift=\g cm]
       \draw (-0.3,0.3) -- (0.3,-0.3);
       \draw (-0.3,-0.3) -- (0.3,0.3);
      \end{scope}
      }

     \draw (-2.5,2) .. controls +(180:0.5cm) and +(180:0.5cm) .. (-2.5,-2);
     \draw[dashed] (-2.5,2) .. controls +(0:0.4cm) and +(0:0.4cm) .. (-2.5,-2);
     \draw (-4,0) node{$\cdots$};

     \draw (47.5,2) .. controls +(180:0.5cm) and +(180:0.5cm) .. (47.5,-2);
     \draw (47.5,2) .. controls +(0:0.4cm) and +(0:0.4cm) .. (47.5,-2);
     \draw (49,0) node{$\cdots$};

     \draw[thick, color=blue] (17.5,2) .. controls +(180:0.5cm) and +(180:0.5cm) .. (17.5,-2);
     \draw[thick, dashed, color=blue] (17.5,2) .. controls +(0:0.4cm) and +(0:0.4cm) .. (17.5,-2);

     \draw[thick, color=blue] (27.5,2) .. controls +(180:0.5cm) and +(180:0.5cm) .. (27.5,-2);
     \draw[thick, dashed, color=blue] (27.5,2) .. controls +(0:0.4cm) and +(0:0.4cm) .. (27.5,-2);

     \draw[color=blue] (22.5,-1) node{$S$};
     \end{tikzpicture}
    \end{center}
    \caption{ }
	\label{Fig:two-ends-nondisplaceable}
\end{center}
\end{figure}

By Theorem~\ref{THM:Charact-Dense-Conjugacy-Classes-simple-version}, we can suppose that the set of maximal ends has exactly two elements, which we denote by $\mathcal{M}=\{x_0,x_1\}$. Let $A=\{a_0,a_1\}$ be a set of separating curves for which $\Sigma\setminus A = \Sigma_0\sqcup S_A\sqcup \Sigma_1$ satisfies that $x_i$ is an end of $\Sigma_i$, $i=0,1$, and $S_A$ is either a 3-punctured sphere or a 2-punctured torus. This set always exists because a surface with exactly two maximal ends and no non-displaceable finite-type subsurfaces has either infinite genus or genus zero and infinitely many isolated ends.
Let $U_A$ be the open set in $\MCG(\Sigma)$ defined analogously to~(\ref{EQ:Basic-permutation-topo}). If any two elements of $U_A$ satisfy \textbf{JEP}, then $\MCG(\Sigma)$ has a conjugacy class which is dense in $U_A$.
The proof of this fact is, \emph{mutatis mutandis}, the same as the proof of Theorem 2.1 in~\cite{kechris-rosendal}.

We now show that any two $h,h'\in U_A$ satisfy \textbf{GEP}.
Without loss of generality, we can suppose that there exist two subsurfaces $N(x_0), N(x_1) \subset \Sigma$ such that $N(x_0)^\ast$ and $N(x_1)^\ast$ are neighbourhoods of $x_0$ and $x_1$ in $E(\Sigma)$, respectively, and such that the supports of $h$ and $h'$ are contained in $\Sigma\setminus\overline{N(x_0)\sqcup N(x_1)}$.
Indeed, both $h$ and $h'$ leave each connected component in the decomposition $\Sigma\setminus A = \Sigma_0\sqcup S_A\sqcup \Sigma_1$ invariant. In particular, $h = h_0 \circ h_{S_A} \circ h_1$ and $h' = h'_0 \circ h'_{S_A} \circ h'_1$, where $h_i,h_i'$ and $h_{S_A},h'_{S_A}$ are homeomorphisms of $\Sigma_i$, $i\in\{0,1\}$, and $S_A$ respectively.
Even though $\Sigma_0$ and $\Sigma_1$ do not necessarily have a unique maximal end (for example, this is the case when $\Sigma$ is homeomorphic to Jacob's ladder), we can still apply the arguments used in the proof of Theorem~\ref{THM:Charact-Dense-Conjugacy-Classes} (more precisely, where we show that $\mathbf{(2)\Rightarrow(3)}$), to all elements of $\{h_i,h_i'\}_{i=0,1}$ and the claim on the supports of $h$ and $h'$ follows.

Let $S, S' \subset \Sigma$ be finite-type subsurfaces.
Since $\Sigma$ has no non-displaceable finite-type subsurfaces, there exist $g,g'\in\Homeo^+(\Sigma)$ such that the support of $f'=g' \circ h' \circ g'^{-1}$ is disjoint from $S$ and the support of $f = g \circ h \circ g^{-1}$ is disjoint from $f'(S')$ (recall that $\supp(g\circ h \circ g^{-1})=g(\supp(h))$). The rest of the proof is as for surfaces with just one maximal end: the map $H:=f\circ f'$ satisfies (\ref{EQ:Disjoint-support-case}) and therefore the pairs $(h,S)$ and $(h,S')$ satisfy \textbf{GEP}, see Lemma~\ref{lemma:GEP}.
\qed

\subsection{Closed subgroups of \texorpdfstring{$\pmap{\Sigma}$}{PMap(Sigma)}}
    \label{SSEC:Closed-Subgroups-MAP}

In this subsection, we prove Theorem~\ref{THM:CC_for_PMap}. Recall that $\pmap{\Sigma}$ denotes the subgroup of $\MCG(\Sigma)$ which fixes all elements of $E(\Sigma)$ and $\pmapc{\Sigma}$ denotes the subgroup generated by compactly supported mapping classes.

\begin{proof}[Proof of Theorem~\ref{THM:CC_for_PMap}]
 Suppose that $\Sigma$ has at least three different ends $x, y, z$. Then there exist separating curves $\alpha$ and $\beta$ such that $\alpha$ separates $x$ and $y$ from $z$, and $\beta$ separates $y$ and $z$ from $x$.

 If $f \in \Gamma \leq \pmap{\Sigma}$, we have that $f$ acts trivially on the space of ends of~$\Sigma$. In particular, $f(\alpha)$ also separates $x$ and $y$ from $z$. Analogously, if $g \in \Gamma$, we have that~$g(\beta)$ separates $y$ and $z$ from $x$.
 In the case that $f(\alpha)$ and $g(\beta))$ intersect, it follows with Lemma \ref{DehnTwistTrick} (the Dehn twist trick) that $\Gamma$ does not have a dense conjugacy class.
 Now we assume that $f(\alpha)$ and $g(\beta)$ are disjoint. Then we can find a curve $\gamma$ such that $f(\alpha)$, $g(\beta)$ and $\gamma$ bound a pair of pants (compare to Figure~\ref{Fig:Construction}). Then $\gamma$ separates $x$ and $z$ from $y$. Hence, the pair of pants is $\Gamma$-non-displaceable. Therefore, $\Gamma$ does not have a dense conjugacy class by Theorem~\ref{THM:non-displaceable-the-nowhere-dense}.

 Suppose now that $\Sigma$ has at most two different ends and is different from the Loch Ness Monster. Since $\Sigma$ has infinite topological type, it can only be homeomorphic to either the once-punctured Loch Ness Monster, or Jacob's Ladder. If $\Sigma$ is the once-punctured Loch Ness Monster then it has a finite-type non-displaceable subsurface and Corollary \ref{Coro:MapNotJEP} tells us that $\Gamma$ does not have a dense conjugacy class.

 If $\Sigma$ is Jacob's Ladder, by Corollary 6 in~\cite{Aramayona-Patel-Vlamis-2020}, we have that $\pmap{\Sigma} \cong \overline{\pmapc{\Sigma}} \rtimes \langle h \rangle$, where $h$ is the homeomorphism described in Figure~\ref{Fig:handle-shift} (\emph{a.k.a.}\ handle-shift). With this in mind, we aim to derive a contradiction and suppose that $\Gamma$ has a dense conjugacy class. Thus by Lemma \ref{lemma:GEP}, we have that $\Gamma$ satisfies \textbf{GEP}.

 \begin{figure}[!ht]
 \begin{center}
  \begin{tikzpicture}[scale=0.5]
  \draw (-2.5,-2) -- (27.5,-2);
  \draw (-2.5,2) -- (27.5,2);

  \foreach \g in {0,5,...,25}{
   \begin{scope}[xshift=\g cm]
    \draw (-1.2,0.05) to[bend left] (1.2,0.05);
    \draw (-1.5,0.15) to[bend right] (1.5,0.15);
   \end{scope}
   }

  \draw[->] (10,2.7) -- node[above]{$h$} (15,2.7);
  \end{tikzpicture}
 \end{center}
 \caption{A handle-shift.}
 \label{Fig:handle-shift}
 \end{figure}
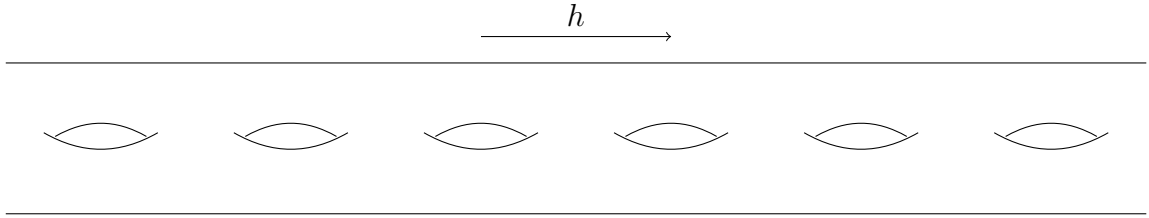

Now we divide the proof into two cases, depending whether $\Gamma$ contains a non-trivial power of $h$ or not.
 
 \textbf{Case 1}. $\Gamma \leq \overline{\pmapc{\Sigma}}$: Let $M_{1}$ be the multicurve depicted in Figure~\ref{Multicurve-M1}. Then, by Lemma \ref{DehnTwistTrick}, there exists $f \in \Gamma$ such that $f(\beta_{1})$ is disjoint and different from every element in $M_{1}$. This implies that $f(\beta_{1})$ and $\beta_{1}$ together bound a subsurface of positive genus. This implies that $f$ is not in $\pmapc{\Sigma}$ (see the proof of Proposition 6.3 in~\cite{Patel-Vlamis2018}). Moreover, since $\beta_1$ is a separating curve that separates the ends of $\Sigma$, we also have that $f = g \circ h^{n}$ with $g \in \overline{\pmapc{\Sigma}}$, $h$ the handle-shift defined above and $n \neq 0$ (again, see the proof of Proposition 6.3 in~\cite{Patel-Vlamis2018}). But this would imply that $h \in \overline{\pmapc{\Sigma}}$, in other words, that the handle-shift can be approximated by elements in $\pmapc{\Sigma}$. This contradicts Proposition~6.3 in~\cite{Patel-Vlamis2018}.

\begin{figure}[!ht]
    \begin{center}
     \begin{tikzpicture}[scale=0.5]
     \draw (-2.5,-2) -- (27.5,-2);
     \draw (-2.5,2) -- (27.5,2);

     \foreach \g in {0,5,...,25}{
      \begin{scope}[xshift=\g cm]
       \draw (-1.2,0.05) to[bend left] (1.2,0.05);
       \draw (-1.5,0.15) to[bend right] (1.5,0.15);
      \end{scope}
      }

     \draw[thick, color=blue] (2.1,2) .. controls +(180:0.5cm) and +(180:0.5cm) .. node[pos=0.75, left]{$\gamma_1$} (2.1,-2);
     \draw[thick, dashed, color=blue] (2.1,2) .. controls +(0:0.4cm) and +(0:0.4cm) .. (2.1,-2);

     \draw[thick, color=blue] (3,-2) .. controls +(180:0.2cm) and +(180:2cm) .. (4.2,1.5) -- (20.8,1.5) .. controls +(0:2cm) and +(0:0.2cm) .. (22,-2);
     \draw[thick, dashed, color=blue] (3,-2) .. controls +(0:0.5cm) and +(180:2.5cm) .. node[pos=0.25, right]{$\gamma_2$} (4.7,1) -- (20.3,1) .. controls +(0:2.5cm) and +(180:0.5cm) .. (22,-2);

     \draw[thick, color=blue] (22.9,2) .. controls +(180:0.5cm) and +(180:0.5cm) .. (22.9,-2);
     \draw[thick, dashed, color=blue] (22.9,2) .. controls +(0:0.4cm) and +(0:0.4cm) .. node[pos=0.75, right]{$\gamma_3$} (22.9,-2);

     \draw[thick, color=purple] (12.5,2) .. controls +(180:0.5cm) and +(180:0.5cm) .. node[pos=0.75, left]{$\beta_1$} (12.5,-2);
     \draw[thick, dashed, color=purple] (12.5,2) .. controls +(0:0.4cm) and +(0:0.4cm) .. (12.5,-2);
     \end{tikzpicture}
    \end{center}

    \caption{The multicurve $M_1= \{\gamma_1, \gamma_2, \gamma_3\}$.}
    \label{Multicurve-M1}
\end{figure}
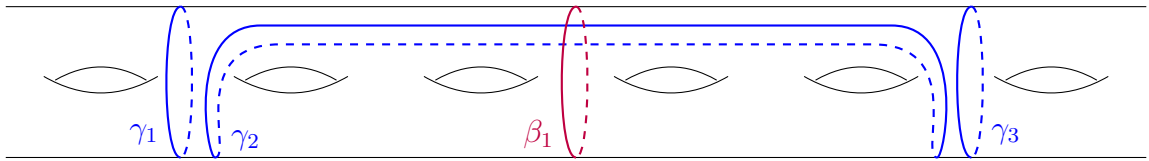

\textbf{Case 2}. $\Gamma$ is not a subgroup of $\overline{\pmapc{\Sigma}}$: Then there exists an $f_0\in \Gamma$ of the form $f_0=f\circ h^n\in\Gamma$, where $n\neq 0$ and $f\in \overline{\pmapc{\Sigma}}$. Let $\beta$, $S$ and $S'$ be as illustrated in Figure~\ref{Fig:Fig3-S}, and $h' = \tau_{\beta}^{N_{\beta}}$ (see Definition~\ref{Def:minpowerandmultitwist}). In particular, we can suppose that ${\supp(h')}\subset S$.
Since~$\Gamma$ satisfies \textbf{GEP}, there exist $g,g', H \in \Gamma$ such that $g \circ f_0 \circ g^{-1} \restriction g(S) = H \restriction g(S)$ and $\tau_{g'(\beta)}^{N_{\beta}}\restriction g'(S') = g' \circ h' \circ (g')^{-1} \restriction g'(S') = H \restriction g'(S')$. Then we have $f' \in \overline{\pmapc{\Sigma}}$ and $k \in \mathbb{Z}$ such that $H = f' \circ h^{k}$.
On one hand, we have that $\tau_{g'(\beta)}^{N_{\beta}}\restriction g'(S') = H \restriction g'(S')$ and hence $H(g'(S')) = g'(S')$.
Given that ${\supp(h')}\subset S$, $f'$ is not a handle-shift and $g'(S')$ separates the two ends of $\Sigma$, we have that $k = 0$. 
On the other hand, $g \circ f_0 \circ g^{-1}=g \circ f\circ h^n\circ g^{-1}\in\Gamma$, which implies that $k \neq 0$. Therefore, we reach a contradiction.

\begin{figure}[!ht]
\begin{center}
 \begin{tikzpicture}[scale=0.5]
  \path[pattern color=blue!10, pattern=north east lines]
  (2.1,2) .. controls +(180:0.5cm) and +(180:0.5cm) .. (2.1,-2)
  -- (22.9,-2) .. controls +(0:0.4cm) and +(0:0.4cm) .. (22.9,2);
  \path[pattern color=red!10, pattern=north west lines]
  (11.7,2) .. controls +(180:0.3cm) and +(180:0.3cm) .. (11.7,-2)
  -- (13.3,-2) .. controls +(0:0.2cm) and +(0:0.2cm) .. (13.3,2);

 \foreach \g in {0,5,10,15}{
 \begin{scope}[xshift=\g cm]
 \fill[color=white] (3.8,0.05) to[bend left] (6.2,0.05) -- (6.5,0.15) to[bend left] (3.5,0.15) -- (3.8,0.05);
 \end{scope}
 }

 \draw (-2.5,-2) -- (27.5,-2);
 \draw (-2.5,2) -- (27.5,2);

 \foreach \g in {0,5,...,25}{
  \begin{scope}[xshift=\g cm]
   \draw (-1.2,0.05) to[bend left] (1.2,0.05);
   \draw (-1.5,0.15) to[bend right] (1.5,0.15);
  \end{scope}
  }

 \draw[thick, color=blue] (2.1,2) .. controls +(180:0.5cm) and +(180:0.5cm) .. node[pos=0.75, left]{} (2.1,-2);
 \draw[thick, dashed, color=blue] (2.1,2) .. controls +(0:0.4cm) and +(0:0.4cm) .. (2.1,-2);

 \draw[thick, color=blue] (22.9,2) .. controls +(180:0.5cm) and +(180:0.5cm) .. (22.9,-2);
 \draw[thick, dashed, color=blue] (22.9,2) .. controls +(0:0.4cm) and +(0:0.4cm) .. node[pos=0.75, right]{} (22.9,-2);

 \draw[thick, color=purple] (12.5,2) .. controls +(180:0.3cm) and +(180:0.3cm) .. node[pos=0.75, right]{$\beta$} (12.5,-2);
 \draw[thick, dashed, color=purple] (12.5,2) .. controls +(0:0.2cm) and +(0:0.2cm) .. (12.5,-2);

 \foreach \g in {-0.8,0.8}{
 \begin{scope}[xshift=\g cm]
  \draw[thick, color=red] (12.5,2) .. controls +(180:0.3cm) and +(180:0.3cm) .. (12.5,-2);
  \draw[thick, dashed, color=red] (12.5,2) .. controls +(0:0.2cm) and +(0:0.2cm) .. (12.5,-2);
 \end{scope}
 }
 \draw[color=red] (12,1.5) node{$S'$};

 \draw[color=blue] (7.5,-1) node{$S$};
 \end{tikzpicture}
\end{center}

  \caption{}
 \label{Fig:Fig3-S}
\end{figure}

\end{proof}

%% file: questions.tex
\section{Miscellaneous results and open questions}
\label{SEC:questions}

In this section, we first discuss two properties of permutation groups which are necessary for ample generics: oligomorphy and Roelcke-precompactness. As we see, these fail for all big mapping class groups. Finally, we present a list of open questions.

\subsection{Oligomorphy and Roelcke-precompactness}
Recall that, when a group~$G$ acts on a set $X$, the action can be extended to the Cartesian product $X^n$ for any $n\in\Z_{>0}$ as $g\cdot(x_1,\ldots,x_n)=(gx_1,\ldots,gx_n)$. In this sense, the group $G$ is called \emph{oligomorphic} if the space of $G$-orbits in $X^n$ is finite for every $n\in\Z_{>0}$.

\begin{lemma}
  \label{lemma:no_olygomorphic}
Let $\Sigma$ be an infinite-type surface. Then the action of $\MCG(\Sigma)$ on the curve graph $C(\Sigma)$ is not oligomorphic.
\end{lemma}

\begin{proof}
For any surface $\Sigma$ of infinite type, there exist infinitely many different positive integers~$n$ and pairs $(\alpha,\beta_n)$ of essential simple closed curves such that the intersection number fulfills $i(\alpha,\beta_n) \geq n$. Indeed, it is sufficient to consider two non-isotopic essential simple closed curves $\alpha$ and $\beta$ whose geometric intersection is different from zero and define $\beta_n=\tau^n_\alpha(\beta)$, where $\tau_\alpha$ is the Dehn twist along $\alpha$. Homeomorphisms of $\Sigma$ preserve intersection number, thus  $(\alpha,\beta_n)$ are all in different $\MCG(\Sigma)$-orbits when considering the action on $C(\Sigma)\times C(\Sigma)$.
\end{proof}

A topological group $G$ is \emph{Roelcke precompact} if, for every neighbourhood~$V$ of the identity, there is a finite set $F\subseteq G$ such that $G=VFV$. This turns out to be equivalent to the following property: for all $n \in \Z_{>0}$ and all neighbourhoods $V$ of the identity in $G$, there is a finite set $F$ such that
$$
G^n=G \times  \ldots \times  G = V \cdot (FV \times \ldots \times FV).
$$

\begin{lemma}
  \label{lemma:no_roelcke_precompact}
Let $\Sigma$ be an infinite-type surface. Then $\MCG(\Sigma)$ is not Roelcke precompact.
\end{lemma}

\begin{proof}
For simplicity, let us write $G=\MCG(\Sigma)$. Let $\beta$ be any essential simple closed curve and $V = U_\beta$ the stabilizer of the isotopy class of $\beta$ in $\MCG(\Sigma)$. Suppose that we can find a finite set $F$ such that $G \times G=V \cdot (FV\times FV)$. Then:
$$
\{ (f_i\beta, f_j\beta)  :   f_i,f_j\in F \}
$$
enumerates representatives for all orbits of $G$ in $(G\cdot \beta)\times (G\cdot \beta)$ (considering the diagonal action). This leads to a contradiction. Indeed, as seen in the proof of Lemma~\ref{lemma:no_olygomorphic}, for any surface $\Sigma$ of infinite type, there exist infinitely many pairs $(\alpha,\beta_n)$ of essential simple closed curves such that the intersection number fulfills $i(\alpha,\beta_n) \geq n$, and moreover $\beta_n$ can be chosen to be in the $G$-orbit of $\beta$. This implies that the set of $G$-orbits in $(G\cdot \beta)\times (G\cdot \beta)$ is infinite.
\end{proof}

\subsection{Questions}

Several of our main results (\emph{e.g.}\ Theorems~\ref{THM:All_CC_are_meager} and~\ref{THM:CC_for_PMap}) are valid for closed subgroups of $\MCG(\Sigma)$ having a lot of Dehn twists.

\begin{question}
Let $\{\alpha_i\}_{i\in I}$ be an infinite set of pairwise non-isotopic simple closed curves in $\Sigma$ such that $\Sigma\setminus\cup_{i\in I}\alpha_i$ is a disjoint union of discs and punctured discs (in other words $\{\alpha_i\}_{i\in I}$ \emph{fills} $\Sigma$). Let $\Sigma$ be a closed subgroup of $\MCG(\Sigma)$ which does not contain any (non-trivial) power of a Dehn twist w.r.t.\ a curve in $\{\alpha_i\}_{i\in I}$. Are all conjugacy classes of $\Gamma$ meager?
\end{question}

As discused in Remark~\ref{Remark:countable-one-max-end}, the set of genus zero surfaces with no non-displaceable finite-type subsurfaces and whose space of ends is countable and has a unique maximal end are in 1-1 correspondence with ordinals of the form $\omega^\alpha+1$, where $\alpha$ is a countable ordinal.

\begin{question}
Is it possible to list the elements of the set formed by all genus zero surfaces whose space of ends is uncountable and has a unique maximal end?
\end{question}

For $n\in\mathbb N$, let $\mathcal F_p^n$ denote the class of all $ S = \langle A,\psi_1 : B_1 \to C_1 , \dots , \psi_n : B_n \to C_n\rangle $, where $A,B_i,C_i \in\mathcal F$, $B_i, C_i \subseteq A$  and $\psi_i$ is an isomorphism between $B_i$ and $C_i$. Furthermore, we say that $ S = \langle A,\psi_i : B_i \to C_i, i=1\dots, n\rangle $ \emph{embeds into} $ S' = \langle A',\psi_i' : B_i' \to C_i', i=1\dots, n\rangle $ if there is an embedding $f:A\to A'$ such that $f$ embeds $B_i$ into~$B_i'$ and $C_i$ into~$C_i'$ and $f\circ\psi_i\subseteq \psi'_i\circ f$ for all $i=1\dots, n$. The following proposition summarizes the importance of the class $\mathcal F_p^n$.

\begin{proposition}~\cite{kechris-rosendal}
    Let $\mathcal F$ be a Fra\"{\i}ss\'e class with Fra\"{i}ss\'e limit $\mathbf{K}$. Then it holds:
\begin{enumerate}
\item $\Aut(\mathbf K)^n$ has a dense conjugacy class iff $\mathcal F_p^n$ satisfies {\bf JEP},
and
\item $\Aut(\mathbf K)$ has \emph{ample generics}, \emph{i.e.}\ $\Aut(\mathbf K)^n$ has a co-meager conjugacy class for every $n\in\mathbb N$, iff $\mathcal F_p^n$ satisfies
{\bf JEP} and {\bf WAP}.
\end{enumerate}
\end{proposition}

\begin{question}
For which infinite-type surfaces $\Sigma$ and $n\in\Z_{>0}$ does $\MCG(\Sigma)^n$ have a dense (or a somewhere dense) conjugacy class?
\end{question}